\documentclass[12pt]{report}
 
 \usepackage[backend=biber, style=alphabetic, date=year, url=false, doi=false, isbn=false, eprint=true, firstinits=true, maxcitenames=5, maxbibnames=5]{biblatex}
\AtEveryBibitem{%
  \ifentrytype{online}{%
  }{%
    \clearfield{eprint}%
    \clearfield{urldate}%
  }%
}
\usepackage{amsmath}
\usepackage{amssymb}
\usepackage{amsthm}
\usepackage{paralist}
\usepackage{fullpage}
\usepackage{hyperref}
\usepackage{tikz-cd}
\usepackage{tikz}
\usepackage{todonotes}
\theoremstyle{plain} 
\newtheorem{thm}{Theorem}[section]
\newtheorem{lem}[thm]{Lemma}
\newtheorem{cor}[thm]{Corollary}
\newtheorem{prop}[thm]{Proposition}
\newtheorem{clm}[thm]{Claim}

\theoremstyle{definition}
\newtheorem{defn}[thm]{Definition}
\newtheorem{expl}[thm]{Example}

\theoremstyle{remark}
\newtheorem{rem}[thm]{Remark}

\newtheorem{exc}{Exercise}
\newcommand{\ii}{\mathrm{i}}
\newcommand{\N}{\mathbb{N}}
\newcommand{\C}{\mathbb{C}}
\newcommand{\Z}{\mathbb{Z}}
\newcommand{\R}{\mathbb{R}}
\newcommand{\HH}{\mathbb{H}}
\newcommand{\g}{\mathfrak{g}}
\newcommand{\de}{\partial}
\newcommand{\ind}{\mathrm{ind}}
\newcommand{\tind}{\mathrm{top-ind}}
\newcommand{\Hom}{\mathrm{Hom}}
\newcommand{\End}{\mathrm{End}}
\newcommand{\str}{\operatorname{str}}
\newcommand{\tr}{\operatorname{tr}}
\newcommand{\Cl}{\mathrm{Cl}}
\newcommand{\Gr}{\mathrm{Gr}}
\newcommand{\Pin}{\mathrm{Pin}}
\newcommand{\Spin}{\mathrm{Spin}}
\newcommand{\Pf}{\mathrm{Pf}}

\newcommand\restr[2]{{
  \left.\kern-\nulldelimiterspace 
  #1 
  \vphantom{\big|} 
  \right|_{#2} 
  }}

\title{Lecture Notes on Spin Geometry}
\author{Konstantin Wernli}
\date{\today}
\begin{document}
\maketitle
\tableofcontents
\chapter*{Introduction}
\addcontentsline{toc}{chapter}{Introduction}
The subject of Spin Geometry has its roots in physics and the study of spinors. However, once adapted to a mathematical framework, it beautifully intertwines the realms of algebra, geometry and analysis. When combined with the Atiyah-Singer index theorem - one of the most remarkable results in twentieth century mathematics - it has far-reaching applications to geometry and topology. \\
This course has three main goals. The first goal is to understand the concept of Dirac operators. The second is to state, and prove, the Atiyah-Singer index theorem for Dirac operators. The last goal is to apply these concepts to topology: A remarkable number of topological results - including the Chern-Gauss-Bonnet theorem, the signature theorem and the Hirzebruch-Riemann-Roch theorem - can be understood just by computing the index of a Dirac operator. Before starting a serious discussion, let us try to illuminate the central concepts in a leisurely way.

\section*{What are Dirac operators?}
\addcontentsline{toc}{section}{Dirac Operators}

Dirac operators were introduced first by Paul Dirac (of course) in an attempt to understand relativistic quantum mechanics. A simplified version of the problem goes as follows. Consider a particle of mass $m$ moving in $\R^3$ with momentum $p = (p_1,p_2,p_3)$. Relativity tells us that the particle has energy\footnote{In units where the speed of light $c=1$.} \begin{equation}
E = \sqrt{m^2 + p^2}
\end{equation}
where $p^2 = p_1^2 + p_2^2 + p_3^2$. Passing to quantum mechanics\footnote{For us this is an entirely formal operation, and we will not discuss it any further.}, the particle is described by a time-dependent ``wave function'' $\psi_t \in L^2(\R^3)$ and we replace $E \to \ii \frac{\partial}{\partial t}, p_j \to -\ii\frac{\partial}{\partial x_j}$, which formally leads to the equation 
\begin{equation}
\ii \frac{\partial}{\partial t}\psi_t = \sqrt{m^2 + \Delta} \psi_t\label{eq:DiracEquation}
\end{equation}
where we define the Laplacian $\Delta = -\sum_{j=1}^3\frac{\partial^2}{\partial x^2_j}$. To understand this equation, Dirac looked for a first order differential operator with constant coefficients 
\begin{equation}
D = \sum_{j=1}^3 \gamma_j\frac{\partial}{\partial x_j} + \gamma_0
\end{equation}
such that
 \begin{equation} D^2 = m^2 + \Delta.\label{eq:Diracsquare}\end{equation}
  Given such an operator $D$, we can solve equation \eqref{eq:DiracEquation} as  $\psi_t = \exp(-itD)\psi_0$. However, it turns out that \eqref{eq:Diracsquare} cannot be solved if the coefficients $\gamma_i$ are real or complex numbers. If we accept for the moment that the $\gamma_i$ do not commute with each other, equation \eqref{eq:Diracsquare} leads to the set of equations 
\begin{equation}
\begin{cases} 
\gamma_i\gamma_j + \gamma_j\gamma_i = 0 & i \neq j \\
\gamma_0^2 = 1 &  \\
\gamma_j^2 = -1 & j=1,2,3 \\
\end{cases}
\end{equation}
which, introducing the anti-commutator $\{a,b\} = ab + ba$ and the standard Minkowski metric on $\R^4$, $\eta_{ij} = \mathrm{diag}(-1, +1,+1,+1)$, can be summarized as 
\begin{equation}
\{\gamma_i,\gamma_j\} = -2\eta_{ij}. \label{eq:CliffordRelation}
\end{equation}We call equation \eqref{eq:CliffordRelation} the \emph{Clifford Relation}. It leads to very interesting algebra that we will study in Chapter \ref{ch:Algebra}. Another implication is that the wave function $\Psi$ is not real valued, rather, it is a section of a certain ``spinor bundle''. This will be discussed in \ref{ch:Geometry}. In chapter \ref{ch:Analysis}, after discussing some analytic preliminaries, we will finally introduce Dirac Operators and their index. 
\section*{What is the index (theorem)?}
\addcontentsline{toc}{section}{Index Theorem}

The index, and the index theorem, are statements about the solutions to equations. Let us illuminate this by means of two simple examples\footnote{It may be worth to note that the operators appearing here are \emph{not} Dirac operators, but the index theorem still holds.}. \\
First, let $V,W$ be vector spaces and consider a linear map $A \colon V \to W$. We are interested in the equation $Av=0$. The number of independent solutions is $\dim \ker A$. A priori, we do not have any information about this number on its own. However, what we can consider is the rank-nullity theorem 
\begin{equation}
\mathrm{rk}(A) + \dim \ker A = \dim V. 
\end{equation}
This equation can be rewritten as 
\begin{equation}
\dim \ker A - (\dim W - \mathrm{rk}(A)) = \dim V - \dim W
\end{equation}
or, introducing the cokernel of A as $\mathrm{coker}(A) = W/\mathrm{im}A$, as 
\begin{equation}
\dim \ker A - \dim \mathrm{coker} A= \dim V - \dim W.\label{eq:index1}
\end{equation}
In other words, while we cannot say anything about $\dim \ker A$, we know exactly what $\dim \ker A - \dim \mathrm{coker} A$ (in fact it is independent of $A$). If we define the index \begin{equation}\mathrm{ind}(A):= \dim \ker A - \dim \mathrm{coker} A
\end{equation} then we can interpret equation \eqref{eq:index1} is a first incarnation of the index theorem. \\
As a second example, consider the family of differential operators \begin{align*} D_\lambda &\colon C^\infty(S^1) \to C^\infty(S^1) \\ D_\lambda f &= \frac{df}{dx} -2\pi i \lambda f
\end{align*} 
Then, the kernel of $D$ is nonzero if and only if $\lambda \in \Z$, and in this case is spanned by $f(x) = e^{2\pi\ii \lambda x}$. Hence, the dimension of the kernel is given by 
\begin{equation}
\dim \ker D_{\lambda} = \begin{cases} 0 & \lambda \notin \Z \\
1 & \lambda \in \Z 
\end{cases}
\end{equation} 
in particular it does not vary continuously in $\lambda$. On the other hand, let us consider the index of $D_\lambda$. By a general fact\footnote{We will deal with these technicalities in more detail later.}, we have $\mathrm{coker} D_\lambda = \ker D_\lambda^*$, where $D_\lambda^* = \frac{d}{dx}+2\pi\ii\lambda$ is the adjoint of $D_\lambda$. Hence, we see that 
\begin{equation}
\ind(D_\lambda) = \dim \ker D_\lambda - \dim \ker D^*_\lambda \equiv 0 
\end{equation}
which does vary continuously in $\lambda$.  The observation that the index of an elliptic operator is invariant over continuous families led to the discovery of the \emph{topological index} - a number associated to an elliptic operator $D$, denoted $\tind(D)$ that can be computed from topological data - and the celebrated Atiyah-Singer index theorem: 
\begin{thm}
Let $E,F$ be vector bundles over a manifold $M$ and $D\colon \Gamma(E) \to \Gamma(F)$ an elliptic differential operator. Then the index of $D$ equals the topological index of $D$:
\begin{equation}
\ind(D) = \tind(D).
\end{equation}
\end{thm}
The central goal of this course are to understand the statement and proof of this theorem for Dirac Operators and study some of its applications. 
This requires that apart from understanding Dirac Operators and  their index, we define their topological index. This is done via the theory of characteristic classes, that will be introduced in Chapter \ref{ch:Geometry}. The precise analytical definitions will be given in \ref{ch:Analysis}. Finally, in Chapter \ref{ch:Index}, we will be able to state and prove the index theorem and consider some applications. 
\section*{Course outline}
\addcontentsline{toc}{section}{Course Outline}

Before we begin, let us briefly discuss the outline of the course. As mentioned in the beginning, the subject of Spin Geometry draws from Algebra, Geometry, and Analysis, and the first three chapters will each be devoted to studying the necessary preliminaries in each of these areas. The last and main chapter is reservers for the index theorem, its proof, and applications. \par
Chapter \ref{ch:Algebra} deals with Algebra. After a brief introduction to superalgebra, we will discuss the essential notions of Clifford Algebras, Spin and Pin Groups, and their representations. \par
Chapter \ref{ch:Geometry} deals with Geometry, and consists of three main sections. First, we will recall preliminaries, such as vector bundles, principal bundles and connections. Then we will be concerned with the Chern-Weil theory of characteristic classes. Lastly, we will take the algebra of Chapter \ref{ch:Algebra} and promote it to spin structures and spinor bundles over manifolds. \par
In Chapter \ref{ch:Analysis} we discus Analysis. After introducing some technical machinery on differential operators, we give a precise definitions of the index. We then proceed to define Dirac operators and Dirac bundles.\par
Finally, in Chapter \ref{ch:Index}, we give the precise statement of the index theorem for Dirac operators. We will a corollary - the Chern-Gauss-Bonnet Theorem - and sketch the ideas of the proof. 
\section*{Further reading}
There are a number of notable omissions in these lecture notes - e.g. several other applications of the index theorem, $\mathrm{Pin}$ structures and associated Dirac operators, and the generalization to elliptic operators in terms of $K$-theory. All these very interesting topics can be found in various sources, of which we only list a few here. 
The subject was founded in the works of Atiyah and collaborators: The first proof of the Index theorem was in \cite{Atiyah1963}, \cite{Atiyah1964} introduces the theory of Clifford Modules, the $K$-theoretic proof appeared in \cite{Atiyah1968} (the first in a series of several papers dealing with the theorem). A proof that uses heat kernels instead of $K$-theory was given by Atiyah, Bott and Patodi in \cite{Atiyah1973} and later simplified by Getzler in \cite{Getzler1983,Getzler1986}. This is the proof that we will follow in these notes.  \\
As for review material, the classic textbook on the subject is \cite{Lawson1990}. The topic is also thoroughly reviewed in \cite{Salamon1996} (the main focus of the book lying on Seiberg-Witten invariants). In these notes we follow closely the lecture notes of L. Nicolaesecu \cite{Nicolaescu2013} and X. Dai \cite{Dai2015}.\\
Finally, for the sake of completeness let us admit that we do not discuss at all the extension  to manifolds with boundary, the Atiyah-Patodi-Singer theorem (\cite{Atiyah1975}, for a  complete review see \cite{Melrose1993}) which has recently found great attention in the physics of topological insulators (see e.g. \cite{Fukaya2018} and references there).
 \chapter{Algebra}\label{ch:Algebra}
In this chapter we will mainly be concerned with Clifford algebras, roughly speaking, these are algebras where a version  of the Clifford relation $$\{\gamma_i,\gamma_j\} = -2\eta_{ij}$$ holds. It is useful to employ the terminology of super algebra, which we will set up in the first section. 
\section{Superalgebra}
In this section, the ground field $k$ is either $\R$ or $\C$. 
The basic notion is that of a \emph{super vector space}.
\begin{defn}[Super vector space] 
\hfill
\begin{enumerate}[i)]
\item
A \emph{super vector space} is a vector space $V$ together with a decomposition $V = V_0 \oplus V_1$. 
\item If $V = V_0 \oplus V_1$ is a super vector space and $v\in V_i$, we say that $v$ is \emph{homogeneous of degree $i$} and use the notation $|v| = i$. 
\item We call $V_0$ the \emph{even part} of $V$, and elements $v \in v_0$ are called \emph{even}. Similarly, we call $V_1$ the \emph{odd part} of $V$ and call elements $v \in V_1$ odd. 
\end{enumerate}

\end{defn}

There are two equivalent terminologies: We can talk about the \emph{degree} of an object, which can be either 0 or 1, or the \emph{parity} of an object, which can be even or odd. Even objects have degree 0, and odd objects have degree 1. Indices indicating the grading are always understood mod 2. 
\begin{defn}
Let $V, W$ be super vector spaces. Then we say that a linear map $A\colon V \to W$ has degree $j$ if $A(V_i) \subset W_{i+j}$. 
\end{defn}
This endows the space of linear maps $\mathrm{Hom}(V,W)$ with the structure of a super vector space. 
\begin{defn}[Supertrace]
Let $V$ be a super vector space. 
\begin{enumerate}[i)]
\item Let $E \in \End(V)_0$. Then the \emph{supertrace} of $E$ is defined to be 
\begin{equation}
\str E := \tr \restr{E}{V_0} - \tr \restr{E}{V_1}.
\end{equation}
\item Let $T \in \End(V)$ and decompose it into even and odd part $T = T_0 + T_1$. Then we define 
\begin{equation}
\str T := \str T_0 
\end{equation}
\item The grading operator $\gamma \in \End(V)$ is defined by 
\begin{equation}
\gamma_V = \mathrm{id}_{V_0} - \mathrm{id}_{V_1}.
\end{equation}
\end{enumerate}
\end{defn}
Obviously, the grading operator satisfies $\gamma_V^2 = 1$.
For a general $T \in \End(V)$, we have $\str T = \tr (\gamma T)$.

A concept of central importance for us is that of a \emph{superalgebra}. 
\begin{defn}[Superalgebra] Let $(A,\cdot)$ be an algebra where $A$ is a super vector space $A = A_0 \oplus A_1$. Then we say that $A$ is a \emph{superalgebra} if $A_iA_j \subset A_{i+j}$.
\end{defn}
If $A,B$ are superalgebras, then we say that a map of algebras$ F\colon A \to B$ is \emph{even} (resp. \emph{odd}) if the underlying linear map is even (resp. odd). Even maps are morphisms of superalgebras, while the superalgebra of all algebra maps is usually called the \emph{inner hom} (in the category of superalgebras). \\ 
The following is an easy but important exercise. 
\begin{exc}\label{exc:gradingop}
\begin{enumerate}[i)]
\item Let $V$ be a vectorspace with an operator $\gamma$ that satisfies $\gamma^2 =1$. Prove that $V$ has a super vector space structure for which $\gamma$ is the grading operator.
\item Now suppose that $V$ additionally has an algebra structure such that $\gamma$ is an algebra homomorphism. Prove that $V$ is a superalgebra.
\end{enumerate}
\end{exc}
The following are natural examples of superalgebras. 
\begin{expl}
\begin{enumerate}[i)]
\item Let $V$ be a vector space. Then the tensor algebra $T(V)$, the symmetric algebra $SV$, and the exterior algebra $\bigwedge V$ are all superalgebras. Here the grading operator is given by the extension of the linear map $\varepsilon \colon V \to V, \varepsilon(v) = -v$. 
\item Now, let $V$ be a super vector space. Then $\End(V)$ is a superalgebra (with multiplication the composition of linear maps). 
\end{enumerate}
\end{expl}
The reason to introduce superalgebras is to have a convenient way to deal with signs. The first appearance of signs is in the definition below.
\begin{defn}[Supercommutative]
\begin{enumerate}[i)]
\item
Let $A$ be superalgebra. Then the \emph{supercommutator} is the bilinear map $[\cdot,\cdot]_s \colon A \times A \to A$ defined on homogeneous elements by $[a,b]_s :=ab - (-1)^{|a||b|}ba $.
\item A superalgebra $A$ is \emph{supercommutative} if the supercommutator vanishes identically, equivalently, for all homogeneous $a,b \in A$ we have 
\begin{equation}
ab = (-1)^{|a||b|}ba. 
\end{equation}
\end{enumerate}
\end{defn}

\begin{expl}
If $V$ is a vector space, then the superalgebra $\bigwedge V$ is supercommutative, but the superalgebra $SV$ is not\footnote{Notice however that the algebra $SV$ is commutative, while the algebra $\bigwedge V$ is not.}. 
\end{expl}
Another important notion is that of the tensor product of superalgebras. 
\begin{defn}[Tensor product] \begin{enumerate}[i)]
\item
Let $V,W$ be super vector spaces. Then the tensor product is the super vector space given by $V\hat{\otimes} W = (V\hat{\otimes} W)_0 \oplus (V\hat{\otimes} W)_1$, where 
\begin{align}
(V \hat{\otimes} W)_0 &= V_0 \otimes W_0 \oplus V_1 \otimes W_1 \\
(V \hat{\otimes} W)_1 &= V_1 \otimes W_0 \oplus V_0 \otimes W_1. 
\end{align}
\item Let $A,B$ be superalgebras. Then their tensor product is the superalgebra with underlying super vector space $A\hat{\otimes} B$ and multiplication defined on the tensor product of homogeneous elements $a_1,a_2 \in A$ and $b_1,b_2 \in B$ by 
\begin{equation}
(a_1 \hat{\otimes} b_1)(a_2 \hat{\otimes} b_2) := (-1)^{|a_2||b_1|}a_1a_2 \hat{\otimes} b_1b_2 
\end{equation}
\end{enumerate}
\end{defn}

The sign above is a good example of the \emph{Koszul sign rule}: Whenever we exchange two objects in the graded world (in this case $a_2$ and $b_1$) then we get a sign determined by the product of their degrees. Let us establish one last definition for future purposes: 
\begin{defn}[Supermodule]
A supermodule $V$ over a superalgebra $A$ is a super vector space $V$ together with an even map $\rho$ of superalgebras $ \rho \colon A \to \End(V)$.  
\end{defn}
\begin{defn}[Morphisms of supermodules]
Let $(V_1,\rho_1)$ and $(V_2,\rho_2)$ be supermodules over a superalgebra. Then a \emph{morphism of supermodules} is a linear map $T\colon V_1 \to V_2$ which supercommutes with the action of $A$: For homogeneous $T$ and $a\in A$ this means that 
\begin{equation}
T \rho_1(a) = (-1)^{|T||a|}\rho_2(a)T.
\end{equation}

\end{defn}
The set of morphisms of supermodules is denoted $\Hom_A(V_1,V_2)$ (here we suppress $\rho_1,\rho_2$). For $(V,\rho)$ a supermodule over $A$ we denote $\End_A(V)$ the morphisms from $(V,\rho)$ to itself. \\
After setting up the terminology of superalgebras, we are ready for the discussion of more interesting topics.

\section{Clifford algebras}
Let $V$ be a vector space over $k = \R$ or $k = \C$. Let $g \colon V \times V \to k$ be a symmetric bilinear form on $V$. We often call the pair $(V,g)$ a \emph{quadratic vector space}\footnote{The name comes from the fact that equivalently one can work with quadratic forms, but we shall stick to bilinear forms, with an eye towards the applications.}. 
We denote by $I_g$ the two-sided ideal in $T(V)$ generated by the set $\{v\otimes v + g(v,v)1, v \in V\} \subset T(V)$. 
\begin{defn}[Clifford algebra]
The \emph{Clifford algebra} $\Cl(V,g)$ is the quotient 
\begin{equation}
\Cl(V,g) := T(V)/I_g
\end{equation}
\end{defn}
We denote the quotient map by $\pi\colon T(V) \to \Cl(V,g)$.
Some remarks are immediate from this definition. 
\begin{rem}
\begin{enumerate}[i)]
\item There is an injection $\iota \colon V \hookrightarrow \Cl(V,g)$ and we will identify $V \cong \iota(V) \subset \Cl(V,g)$. 
\item $\Cl(V,g)$ is the algebra generated by the vector space $V$, subject to the relation $v \cdot v = -g(v,v)$. This relation is equivalent to 
\begin{equation}
v\cdot w + w \cdot w = -2g(v,w),
\end{equation}
the Clifford relation that we discussed before.
\item This construction is universal: Given any $k$-algebra $A$ 
and a linear map $f\colon V \to A$ such that $f(v)f(v) = -2g(v,v)1$, there exists a unique algebra map $\tilde{f} \colon Cl(V,g) \to A$ such that $\tilde{f}\circ\iota = f$, i.e. the following diagram commutes: \\
\begin{equation}
\begin{tikzcd}
{Cl(V,g)} \arrow[rd, "\tilde{f}", dotted] &   \\
V \arrow[u, "\iota", hook] \arrow[r, "f"] & A
\end{tikzcd}
\end{equation}\item As a consequence, the association $(V,g)\mapsto\Cl(V,g)$
is \emph{functorial}, i.e. for a linear map $f\colon (V,g) \to (V',g')$ such that\footnote{Here we use the \emph{pullback} defined by $f^*g(v,w) = g(f(v),f(w))$.} $f^*g' = g$, there is a unique map $\tilde{f}\colon \Cl(V,g) \to \Cl(V',g')$. In particular, if $(V,g) \cong (V',g')$, then also $\Cl(V,g) \cong \Cl(V',g')$. 
\end{enumerate}

\end{rem}
Let us consider some examples. 
\begin{expl}
 In the definition, the bilinear form $g$ is not supposed to be non-degenerate. Thus, we can take $g = 0$. By definition, the Clifford algebra then satisfies $v \cdot w + w\cdot v = 0$, hence $\Cl(V,0) \cong \bigwedge V$.  
\end{expl}
\begin{expl}
Let $k=\R$. We denote by $\R^{p,q}$ the vector space $V = \R^n$ with the standard bilinear form $g_{p,q}$ of signature\footnote{Our convention for the signature $(p,q)$ is that $p$ denotes the number of positive eigenvalues and $q$ denotes the number of negative eigenvalues so that $p+q \leq n$. } $(p,q)$ where $p+q = n$, i.e. the matrix of $g_{p,q}$ in the standard basis is $$g_{p,q}=\mathrm{diag}(\underbrace{1,\ldots,1}_p,\underbrace{-1,\ldots,-1}_q).$$ Then we denote 
$\Cl(\R^{p,q}) =: \Cl_{p,q}$. This is calles the \emph{standard Clifford algebra of signature $(p,q)$}. By the remark above, all Clifford algebras over real vector spaces with a non-degenerate symmetric bilinear form is isomorphic to one of the $Cl_{p,q}$. We also denote $\Cl_n := \Cl_{n,0}$.
\end{expl}
\begin{expl}
Consider $\Cl_1$. By definition, is algebra is $T(\R)/I_{g_1}$. Since $T(\R) \cong \R[x]$, we have $\Cl_1 \cong \R[x]/<x^2 + 1> \cong \C$ where we consider $\C$ as an $\R$-algebra. 
\end{expl}

The Clifford algebra inherits some properties from the tensor algebra and does not inherit some others. Let us discuss this in more detail. A $\Z$-graded algebra is an algebra $A$ with a direct sum decomposition $A = \bigoplus_k A_k$ such that $A_iA_j \subset A_{i+j}$. The tensor algebra $T(V)$ is, naturally, a $\Z$-graded algebra. Every $\Z$-graded algebra can be turned into a $\Z_2$-graded algebra by taking the direct sum of even and odd components. A crucial fact about the Clifford-algebra is that it \emph{does not} inherit the $\Z$-grading. This can be seen in by considering e.g. the image of $v^{\otimes 2k}\in T(V)_{2k}$, the image under the projection is 
\begin{equation}
\pi(v^{\otimes 2  k}) = (-1)^kg(v,v)^k \in \pi(T(V)_0).
\end{equation}
However, the Clifford algebra inherits the $\Z_2$-grading of the tensor algebra, and thus is naturally a superalgebra. 
\begin{prop}
Let $(V,g)$ be a vector space with a symmetric bilinear form. Denote $T(V) = T(V)_0 \oplus T(V)_1$ the superalgebra structure of $T(V)$ Then $\Cl(V,g)$ is a superalgebra with $\Cl(V,g)_0 =  \pi(T(V)_0)$, $\Cl(V,g)_1 = \pi(T(V)_1)$. 
\end{prop}
\begin{proof}
Consider again the map $\varepsilon \colon V \to V, \varepsilon(v) = -v$. Since its preserves the bilinear form it extends to an algebra isomorphism $\varepsilon\colon \Cl(V,g) \to \Cl(V,g)$ which squares to 1. Then, by Exercise \ref{exc:gradingop} we know that $\Cl(V,g)$ is a superalgebra for which $\varepsilon$ is the grading operator. Then the claim follows from the easy observation that $\varepsilon(\pi(t)) = \pi(\varepsilon(t))$, for any $t \in T(V)$. 
\end{proof}
There is another structure that $\Cl(V,g)$ inherits from $T(V)$, namely that of a \emph{filtered algebra}\footnote{There are many versions of the concept of filtration in the literature. This one is sufficient for our purposes.}. 
\begin{defn}[Filtered algebra]
A \emph{filtered algebra} $A$ over $k$ is an algebra $A$ over $k$, together with a collection of linear subspaces $\{F^iA\}_{i\in\N}$ satisfying  \begin{align*} A &= \bigcup_{i\in\N}F^iA\\
(F^iA)(F^jA)&\subset F^{i+j}A
\end{align*} 
\end{defn}
The tensor algebra (and, in fact, every $\N$-graded algebra) has a natural filtration $F^iT(V) = \oplus_{k=0}^iT(V)_k$.
While the map $\pi\colon T(V) \to \Cl(V,g)$ does not preserve the grading, it does preserve the filtration in the sense that $F^i\Cl(V,g):=\pi(F^iT(V))$ defines a filtration on $\Cl(V,g)$.   
To every filtered algebra, one can associate a graded algebra as follows. 
\begin{defn}[Associated graded]
Let $A$ be a filtered algebra. Then, the \emph{associated graded algebra} $\Gr(A)$ is the $\N$-graded algebra given by 
$\Gr(A) = \oplus_{i\in \N}\Gr(A)_i$ with 
$$\Gr(A)_i = F^iA/F^{i-1}A.$$
\end{defn}
\begin{exc}Spell out the multiplication on $\Gr(A)$, and prove that it is well-defined and makes $\Gr(A)$ into a graded algebra. 
\end{exc}
\begin{prop}
Let $V$ be a vector space with a symmetric bilinear form $g$. Consider the Clifford algebra $Cl(V,g)$ with its natural filtration. Then $\Gr(\Cl(V,g)) \cong \bigwedge V$, i.e. the associated graded of the Clifford algebra is isomorphic to the exterior algebra of $V$.
\end{prop}
Note that the associated graded of $\Cl(V,g)$ is independent of $g$.  
\begin{proof}
Consider the map $\varphi^k\colon V^{\otimes k} \to \Gr^k(Cl(V,g))$ given by the composition
$$V^{\otimes k} \to F^kCl(V,g) \to F^k\Cl(V,g)/F^{k-1}\Cl(V,g)$$ which takes $v_1\otimes \ldots \otimes v_k \mapsto [v_1\cdot\ldots\cdot v_k].$  We claim that this map descends to $\bigwedge^k V$. Indeed, the image of the ideal $I$ generated by $\{v\otimes v, v\in V\}$ under $\pi$ lies in $F^{k-1}\Cl(V,g)$ because $\pi(v\otimes v) = 2g(v,v)$. Hence $I \subset \ker\varphi^k$ and the map descends to a map $\bar{\varphi}^k \colon \bigwedge^kV \to \Gr(\Cl(V,g))_k$. This map is surjective (since $\varphi^k$ is surjective). The direct sum $\bar{\varphi}=\bigoplus_{k=1}^n\bar{\varphi}^k$ (where $n = \dim V$) is an algebra homomorphism. We will now show that the maps $\bar{\varphi}^k$ are injective for every $k$. Namely, consider a tensor $t \in \ker \varphi^k = \pi^{-1}(F^{k-1}\Cl(V,g))$. These are tensors of degree $r$ that are sum of terms of the form $a \otimes v \otimes v \otimes b$, where $a$ and $b$ are of homogeneous degree that add up to $r-2$. Then $t \in I$, i.e. $t$ vanishes in the exterior algebra, which proves that $\bar{\varphi^k}$ is injective.
\end{proof}
\begin{cor}
There is a canonical vector space isomorphism 
\begin{equation}\sigma\colon \Cl(V,g) \to \bigwedge(V)
\end{equation}
preserving the $\Z_2$-grading. 
\end{cor}
Here, by canonical we mean that the isomorphism does not depend on any choices. 
\begin{proof}
This follows from the general fact that for a filtered algebra $A$ with a finite filtration, the map $A \to \Gr(A)$ is an isomorphism of vector spaces which depends only on the filtration. Since the filtration is induced by the (canonical) filtration of the tensor algebra, the isomorphism is canonical.
\end{proof}
This theorem allows us to find an easy basis of the Clifford algebra. 
\begin{prop}
Let $(V,g)$ be a quadratic vector space of dimension $n$ and let $v_1,\ldots,v_n$ be a basis of $V$. Then the set $$\{v_{i_1}\cdots v_{i_k}| k=0,\ldots,n, 1\leq i_1 <i_2 \ldots <i_k \leq n\}$$
is a basis of the associated Clifford algebra $\Cl(v,g)$.
\end{prop}
An ordered multiindex is a multiindex $I=(i_1,\ldots,i_k)$ satisfying $i_1 < i_2 \ldots <i_k$. $k$ is called the length of the multiindex and denoted $|I|=k$ ($k=0$ is allowed). By $e_I$ we abbreviate $e_{i_1}\cdots e_{i_k}$. By definition, the empty product ($k=0$) is 1.
\begin{proof}
First note that there are $\sum_{k=0}^n {n\choose k} = 2^n$ ordered multiindices in $\{1,\ldots,n\}$. Since $\dim \Cl(V,g) = \dim \bigwedge V = 2^n$, it is enough to show that the vectors $e_I$ span $\Cl(V,g)$. We prove by induction that $\{e_I||I|\leq k\}$ spans $F^k\Cl(V,g)$. Clearly, this true for $k=0$, since $F^0\Cl(V,g) = k$. Now suppose the claim is true for all $k'<k$. Denote $W_k:= \mathrm{span}\{e_I,|I| \leq k$. We have to show that  $W_k \supset \pi(F^kT(V))$. Since we already know that $\pi(F^{k-1}T(V)) \subset W_{k-1}\subset W_k$, it is enough to show that $\pi(T^k(V)) \subset W_k$. Any element of $T^kV$ is a sum of terms of the form $v_{j_1}\otimes  \cdots \otimes v_{j_k}$, whose image in the Clifford algebra is $v_{j_1} \cdots v_{j_k}$. Using the Clifford relation, we can rewrite this as $v_{j_{\sigma(1)}}\cdots v_{j_{\sigma(k)}} + \alpha$, where $j_{\sigma(1)} < \ldots < j_{\sigma(k)}$ and $\alpha$ contains at most $k-2$ basis vectors. Hence, $alpha \in \pi(F^{k-2}T(V)) \subset W_{k-2} \subset W_k$. This proves the claim. 
\end{proof}
\section{Pin and Spin groups}
Now we fix the ground field $k=\R$. Consider the group of invertible elements $\Cl(V,g)^\times \subset \Cl(V,g)$. Since it is an open subset of $Cl(V,g)$ it is a Lie group of dimension $2^{\dim V}$. The tangent space at the neutral element $1$ is $\Cl(V,g)$. It carries a natural Lie bracket given by the commutator with respect to Clifford multiplication. The Lie group $\Cl(V,g)^\times$ acts on its Lie algebra $\Cl(V,g)$ by the adjoint representation 
\begin{align}
\rho\colon \Cl(V,g)^\times &\to GL(\Cl(V,g)) \notag \\
x &\mapsto (v \mapsto xvx^{-1}) \label{eq:def_adjoint_rep}
\end{align}
For our purposes, the \emph{twisted adjoint representation} will be more important. 
\begin{defn}[Twisted adjoint representation]
The \emph{twisted adjoint representation} is the representation  $\tilde{\rho} \colon\Cl(V,g)^\times \to GL(\Cl(V,g))$ given by $\tilde{\rho}(x)v = \varepsilon(x)vx^{-1}$ where $\varepsilon\colon \Cl(V,g) \to \Cl(V,g)$ is the grading operator. 
\end{defn}
\begin{prop}\label{prop:reflections}
Let $v \in V$ with $g(v,v) \neq 0$. Then
\begin{itemize}
\item $v \in \Cl(V,g)^\times$. 
\item $\tilde{\rho}(v)$ stabilizes $V$, i.e. $\tilde{\rho(v)}(V) \subset V$,
\item for any $w \in V$, we have 
\begin{equation}\tilde{\rho}(v)w  = w - 2\frac{g(v,w)}{g(v,v)}v
\end{equation}
i.e. $\tilde{\rho}(v)$ acts as a reflection by the hyperplane $v^\perp$.
\end{itemize}
\end{prop}
\begin{proof}
For i) simply notice that since $v^2 = - g(v,v)1$, if $g(v,v) \neq 0$ we have $v^{-1} = \frac{-v}{g(v,v)}$. Of course ii) is implied by iii), which is a simple computation: 
\begin{align*}
\tilde{\rho}(v)w &= \varepsilon(v) w v^{-1} = \frac{vwv}{g(v,v)} = \frac{v(-vw -2g(v,w)1)}{g(v,v)} \\
&= w - \frac{2g(v,w)}{g(v,v)}v.
\end{align*}
Here we have used the Clifford relations $vw = -wv -g(v,w)1$ and $-v^2 = g(v,v)$.
\end{proof}
\begin{defn}[Clifford group]
The \emph{Clifford group} $\Gamma(V,g)$ is the subgroup of $\Cl(V,g)^\times$ stabilizing $V$ in the twisted adjoint representation, i.e. 
\begin{equation}
\Gamma(V,g) := \{g \in \Cl^\times(V,g) , \tilde{\rho}(g)(V) \subset V\}.
\end{equation}
\end{defn}
Clearly, nonzero elements of $k$ act trivially in the twisted adjoint representation. Hence we have $k^\times \subset \ker \tilde{\rho}$. If $g$ is non-degenerate, also the converse is true: 
\begin{prop}\label{prop:ker}
If $g$ is non-degenerate, then $\ker \restr{\tilde{\rho}}{\Gamma(V,g)} = k^\times$.
\end{prop}
\begin{proof}
Choose a basis $v_1,\ldots,v_n$ of $v$ such that $g(v_i,v_i) \neq 0$ and $g(v_i,v_j) = 0$ for $i \neq j$. Let $\varphi \in \ker \tilde{\rho}$ be non-zero. Then $\epsilon(\varphi)v = v \varphi$ for all $v \in V$. Decomposing $\varphi = \varphi_0 + \varphi_1$,  we get \begin{align} \varphi_0v &= v\varphi_0\label{eq:proofprop1} \\ 
\varphi_1v &= -v\varphi_1 \label{eq:proofprop2}
\end{align} for all $v \in V$. We can write $\varphi_0 = a_0 + v_1a_1$, where $a_0,a_1$ do not involve $v_1$. Now set $v=v_1$ in \eqref{eq:proofprop1}. Since $a_0$ commutes with $v_1$ and $a_1$ anticommutes with $v_1$, this gives $a_0v_1 - v_1^2va_1 = a_0v_1 + v_1^2a_1$. This implies $a_1 = 0$, which means $\varphi_0$ does not contain $v_1$. Proceeding with $a_0$ we can prove in the same way that it does not contain $v_2$, and so on. We conclude that $\varphi^0$ does not contain any $v_i$.
\begin{exc} Give a similar argument proving that $\varphi_1$ does not contain any $v_i$, using equation \eqref{eq:proofprop2}. \end{exc} Hence, $\varphi$ does not involve any of the $v_i$, which means $\varphi \in k^\times$. 
\end{proof}

\begin{defn}[Pin and Spin]
The \emph{Pin group} $\Pin(V,g)$ associated to $\Cl(V,q)$ is the subgroup of $\Cl(V,q)$ generated by elements $v \in V$ with $g(v,v) = \pm 1$. The \emph{Spin group} $\Spin(V,g)$ is given by $$\Spin(V,g) := \Pin(V,g) \cap \Cl(V,g)_0.$$ 
\end{defn}
\begin{expl}Again, we can consider the standard Pin and Spin groups
\begin{align} \Pin_{p,q} &:= \Pin(\R^{p,q}) \\ \Spin_{p,q}&:=\Spin(\R^{p,q}).\end{align}.In this course, the Euclidean groups $$\Pin_n := \Pin_{n,0}$$ and $$\Spin_n := \Spin_{n,0}$$ will be most important. Below, we will see that $\Spin_1 = \{1\}, \Spin_2 \cong S^1, \Spin_3 \cong SU(2) \cong S^3$. In physics, one usually considers the Lorentz Spin group $\Spin_{1,n}$ (and sometimes also the Pin group $\Pin_{1,n}$).
\end{expl}
Notice that by Proposition \ref{prop:reflections}, these groups act by reflections on $V$. Since reflections with respect to a certain symmetric bilinear form preserve that form, it follows that $\tilde{\rho}(\Pin(V,g)) \subset O(V,g)$. Since reflection always have determinant $-1$, we can further conclude that $\tilde{\rho}(\Spin(V,g)) \subset SO(V,g)$ This is an important step in the proof of the following central theorem.
\begin{thm}\label{thm:exactsequence}
Let $k = \R$. Then there are exact sequences 
\begin{align}1 \to \Z_2 \to &\Pin_{p,q} \to O(p,q) \to 1 \label{eq:Pin_SES}\\
1 \to \Z_2 \to &\Spin_{p,q} \to SO(p,q) \to 1 \label{eq:Spin_SES}
\end{align}
\end{thm}
\begin{rem}Interestingly, if $k \neq \R$, a number of things can go wrong. If $\sqrt{-1} \in k$, then \footnote{This happens in $\C$ but also in finite fields, e.g. in $\Z_5$ we have $2^2 = 4 = -1$.} $\Z_2$ has to be replaced by $F = \pm 1, \pm \sqrt{-1}$. If there is $a \in k$ such that neither $a$ nor $-a$ have a square root, then the maps do not cover the orthogonal group. Also, the theorem is not true for degenerate bilinear forms, since the kernel of the twisted adjoint representation becomes larger.
\end{rem}
For the proof we have to introduce another technical tool, the \emph{norm mapping} on the Clifford algebra.
\begin{defn}[Transpose and Norm]
The \emph{transpose} ${}^T\colon \Cl(V,g) \to \Cl(V,g)$ is defined by $(v_1\cdots v_k)^T = v_k \cdots v_1$ and extending linearly.  
The \emph{norm} is the map $N\colon \Cl(V,g) \to \Cl(V,g)$ defined by \begin{equation}N(\varphi) = \varphi \cdot \epsilon(\varphi^T)\end{equation}
\end{defn}
The importance of the norm is explained by the following proposition. 
\begin{prop}
Suppose $g$ is non-degenerate, then $N \colon \Gamma(V,g) \to k^\times$ is a group homomorphism. 
\end{prop}
\begin{proof}
First we have to check that for $\varphi \in \Gamma(V,g)$ we have $N(\varphi) \in k^\times$. By definition, we have 
$\varepsilon(\varphi)v\varphi^{-1} \in V$. The transpose acts trivially on $V$, hence we have
\[ \varepsilon(\varphi)v\varphi^{-1} = \left(\varepsilon(\varphi)v\varphi^{-1}\right)^T = (\varphi^T)^{-1}v\varepsilon(\varphi^T). \]
Hence we get 
$$v = \varphi^T\varepsilon(\varphi)v\varphi^{-1}\varepsilon(\varphi^T)^{-1} = \tilde{\rho}(\varepsilon(\varphi^T)\varphi)v.$$
Since for $\varphi \in \Gamma(V,g)$ we have also $\varphi^T,\varepsilon(\varphi) \in \Gamma(V,g)$, we have $\varepsilon(\varphi^T)\varphi \in \ker {\tilde{\rho}}\big|_{\Gamma(V,g)} = k^\times$. Applying $\varepsilon$ and the transpose, we conclude $N(\varphi) \in k^\times.$\\
Next we check that $N$ is a homomorphism. For this, we simply compute
$$N(\varphi\tau) = \varphi\tau\varepsilon(\tau^T)\varepsilon(\varphi^T) = \varphi N(\tau) \varepsilon(\varphi^T) = N(\varphi)N(\tau).$$
\end{proof}
\begin{proof}[Proof of theorem \ref{thm:exactsequence}]
By the theorem of Cartan-Dieudonn\'e, every element of $g \in O(p,q)$ is a product of reflections $g=r_{v_1}\cdots r_{v_k}$. Since we have $r_{tv} = r_v$, and over $\R$ we can rescale every vector to have length $\pm 1$, we conclude that the Pin group surjects on the orthogonal group. Similarly, the special orthogonal group is generated by even products of reflections, hence the Spin group surjects on the special orthogonal group.
Let $x \in \ker \tilde{\rho} \cap \Pin$. Then $x \in k^\times$ by proposition \ref{prop:ker}. Write $x=v_1\cdots v_k$. Then $N(x) = N(v_1)\cdots N(v_k) = \pm 1$. But for $x \in k^\times$, we have $N(x) = x^2$ and hence $$x \in \ker \tilde{\rho} \cap \Pin  \Leftrightarrow x \in k^\times, x^2 = 1 \Leftrightarrow x= \pm 1.$$
\end{proof}
\section{Classification of Clifford algebras over $\R$ and $\C$}
The classification is not strictly needed for the purpose of this course, but we include for two reasons. First, it is a very interesting topic in its own right, and linked to the deep phenomenon of Bott periodicity. Second, we hope that it serves to reduce confusion over complexifying real algebras and representations. \\
We have the following important fact about the tensor product of Clifford algebras. 
\begin{thm}\label{thm:decomposition}
Let $(V_1,q_1)$ and $(V_2,q_2)$ be $k$-vector spaces with symmetric bilinear forms. Denote $(V,g) = (V_1\oplus V_2,g_1\oplus g_2)$ the orthogonal sum of $V_1,V_2$. Then the graded tensor product of $\Cl(V_1,g_1)$ and $\Cl(V_2,g_2)$ satisfies 
$$\Cl(V_1,g_1) \hat{\otimes} \Cl(V_2,g_2) = \Cl(V,g).$$
\end{thm}
\begin{proof}
Since $V = V_1 \oplus V_2$, any $v \in V$ decomposes as $v = v_1 + v_2$ with $v_i \in V_i$. The map $f \colon V \to  \Cl(V_1,g_1) \hat{\otimes} \Cl(V_2,g_2)$ defined by $f(v) = v_1 \hat{\otimes} 1 + 1 \hat{\otimes} v_2$ satisfies $f(v)f(v) = -g(v,v)1$ (check it), hence by universality $f$ extends to a algebra map $\Cl(V,g) \to \Cl(V_1,g_1) \hat{\otimes} \Cl(V_2,g_2)$. 
\end{proof}
\subsection{Classification of real algebras}
Recall the standard Clifford algebra $\Cl_{p,q}=\Cl(\R^{p+q},g_{p,q})$, where $g_{p,q}$ is the standard metric of signature $(p,q)$. The following is an immediate consequence of theorem \ref{thm:decomposition}. 
\begin{cor}
The standard Clifford algebra $\Cl_{p,q}$ is isomorphic to a tensor product 
\begin{equation}
\Cl_{r,s} \cong \hat{\otimes}^r \Cl_1 \hat{\bigotimes} \hat{\otimes}^s\Cl_{0,1}.
\end{equation}
\end{cor}

\begin{thm}\label{thm:Clifford_Iso}
There is an algebra isomorphism 
\begin{equation}
\Cl_{p,q} \cong \Cl_{p+1,q}^0.
\end{equation}
\end{thm}
\begin{proof}
Again we use the universal property of Clifford algebras. Choose a standard basis $e_0,\ldots, e_{p}, e_{p+1},\ldots,e_{p+q}$ of $\R^{p+1,q}$, i.e. $g_{p+1,q}(e_i,e_i)= 1$ for $0 \leq i \leq r$ and $g_{p+1,q}(e_i,e_i)= -1$ for $r+1 \leq i \leq r+q$. Embed $\R^{p,q}$ by dropping $e_0$ and define $f\colon \R^n \to \Cl^0_{p+1,q}$ by $f(e_i) = e_0e_i$ for $i \geq 1$. Observe that $f(e_i)^2 = g_{p,q}(e_i,e_i)$ (Computation). Hence $f$ extends to an algebra morphism $\Cl_{p,q} \to \Cl_{p+1,q}$.  
\end{proof}
To classify the Clifford algebras over $\R$ we revert to the usual (ungraded) tensor product. This will allow us as identifying Clifford algebras as tensor product of known algebras. We will only sketch the proofs of the classification, for details the reader is referred to \cite{Lawson1990}.
The first observation is the following. 
\begin{thm}\label{thm:clifford2}
There are isomorphims 
\begin{align}
\Cl_{n,0} \otimes \Cl_{0,2} &\cong \Cl_{0,n+2} \\
\Cl_{0,n} \otimes \Cl_{2,0} &\cong \Cl_{n+2,0} \\
\Cl_{p,q} \otimes \Cl_{1,1} &\cong \Cl_{p+1,q+1} 
\end{align}
\end{thm}
\begin{proof}
The idea of the proof is very similar to proofs we have seen before, and uses again the universal property of Clifford algebras. To prove e.g. the first one defines a map 
$\R^{n+2} \to \Cl_n \otimes \Cl_{0,2}$ by choosing orthonormal bases $e_i,e_i', e_i''$ of $\R^{n+2},\R^n,\R^2$ and defining $f(e_i) = e_i' \otimes e_1''e_2''$ for $1\leq i \leq n$ and $f(e_{n+1}) = 1\otimes e_1''$, $f(e_{n+2}) = 1\otimes e_2''$. One then checks easily that this map satisfies the Clifford relation: For $1\leq i \leq n$ we have 
\begin{align*}
f(e_i)^2 = (e'_i \otimes e_1''e_2'')^2 = (e_i')^2 \otimes (e_1''e_2''e_1''e_2'') = (-1) \otimes (-1) = 1,
\end{align*}
and of course this holds also for $e_{n+1}$ and $e_{n+2}$. Hence $f$ extends to an algebra homomorphism $f\colon \Cl_{0,n+2} \to \Cl_{n,0} \otimes \Cl_{0,2}$. Next, we note that the algebra homomorphism is surjective, since it contains all generators in its image: Clearly it contains the generators $1 \otimes e_i''$, $i=1,2$. On the other hand, we have $e_i' \otimes 1 = f(e_i)(1 \otimes e_2'')(1\otimes e_1'')$, hence also the generators $e_i' \otimes 1$ are contained. Since the two algebras have the same dimension ($2^{n+2} = 2^n\cdot 2^2$), we conclude they are isomorphic. The second case is very similar to the first one. For the last case, let $e_1,\ldots,e_{p+1},\varepsilon_1,\ldots, \varepsilon_{q+1} $ be a basis of $\R^{p+1,q+1}$ satisfying $g(e_i,e_j) = \delta_{ij}, g(\varepsilon_i,\varepsilon_j) = -\delta_{ij}$ and $g(e_j,\varepsilon_k) = 0$, and let $e_i',\varepsilon'$ and $e_1'',\varepsilon_1''$ be similar bases of $\R^{p,q}$ and $\R^{1,1}$ respectively. Then one defines the map $f\colon\R^{p+1,q+1} \to \Cl_{p,q}\otimes \Cl_{1,1}$ by 
\begin{align}
f(e_i) &= \begin{cases} e_i' \otimes e_1''\varepsilon'' & 1 \leq i \leq p \\ 1 \otimes e_1'' &i = p+1 \end{cases} \\
f(\varepsilon_i) &= \begin{cases} \varepsilon_i' \otimes e_1''\varepsilon'' & 1 \leq i \leq p \\ 1 \otimes \varepsilon_1'' & i=q+1 \end{cases}
\end{align}
and proceed as before (notice that $(e_1''\varepsilon_1'')^2 =1)$). 
\end{proof}
This means that all Clifford algebras can be expressed as tensor products of $\Cl_{1,0}, \Cl_{2,0},\Cl_{0,1},\Cl_{0,2}$ and $\Cl_{1,1}$. We summarize those in the next proposition. 
\begin{prop}\label{prop:simplecliffs}
We have  (all isomorphisms are isomorphisms of $\R$-algebras)
\begin{align*}
\Cl_{1,0} &\cong \C \\
\Cl_{2,0} &\cong \HH \\
\Cl_{0,1} &\cong \R \oplus \R \\
\Cl_{0,2} &\cong M_2(\R)  \\
\Cl_{1,1} &\cong M_2(\R)
\end{align*}
We denote by $M_n(k)$  the algebra of $n \times n$ matrices over $k$.
\end{prop}
\begin{proof}
The proof goes by simply inspecting basis generators and relations. A basis of $\Cl_{1,0}$ is $1,v$ with relation $v^2 = -1$. For $\Cl_{2,0}$, we have a basis $1,v_1,v_2,v_{12}$. Denoting $i = v_1,j=v_2,k=v_1v_2$, we get $i^2=j^2=k^2 = -1$, $ij = k$, etc. 
$\Cl_{0,1}$ has a basis $1,v$ where $v^2 = +1$. Again we can look at the $\pm 1 $ eigenspaces of $v$, generated by $1+v$ and $1-v$. This gives a decomposition into commuting copies of $\R$. $\Cl_{0,2}$ has a generator $1,v_1,v_2$ satisfying $v_1^2 = v_2^2 = 1, v_1v_2 = -v_2v_1$. An isomorphism with $M_2(\R)$ is given by sending e.g.
$$1 \mapsto \begin{pmatrix}
1 & 0 \\
0 & 1
\end{pmatrix},
v_1 \mapsto \begin{pmatrix}
1 & 0 \\
0 & -1
\end{pmatrix},
v_2 \mapsto \begin{pmatrix}
0 & 1 \\
1 & 0
\end{pmatrix}
$$
One can check this satisfies the Clifford relations, use the universal property and observe that the resulting map is an isomorphism.
Finally, $\Cl_{1,1}$ is generated by $1,v_1,v_2$ with $v_1^2 = 1, v_2^2 =-1, v_1v_2 = -v_2v_1$. An isomorphism can be constructed via 
$$1 \mapsto \begin{pmatrix}
1 & 0 \\
0 & 1
\end{pmatrix},
v_1 \mapsto \begin{pmatrix}
1 & 0 \\
0 & -1
\end{pmatrix},
v_2 \mapsto \begin{pmatrix}
0 & 1 \\
-1 & 0
\end{pmatrix}
$$
\end{proof}

In particular, all Clifford algebras can be constructed as tensor products of $\R,\C,\HH$ and matrix algebras. This is useful because of the following facts from linear algebra: 
\begin{prop}\label{prop:linalgfacts}
Let $A = \C,\HH$. 
\begin{enumerate}[i)]
\item $M_n(A) \cong M_n(\R) \otimes A$, $M_n(\R) \otimes M_k(\R)  \cong M_{nk}(\R)$
\item $\C \otimes_\R \C \cong \C \oplus \C$
\item $\C \otimes_\R \HH \cong M_2(\C)$
\item $\HH \otimes_\R \HH \cong M_4(\R)$
\end{enumerate}
\end{prop}
Using the two propositions above, one can express \emph{any} standard Clifford algebra over the reals as a tensor product of matrix algebras. For the definite Clifford algebras, this is usually summarized in the following table, which is produced using Theorem \ref{thm:clifford2} and Propositions \ref{prop:simplecliffs} and \ref{prop:linalgfacts}:
\begin{table}[h]
\centering
\begin{tabular}{c|c|c}
$n$ & $\Cl_{n,0}$ & $\Cl_{0,n}$ \\
\hline
$1$ & $\C$ & $\R \oplus \R$ \\
$2$ & $\HH$ & $M_2(\R)$ \\
$3$ & $\HH \oplus \HH $ & $M_2(\C)$ \\
$4$ & $M_2(\HH)$ & $M_2(\HH)$ \\
$5$ & $M_4(\C)$ & $M_2(\HH) \oplus M_2(\HH)$ \\
$6$ & $M_8(\R)$ & $M_4(\HH)$ \\
$7$ & $M_8(\R) \oplus M_8(\R)$ & $M_8(\C)$ \\
$8$ & $M_{16}(\R)$ & $M_{16}(\R)$
\end{tabular}
\caption{Clifford algebras $\Cl_{n,0}$ and $\Cl_{0,n}$ for $n \leq 8$.}
\end{table}
After that, the table repeats in the following sense. Theorem \ref{thm:clifford2} implies that $\Cl_{n+8,0} \cong \Cl_{n,0} \otimes \Cl_{8,0} \cong \Cl_{n,0} \otimes M_{16}(\R)$. Since all Clifford algebras are matrix algebras $M_k(A)$ (or direct sums thereof), we immediately get $M_k(A) \otimes M_{16}(\R)$. This is what people call the $8$-periodicity of real Clifford algebras. One can also read off the isomorphism type of the indefinite Clifford algebras by \ref{thm:clifford2}. 
\begin{rem}
If the bilinear form $g$ is degenerate, one can write $V$ as $(V,g) = (V',g') \oplus (W,0)$. Then one can use theorem \ref{thm:decomposition} to rewrite the corresponding Clifford algebra as the graded tensor product of a matrix algebra and the exterior algebra on $V$. 
\end{rem}
\subsection{Classification of complex Clifford algebras}
Over the complex numbers, the result is considerably simpler. 
First, let $(V,g)$ be a real vector space and denote by $(V_\C,g_\C)$  the complexification $V_\C = V \otimes_\R \C$ and the $g_\C$ the complex bilinear extension of $g$. Then, we observe that $\Cl(V_\C, g_\C) \cong \Cl(V,g)_\C = \Cl(V,g) \otimes \C$. Next, remember that over the complex numbers all non-degenerate symmetric bilinear forms are isometric. These two observations result in the following corollary of theorem \ref{thm:clifford2}: 
\begin{cor}\label{cor:ComplexPeriod}
Denote $\Cl^c_n$ the Clifford algebra on $\C_n$ with standard symmetric bilinear form. Then $\Cl^c_{n+2} \cong \Cl^c_n \otimes \Cl_2^c$.
\end{cor} 
\begin{proof} 
This follows directly from the observations above applied to any of three isomorphisms in theorem \ref{thm:clifford2}.\end{proof}
\begin{exc}\label{exc:iso}
Prove Corollary \ref{cor:ComplexPeriod} directly by following the proof of Theorem \ref{thm:clifford2} for the following map $f\colon \C^{n+2}\to \Cl^c_{n}\otimes \Cl_{2}^c$: 

\begin{equation}
f(e_i) = \begin{cases} \ii e_i' \otimes e_1''e_2'' & 1 \leq i \leq n\\
1 \otimes e_{i-n}'' & i = n+1,n+2
\end{cases}
\end{equation}
where $e_i,e_i',e_i''$ are ONB of $\C^{n+2},\C^n,\C^2$ respectively. 
\end{exc}
Now we can classify the complex Clifford algebras in the same way we classified the real ones. The first two can be read off directly from complexifying proposition \ref{prop:simplecliffs}: 
\begin{prop}\label{prop:ComplexSimpleCliffs}
The first two complex Clifford algebras are given by 
\begin{align}
\Cl_1^c &\cong \C \oplus \C \\
\Cl_2^c &\cong M_2(\C)
\end{align}
\end{prop}

Hence the complex Clifford algebras are 2-periodic in the same sense the real Clifford algebras are 8-periodic. Let us explicitly write down $\Cl^c_n$ for all $n$: 
\begin{cor}
Let $n \in \N$. 
Then we have 
\begin{equation}
\Cl_n^c \cong \begin{cases} M_{2^k}(\C) \oplus M_{2^k}(\C) & n=2k+1 \text{ odd}\\ 
M_{2^{k}}(\C) & n=2k \text{ even} \end{cases}\label{eq:ComplexCliffClassification}
\end{equation}
\end{cor}
\section{Representations}
We now want to study some representations of Clifford algebras. Since the Clifford algebra is naturally a superalgebra, we are interested in supermodules over $\Cl(V,g)$. Any supermodule defines, in particular, a representation of $\Cl(V,g)$. All the Clifford algebras are matrix algebras or direct sums of two matrix algebras, whose irreducible representations are well known: 
\begin{thm} \label{thm:irreps}
Let $A$ be one of  the $\R$-algebras $\R,\C$ or $\HH$. Then up to isomorphism there is a unique irreducible representation of the matrix algebra $M_n(A)$ and it is given by the action on $A^n$. The algebra $M_n(A) \oplus M_n(A)$ has two irreducible representations up to isomorphism, given by the vector representation of one factor and the trivial representation of the other factor. 
\end{thm}
\begin{proof}
This follows from the fact that matrix algebras are simple. See e.g. the book by Etingof et al on representations of algebras. 
\end{proof}
\subsection{Representations of the complex Clifford algebra $\Cl_n^c$} As we have seen before, the complex Clifford algebra $\Cl_n^c$ is either a matrix algebra (for $n$ even) or two copies of a matrix algebra (for $n$ odd). 
Hence, there is exactly one irreducible representation of $\Cl_nc$ for $n$ even, and two irreducible representations for $n$ odd. The irreducible representations of the $\Cl_n^c$ are known as the \emph{complex spinors}.
\begin{defn}[Spin representation of $\Cl_n^c$] \label{def:complex_spin_rep_cl}
For $n = 2k$ even, the \emph{spin representation} of $\Cl_n^c$ is the unique irreducible representation of $\Cl_n^c$ of dimension $2^k$. The action is given by the isomorphism $\Cl_n^c \cong M_{2^k}(\C)$. \\
If $n = 2k+1$ is odd, then the two irreducible representations of dimension $2^k$ are called \emph{spin representations}.
\end{defn} It is desirable to have a better description that does not use the indirect isomorphism above. For even $n$ a nice description exists. 

Let $(V,g)$ be a $2k$-dimensional euclidean space (i.e. $g$ is positive definite). Let $J$ be a complex structure on $V$, i.e. an antisymmetric map $(g(v,Jw) = -g(Jv,w))$ with $J^2 = -1$. Then $J_\C \colon V_\C \to V_\C$ also squares to $-1$, hence 
\begin{equation}V_\C = V^{1,0} \oplus V^{0,1}\label{eq:decomposition}\end{equation} where $V^{1,0}$ (resp. $V^{0,1}$) denotes the $+\ii$ ($-\ii$) eigenspace of $J$. Conjugation on $\C$ extends to an antilinear map $V_\C \to V_\C, v \to \bar{v}$ which exchanges $V^{1,0}$ and $V^{0,1}$. 
\begin{prop}$g$ vanishes when restricted to either $V^{1,0}$ or $V^{0,1}$. 
\end{prop} 
\begin{proof}To see this, consider $v,w \in V^{1,0}$. Then $Jv = \ii v, Jw = \ii w$. But then antisymmetry of $J$ implies that 
$$ \ii g(v,w) = g(\ii v, w) = g(Jv,w) = -g(v,Jw) = -g(v, \ii w) = -\ii g(v,w).$$
Hence $g(v,w) = 0.$
\end{proof}
\begin{prop}\label{prop:spinors}
The spinor representation of $\Cl(V_\C,g_\C)$ is isomorphic to the representation $S_V = \bigwedge V^{1,0}$ with action given by the formula

\begin{equation}
v.\alpha = \sqrt{2} (v^{1,0} \wedge \alpha - \iota_{v^{0,1}}\alpha)
\end{equation}
where $v = v^{1,0} + v^{0,1}$ is the decomposition \eqref{eq:decomposition} and the contraction is defined using $g_\C$.
\end{prop} 
\begin{proof}
 Denote $e_1,f_1, \ldots e_n,f_n$ an ONB of $V$ over $\R$ such that $Jei = f_i, Jf_i = -e_i$.
Then $\varepsilon_i= \frac{1}{\sqrt{2}}(e_i - if_i) $ is a basis of $V^{1,0}$ and $\bar{\varepsilon}$ is a basis of $V^{0,1}$.
Take a vector and decompose it as $v = v^{1,0} + v^{0,1}$. We have to show the Clifford relation 
$$v.(v.(\alpha)) = -g(v,v)\alpha.$$ 
Since $v^{1,0}.(v^{1,0}\alpha ) = v^{0,1}.(v^{0,1}\alpha ) = 0$, we just compute 
\begin{align*}
v^{0,1}.(v^{1,0}.\alpha) + v^{1,0}.(v^{0,1}.\alpha) &= -2\iota_{v^{0,1}}(v^{1,0}\alpha) - 2v^{1,0}\wedge\iota_{v^{0,1}}\alpha  \\
&= -2(\iota_{v^{0,1}}v^{1,0})\alpha + 2v^{1,0}\iota_{v^{0,1}}\alpha = -2g(v^{0,1},v^{1,0})\alpha = -g(v,v)\alpha.
\end{align*}
Hence this defines a representation of $\Cl_{2n}^c$. It also has the correct dimension $2^n$. So, we simply have to check that the induced map $\rho\colon\Cl_{2n}^c \to \End(\bigwedge V^{1,0})$ is surjective. To do so, consider the basis $\varepsilon^\wedge_I$ of $\bigwedge V^{1,0}$. Then, we claim that for any $I$ and $J$ there exists an element $\varphi_{IJ}$ such that $\rho(\varphi)\varepsilon^\wedge_I = \varepsilon^\wedge_J$. Up to a sign, this element is given by  $\varphi_{IJ} = \varepsilon_{J - I} \bar{\varepsilon}_{I-J}$. 
\end{proof}
\subsection{The volume element}
 Choose an orientation of $\C^n$ and let $v_1,\ldots,v_n$ be an oriented basis of $\C^n$. 
\begin{defn}\label{def:complex_volume}
The \emph{complex volume element} $\omega \in \Cl_{n}^c$ is 
\begin{equation}
\omega = \ii^{\lfloor \frac{n+1}{2}\rfloor}v_1\cdot \ldots \cdot v_n
\end{equation}
\end{defn}
\begin{exc}\label{exc:complex_volume}
\begin{enumerate}
\item The complex volume element is independent of the choice of basis. 
\item The complex volume element satisfies 
\begin{align}
\omega^2 &= 1 \\
v\omega &= (-1)^{n-1}\omega v
\end{align}
\end{enumerate}
\end{exc}
Since $\omega^2 = 1$, then as usual we can construct projectors $\pi^+ = \frac{1}{2}(1+\omega), \pi^-= \frac12(1-\omega)$ satisfying $\pi^++\pi^-=1, \pi^+\pi^- = \pi^-\pi^+ = 0$, $(\pi^\pm)^2 = \pi^\pm$. The second equation implies that $\omega$ is central if $n$ is odd. On the other hand, if $n$ is even the $\omega$
\emph{supercommutes} with all elements of $\Cl_n^c$.
This has the following implications: 
\begin{prop}
\hfill 
\begin{enumerate}[i)]
\item If $n$ is odd, then $\Cl^c_{n} \cong \Cl_n^{c,+} \oplus \Cl_{n}^{c,-}$, 
where $\Cl^{c,\pm}_{n} = \pi^{\pm}\Cl_{n}^c$ are isomorphic subalgebras satisfying $\alpha(\Cl^{c,\pm}_{n}) = \Cl^{c,\mp}_{n}$. This decomposition coincides with the one in \eqref{eq:ComplexCliffClassification}.
\item If $n$ is even, any supermodule $(V,\rho)$ of $\Cl_{n}^c$ decomposes into the $+1$ and $-1$ eigenspaces of $\omega$, $V = V^+ \oplus V^-$. For any $v \in \C^n-\{0\}$, $\rho(v)\colon V^\pm \to V^\mp$ is an isomorphism.
\end{enumerate}
\end{prop}
\begin{proof}
For i), we notice that $\omega$ is central (i.e. commutes with all elements) if $n$ is odd. Hence also $\pi^\pm$ are central, so elements of the form $\pi^\pm\alpha$ form a subalgebra. Also notice that for $n$ odd we have $\alpha(\omega) = -\omega$ and hence $\alpha\pi^\pm = \pi^\mp$. To see the last point, first notice that it is true for $n=1$: Here the volume element is $\ii v_1$, which is the element defining the splitting to $\C \otimes_\R \C \cong \C \oplus \C$. Next notice that under the isomorphism constructed in Exercise \ref{exc:iso}, the volume element gets sent to 
\begin{align*}
f(\omega_{2k+1}) &= f(\ii^{k+1}e_1\cdots e_{2k+1}) = \ii^{k+1}\ii^{2k-1}(e_1'\otimes e_1''e_2'') \cdots (e_{2k-1}'\otimes e_1''e_2'')(1\otimes e_1'')(1\otimes e_2'') \\
&=\ii^k(-1)^ke_1'\cdots e_{2k-1}'\otimes (e_1''e_2'')^k = \ii^ke_1'\cdots e'_{2k-1}\otimes 1 = \omega_{2k-1}\otimes 1.
\end{align*} Hence, the splitting induced by $\omega_{2k+1}$ is the same as the one induced by $\omega_{2k-1}$.  \\For ii), notice that   $[\omega,\alpha]_s = 0$ for all $\alpha \in \Cl_n^c$. Hence $\rho(\omega) \in \End_{\Cl_{n}^c}V$ In particular, Clifford multiplication by any $v \in \R^n$  exchanges the $\pm 1$ eigenspaces of $\omega$. 
\end{proof}
Next we characterize the complex spinors for odd $n$. 
\begin{prop}
Let $V$ be an irreducible representation of $\Cl_n^c$, $n$ odd. Then $\omega$ acts either by $+1$ or $-1$ on $V$, and these are the two inequivalent representations of $\Cl_n^c$. In particular, if $\omega$ acts by $\pm 1$ then $V$ is the unique irreducible representation of $\Cl_n^{c,\pm}$. 
\end{prop}
\begin{proof}
Since $\omega^2 =1$, $V$ decomposes into the $\pm 1$ eigenspaces $V^\pm$ of $\omega$. But since $\omega$ is central, we have $V = V^+$ or $V^-$. It is clear that these representations are inequivalent. If $\omega$ acts by $\pm 1$, then $\Cl_n^{c,\mp}$ acts trivially. This proves the claim.
\end{proof}
\begin{prop} Let $n$ be even and let $(V,\rho)$ be the unique irreducible representation of $\Cl_n^c$ with decomposition into $\omega$-eigenspaces $V = V^+ \oplus V^-$.  Then these two representations are invariant under $(\Cl_n^c)_0$, and they give the two irreducible representations of $\Cl_{n-1}^c$ under the isomorphism $\Cl_{n-1}^c \cong (\Cl_n^c)_0$ of Theorem \ref{thm:Clifford_Iso}.
\end{prop}
\begin{proof}
It is clear that $V^\pm$ are invariant under $(\Cl_n^c)_0$, since this subalgebra commutes with $\omega_n$. Next, recall that the isomorphism $\Cl_{n-1}^c \to (\Cl_{n}^c)_0$ is given by $f(e_i) =e_ie_n$. Hence, $\omega_{n-1}$ gets sent to $\pm \omega_n$ under this isomorphism, which proves the claim.
\end{proof}
\begin{cor}\label{cor:diagonal}
Let $n$ be even. Then, the isomorphism $\Cl_n^c \cong (\Cl_{n+1}^c)_0$ is the diagonal embedding $M_{2^n}(\C) \to M_{2^n}(\C) \oplus M_{2^n}(\C)$.
\end{cor}
\begin{proof}
$\alpha$ acts trivially on $(\Cl_{n+1}^c)_0$ but exchanges the two copies $\Cl_{n+1}^\pm$.
\end{proof}
\subsection{Representations of the Spin group}
Given the spinor representations of the complex Clifford algebra, we can now easily define the Spin representations of the spin groups. 
\begin{defn}[Complex Spin representation of $\Spin_n$]
Let $\Spin_n \subset \Cl_n \subset \Cl_n^c$. Then the \emph{complex spin representation} of $\Spin_n$ is the restriction of the complex spin representation of $\Cl_n^c$. We denote the complex spin representation by $\Delta_n$.
\end{defn}
\begin{rem}
If $n$ is odd, this is independent of which of the two spin representations of $\Cl_n^c$ is used. This follows from the fact $(\Cl_n^c)_0$ sits diagonally (Corollary \ref{cor:diagonal}).
\end{rem}
\begin{prop}
When $n$ is odd, the representation $\Delta_n$ is irreducible. When $n$ is even, the representation $\Delta_n = \Delta_n^+ \oplus \Delta_n^-$ splits into the direct sum of two irreducible representations. 
\end{prop}
\begin{rem}
Elements of $\Delta_n^\pm$ are known as \emph{chiral} or \emph{Weyl} spinors. 
\end{rem}
\begin{proof}
If $n$ is odd, then $\Delta_n$ is an irreducible representation of $(\Cl_n^c)_0$. Hence it is also irreducible under $\Spin_n$, since $\Spin_n$ contains a basis of $(\Cl_n^c)_0$. \\
If $n$ is even, then the complex spinor representation splits as a direct sum of two subspaces invariant under $(\Cl_n^c)_0$. These are two irreducible inequivalent representations of $\Spin_n$. 
\end{proof}
To study representations of Lie groups it is always desirable to know their Lie algebras, in particular because of the following fact: 
\begin{prop}
$\Spin_n$ is simply connected for $n\geq 3$. Hence $\Spin_n \to SO(n)$ is the universal cover for $n\geq 3$. For $n = 2$, it is the nontrivial double cover of $SO(2) \cong S^1$. 
\end{prop}
\begin{proof}
Recall the short exact sequence \eqref{eq:Spin_SES} from theorem \ref{thm:exactsequence}
\begin{equation}
1 \to \Z_2 \to \Spin_n \to SO(n) \to 1.
\end{equation}
We know that $\pi_1(SO(n)) = \Z_2$. Hence, from covering theory, to prove these statements it is sufficient to prove that the elements $\pm 1 \in \ker \tilde{\rho}$  are connected by a continuous path in $\Spin_n$. Such a path is given by 
$$\gamma(t) = (e_1 \cos t + e_2 \sin t)(-e_1\cos t + e_2 \sin t) \in \Spin_n$$
where $0\leq t \leq \pi/2.$
\end{proof}
We know that the Lie algebras of $\Spin_n$ and $SO(n)$ are isomorphic, but sometimes it will be useful to have an explicit isomorphism. 
\begin{prop}
The Lie subalgebra of $(\Cl_n,[\cdot,\cdot])$ of the subgroup $\Spin_n \subset Cl_n^\times$ is spanned by $\{e_ie_j\}_{i < j}$. 
\end{prop}
\begin{proof}
This Lie algebra consists of tangent vectors to curves in $\Spin_n$ at $1$. In particular, consider a curve similar to the one above
\begin{equation}
\gamma_{ij}(t) =  (e_i \cos t + e_j \sin t)(-e_i\cos t + e_j \sin t) = \cos 2t + \sin 2t e_ie_j
\end{equation}
Then $\gamma(0) = 1$ and $\dot{\gamma}(0) = 2e_ie_j$. The claim now follows from the fact that $\dim \Spin_n = \dim SO(n) = n(n-1)/2$. 
\end{proof}

Recall that the Lie algebra of $SO(n)$ is the space of skew-symmetric matrices and denoted $\mathfrak{so}_n$. There is an isomorphism 
$\bigwedge^2 \R^n \to \mathfrak{so_n}$ given by sending $e_i \wedge e_j\mapsto E_{ji} - E_{ij}$ or equivalently, we send $v \wedge w$ to the antisymmetric endomorphism  $\R^n \to \R^n$ given by 
\begin{equation}v \wedge w (x) = v g(w,x) - w g(v,x).\label{eq:isosonwedge2}\end{equation}
Consider again the adjoint representation 
\begin{equation}
\Spin_n  \xrightarrow{\rho} SO(n)
\end{equation}
(on $(\Cl_n)_0$ the twisted adjoint representation and the adjoint representation agree). 
\begin{prop}
Identify $\mathfrak{so}_n \cong \bigwedge^2\R^n$ using the isomorphism \eqref{eq:isosonwedge2}. Then the Lie algebra isomorphism $\rho_*\colon \Spin_n \to \bigwedge^2\R^n$ induced by $\rho$ is given by 
\begin{equation}
e_ie_j \mapsto 2e_i\wedge e_j\label{eq:propLieAlgIso1}
\end{equation}
Equivalently 
\begin{equation}
(\rho_*)^{-1}(v\wedge w) = \frac{1}{4}[v,w]\label{eq:propLieAlgIso2}
\end{equation}
\end{prop}
\begin{proof}
Consider again the curve $\gamma(t) = \gamma_{ij}(t/2) = \cos t + \sin t e_ie_j$ defined above. Then for $x \in \R^n$,
\begin{align*}(\rho_* e_ie_j)(x) &= \restr{\frac{d}{dt}}{t=0}\rho(\gamma(t))(x) = \restr{\frac{d}{dt}}{t=0}\gamma(t)x\gamma(t)^{-1} \\
&= e_ie_j x - x e_ie_j = -2g(e_j,x)e_i + 2g(e_i,x)e_j = 2(e_i \wedge e_j)x.
\end{align*}
This proves \eqref{eq:propLieAlgIso1}. 
To see \eqref{eq:propLieAlgIso2}, note that 
$$\frac14[e_i,e_j] = \frac12e_ie_j = \rho_*^{-1}e_i\wedge e_j.$$
\end{proof}
\begin{expl}[$\Spin_3$]
The Lie algebra of $\Spin_3$ is generated by $u = e_1e_2, v=e_2e_3, w=e_1e_3$ subject to the relations 
\begin{align*}
[u,w] &= [e_1e_2,e_1e_3] = e_1e_2e_1e_3 - e_1e_3e_1e_2 = -e_1^2e_2e_3 + e_1^2e_3e_2 = 2v \\
[u,v] &= [e_1e_2,e_2e_3] = e_1e_2^2e_3 - e_2e_3e_1e_2 = -2w \\
[v,w] &= [e_2e_3,e_1e_3] = e_2e_3e_1e_3 - e_1e_3e_2e_3 = -2u\\
\end{align*}
The complexification of this Lie algebra is isomorphic to $\mathfrak{sl}_2(\C)$: Setting \begin{align*}h &= \ii u \\ e &= 1/2(iv+ w)  \\ f&=1/2 (iv - w)\end{align*}
we obtain the brackets of $\mathfrak{sl}_2(\C)$given by $[e,f] = h, [h,e]=2e, [h,f]=-2f$.
 This Lie algebra sits diagonally in $Cl_3^c = M_2(\C) \oplus M_2(\C)$, and the spinor representation is the fundamental representation of $\mathfrak{sl}_2(\C)$ where $h$ has heighest weight $+1$. In physics this is called the Spin-$\frac12$-representation\footnote{In physics sometimes a different convention is used where $[h,e] = e, [h,f] = -f$, then $h$ has weight $\frac12$ in the fundamental representation}. 
\end{expl}

\subsection{Some remarks on the real case}
In the real case, we can again use theorem \ref{thm:irreps} to classify irreducible representations of the $\Cl_{p,q}$, since they are isomorphic to matrix algebras or a direct sum of matrix algebras. The representation theory of this algebras has very interesting details, but we restrict ourselves to a single case interesting in physics. 
\begin{expl}[Majorana representation]
Consider the algebra $\Cl_{1,3}$. By theorem \ref{thm:clifford2} we know that \begin{equation}\Cl_{1,3} \cong \Cl_{1,1} \otimes \Cl_{0,2} \cong M_2(\R) \otimes M_2(\R) \cong M_4(\R).
\end{equation}
Hence, the \emph{real} algebra $\Cl_{1,3}$ has an irreducible real representation on $\R^4$. In physics, this representation is known as the \emph{Majorana representation}. Note that this fact is unique to signature $(+,-,-,-)$! In fact, $\Cl_{3,1} \cong \Cl_{2,0} \otimes \Cl_{1,1} \cong M_2(\HH)$ does not have a similar representation.
\end{expl}
\chapter{Geometry}\label{ch:Geometry}
In this chapter, we set up the geometric foundations for Spin Geometry. In the first section we will look at the elementary objects that shall play a big role in this course: Vector bundles, principal bundles and connections on them. In the next section, we will see how one constructs \emph{Characteristic classes} out of these objects via Chern-Weil theory. 
Finally, we will discuss some aspects of Riemannian Geometry.
\section{Vector bundles, principal bundles, connections}
\subsection{Vector bundles}
We start with the definition of a vector bundle. 
\begin{defn}[Vector bundle]
Let $M$ be a manifold and $k \in \{\R,\C\}$.
A \emph{rank $n$ $k$-vector bundle over $M$} is  
a pair $(E,\pi)$, where $E$ is a manifold and $\pi \colon E \to M$ is a surjective submersion, such that there is a cover $\mathfrak{U}=\{U_\alpha\}_{\alpha \in A}$ of $M$ satisfying
\begin{enumerate}[i)]
\item{The cover $\mathfrak{U}$ \emph{trivializes} $E$, that is, for every $\alpha \in A$ there exists a diffeomorphism $\Psi_\alpha\colon \pi^{-1}(U_\alpha) \to U_\alpha \times k^n$ such that 
\begin{equation} \begin{tikzcd}
\pi^{-1}(U_\alpha) \arrow[rd, "\pi"] \arrow[rr, "\psi_{\alpha}"] &          & U_\alpha \times k^n \arrow[ld, "\pi_1"'] \\
                                                              & U_\alpha &                                       
\end{tikzcd}
\label{eq:defVB1}
\end{equation}
commutes, 
}
\item For all $\alpha \in A$ and $u \in U_\alpha$, $\pi^{-1}(u)$ is a $k$-vector space and the map \begin{equation}
\restr{\psi_\alpha}{\pi^{-1}(u)}\colon \pi^{-1}(u) \to \{u\} \times k^n \label{eq:defVB2}
\end{equation}
is an isomorphism of vector spaces. 
\end{enumerate}
\end{defn}
Let us introduce some terminology. $M$ is called the \emph{base} (or \emph{base space}) of the vector bundle. $E$ is called the \emph{total space}, and $\pi$ the projection. For $u\in M$, $\pi^{-1}(u)$ is called the \emph{fiber (of $E$) over $u$}, and denoted $E_u$. $(U_\alpha,\psi_\alpha)$ is called a \emph{local trivialization} and $\mathfrak{U}$ is called a \emph{trivializing cover}. For $\alpha,\beta \in A$, let $U_{\alpha\beta} = U_\alpha \cap U_\beta.$ By diagram \eqref{eq:defVB1} and \eqref{eq:defVB2}, the maps\footnote{Often, in the literature one finds opposite convention for the indices. However, we find this intuitive because it is the transition map \emph{from $\alpha$ to $\beta$}.}
\begin{equation}
\tilde{g}_{\alpha\beta}= \psi_\beta\circ\psi^{-1}_\alpha \colon U_{\alpha\beta}\times k^n \to U_{\alpha\beta} \times k^n
\end{equation}
satisfy $\tilde{g}_{\alpha\beta}(u,v) = (u,g_{\alpha\beta}(u)v)$, where $g_{\alpha\beta}(u) \in GL_n(k)$. The corresponding maps 
\begin{equation}
g_{\alpha\beta}\colon U_{\alpha\beta} \to GL_n(k)
\end{equation}
are called the \emph{gluing maps}. By construction, they satisfy, for all $\alpha,\beta,\gamma\in A$  and 
\begin{subequations}\label{eqs:gluingmaps}
\begin{align}
g_{\alpha\alpha}(u) &= \mathrm{id}_{k^n} \forall, \quad u \in U_{\alpha} \\
g_{\alpha\beta}(u) &= g_{\beta\alpha}(u)^{-1} ,\quad \forall u \in U_{\alpha\beta} \\
g_{\beta\gamma}(u)g_{\alpha\beta}(u) &= g_{\alpha\gamma(u)},\quad\forall u \in U_{\alpha\beta\gamma} = U_{\alpha} \cap U_{\beta} \cap U_{\gamma}.
\end{align}
\end{subequations}
\begin{defn}[Vector bundle morphisms]
If $E,F$ are vector bundles over $M$ then a \emph{vector bundle morphism} is a smooth map $\Psi\colon E \to F$ such that the diagram 
\begin{equation}
\begin{tikzcd}
E \arrow[rd, "\pi"] \arrow[rr, "\Psi"] &   & F \arrow[ld, "\pi'"'] \\
                                       & M &                      
\end{tikzcd}
\end{equation}
commutes and $\restr{\Psi}{E_u}=:\Psi_u \colon E_u \to F_u$ is linear. A \emph{vector bundle isomorphism} is a vector bundle morphism which is also a diffemorphism. The set of vector bundle morphisms from $E$ to $F$ is denoted $\underline{\Hom}(E,F)$.
\end{defn}
\begin{rem}
 Given two vector bundles $E$ and $F$ over $M$, we can always find a cover of $M$ that trivializes both. Namely, given a trivializing cover $\{U_\alpha\}_{\alpha \in A}$ of $E$ and $\{V_\beta\}_{\beta \in B}$ of $F$, the cover $\{U_\alpha \cap V_\beta\}_{(\alpha,\beta) \in A \times B}$ is a trivializing cover of both $E$ and $F$. 
\end{rem}
\begin{exc} 
Suppose the rank of $E$ is $n$ and the rank of $F$ is $m$. Prove that a vector bundle morphism $\Psi \colon E \to F$ is given by a collection  of maps $\Psi_{\alpha}\colon U_\alpha \to \Hom(k^n,k^m)$ such that 
\begin{equation}
\Psi_\beta = g^F_{\alpha\beta}\Psi_\alpha g^E_{\beta\alpha}.
\end{equation}
\end{exc}
\begin{expl}
\begin{enumerate}[i)]
\item A vector space is a vector bundle over a point. 
\item For every manifold $M$, the tangent bundle $TM$ is a vector bundle over $M$. The transition maps of the tangent bundle $TM$ can be computed in the following way. Let $(U_\alpha,\varphi_\alpha)$ be an atlas of $M$. Then $d\varphi_{\alpha\beta}(u)\colon U_{\alpha\beta} \to GL_n(\R)$ are the transition maps of $TM$.
\item For any manifold $M$ and natural number $n$ there is the \emph{trivial rank $n$ vector bundle over $k$}, simply given by the direct product $M \times k^n$ with the canonical projection to $M$. This bundle is often denoted $\underline{k}^n$.

\item Recall that $\C\mathbb{P}^n$ is the space of lines in $\C^n$, i.e. $\C\mathbb{P}^n = \C\mathbb{P}^{n+1}/\C^\times$, where $\C^\times$ acts diagonally. The quotient map $\pi\colon\C^{n+1}\to \C\mathbb{P}^n$ is a rank 1 complex vector bundle over $\C\mathbb{P}^n$. Working out the details of this is a marvelous exercise.  
\end{enumerate}
\end{expl}
It is an important fact that the vector bundle is entirely determined up to isomorphism by its trivializing cover and the gluing maps. 
\begin{prop}
Given a cover $\mathfrak{U} = \{U_{\alpha}\}_\alpha$ of a manifold $M$ and a family of snooth maps $g_{\alpha\beta}\colon U_{\alpha\beta} \to GL_n(k)$ satisfying \eqref{eqs:gluingmaps}, there exists a unique (up to isomorphism) vector bundle $\pi\colon E \to M$ with trivializing cover $\mathfrak{U}$ and gluing maps $g_{\alpha\beta}$. 
\end{prop}
\begin{proof}
\emph{Existence:}
We can assume that each $U_{\alpha}$ is a contained in a domain of a chart of $M$ (otherwise, cover each $U_{\alpha}$ by charts $V_{\alpha\beta}$ and note that the transition restricted to each $V_{\alpha\beta}$ still satisfy \eqref{eqs:gluingmaps}.) First, construct the fiber over $u$ by \begin{equation} E_u:= \left(\coprod_{\alpha\in A,u\in U_{\alpha}}k^n\right)/\sim = \left(\{\alpha\in A,u \in U_{\alpha}\} \times k^n\right)/\sim
\end{equation}
where $(\alpha,v) \sim (\beta,w)$ if $g_{\alpha\beta}(u)v = w$. This is an equivalence relation \emph{since $g_{\alpha\beta}$ satisfy \eqref{eqs:gluingmaps}.} Then, let $E := \coprod_{u\in M} E_u$ and $\psi_{\alpha}[(u,\alpha,v)] = (u,v).$ Since $U_\alpha$ is contained in a chart, composition with this chart yields a chart of $E$. It is easily checked that this is indeed a smooth atlas.  \\
\emph{Uniqueness:} It is enough to show that two vector bundle with the same trivializing cover and gluing maps are isomorphic. Let $E$,$F$ be such vector bundles. Then, we construct the isomorphism over $U_\alpha$ by the diagram 
\[\begin{tikzcd}
F_\alpha\colon\pi^{-1}(U_\alpha) \arrow[rd, "\pi"] \arrow[r, "\psi^E_{\alpha}"] & U_\alpha \times k^n \arrow[r, "(\psi_\alpha^F)^{-1}"] & (\pi')^{-1}(U_\alpha) \arrow[ld, "\pi'"'] \\
                                                                  & U_\alpha                                              &                                          
\end{tikzcd}\]
Since the gluing maps are the same, the maps $F_\alpha$ and $F_\beta$ agree on $U_{\alpha\beta}$: We have 

\begin{align} F_\beta &= (\psi_\beta^F)^{-1} \circ \psi_\beta^E  \\
&= ( \tilde{g}_{\alpha\beta}^F)\circ \psi_\alpha^F)^{-1} \circ (\tilde{g}_{\alpha\beta}^E\circ \psi_\alpha^E) \\
&= (\psi_{\alpha}^F)^{-1}\circ (\tilde{g}^F_{\alpha\beta})^{-1}\circ \tilde{g}_{\alpha\beta}^E \circ \psi_\alpha^E = F_\alpha, 
\end{align}
since $(\tilde{g}^F_{\alpha\beta})^{-1}\circ \tilde{g}_{\alpha\beta}^E = \mathrm{id}_{U_{\alpha} \times k^n}$.
\end{proof}

This central fact will often help us define vector bundles via trivializing covers and gluing maps. Given a manifold over $M$, we can define a rank $n$ vector bundle $E$ over $k$ by specifying a trivializing cover $\mathfrak{U}$ and transition maps $g_{\alpha\beta}\colon U_{\alpha\beta} \to GL_n(k)$, and we write $E = (\mathfrak{U},g_{\alpha\beta})$ for this vector bundle.
\begin{expl}
Consider the circle $S^1 = \R/\Z$ with the open cover $U_1=(0,1), U_2=(1/2,3/2)$. Then the M\"obius band is the vector bundle with transition function $g_{12}\colon U_{12} \to GL_1(\R) = \R^\times$ given by 
$$g_{12}(x) = \begin{cases} 1 & 
x \in (1/2,1) \\
-1 &x \in (1,3/2)
\end{cases}
$$
\end{expl}
\begin{rem}
One can show this is the only non-trivial vector bundle over $S^1$. In fact, the tangent bundle of $S^1$ is trivial $TS^1 \cong S^1 \times \R$. 
\end{rem}
\begin{defn}[Section]
A \emph{section} of a vector bundle $\pi \colon E \to M$ is a smooth map $\sigma \colon M \to E$ such that $\pi \circ \sigma = \mathrm{id}_M$. 
\end{defn}
The set of sections of $E$ is denoted $\Gamma(M,E)$ or simply $\Gamma(E)$ when no confusion is possible. Note that since $E_x$ is a vector space for all $x \in M$, we can naturally add sections and multiply them by scalars: 
\begin{equation}
(\sigma_1 + \sigma_2)(x) = \sigma_1(x) + \sigma_2(x), \qquad, (\lambda\sigma)(x) = \lambda\sigma(x).
\end{equation}
Thus, $\Gamma(E)$ is a $k$-vector space. 
\begin{expl}
\begin{enumerate}[a)]
\item A section of a trivial bundle $M \times k^n$ is given by $\sigma(x)= (x,f(x))$, where $f\colon M \to k^n$ is a smooth map. Thus, $\Gamma(M,\underline{k}^n) \cong C^\infty(M,k^n)$.
\item 
Over a trivializing cover $\mathfrak{U} = \{U_\alpha\}_{\alpha_\in A}$, a section is given by smooth functions $\sigma_\alpha\colon U_\alpha \to k^n$ satisfying 
\begin{equation}
\sigma_\beta(x) = g_{\alpha\beta}(x)\sigma_\alpha(x).
\end{equation}
\item  A section of the tangent bundle $TM$ is called a \emph{vector field}. 
\end{enumerate}
\end{expl}
The natural constructions on vector spaces, such as dualizing, direct sums and tensor products, carry over to vector bundles. Here the description in terms of transition functions comes in handy. 
\begin{defn} Let $E = (\mathfrak{U},g_{\alpha\beta})$ and $F=(\mathfrak{U},h_{\alpha\beta})$ be two vector bundles over the same trivializing cover. Then we define the following bundles:
\begin{enumerate}[i)]
\item The \emph{dual bundle} $E^*$ by
\begin{equation}
E^* = (\mathfrak{U},(g_{\alpha\beta}^*)^{-1})
\end{equation}
\item The \emph{direct sum} $E \oplus F$ by 
\begin{equation}
E \oplus F = (\mathfrak{U},g_{\alpha\beta}) \oplus h_{\alpha\beta}
\end{equation}
\item The \emph{tensor product} $E \otimes F$ by 
\begin{equation}
E \otimes F = (\mathfrak{U},g_{\alpha\beta}) \otimes h_{\alpha\beta}
\end{equation}
\item The \emph{symmetric} and \emph{exterior powers} $\mathrm{Sym}^kE$ and $\bigwedge^kE$by 
\begin{align}
\mathrm{Sym}^kE &= (\mathfrak{U},\mathrm{Sym}^kg_{\alpha\beta})\\
\bigwedge^kE &= \left(\mathfrak{U},\bigwedge^k g_{\alpha\beta}\right).
\end{align}
\item The \emph{determinant line} 
$\det E$ by 
\begin{equation}
\det E = \bigwedge^{rk(E)}E.
\end{equation}
\end{enumerate}
\end{defn}
\begin{expl}
\begin{enumerate}[a)]
\item The dual of the tangent bundle $TM$ is called the cotangent bundle and denoted $(TM)^* = T^*M$. 
\item Sections of $\bigwedge^k T^*M$ are called differential $k$-forms on $M$. 
\item If $E$ is a vector bundle, then sections of $\bigwedge^k T^*M \otimes E$ are called 
\emph{differential k-forms with values in $E$}.
\item Sections of $E^*\otimes F$ are the same as vector bundle morphisms $E \to F$: $$\Gamma(E^*\otimes F) \cong \underline{\Hom}(E,F)$$
(Exercise!)
\end{enumerate}
\end{expl}

\subsection{Principal bundles}
We begin with a definition. 
\begin{defn}[Principal bundle]
Let $G$ be a Lie group and $M$ be a manifold. A principal $G$-bundle is a triple $(P,\pi,M)$ such that
\begin{enumerate}[i)]
\item $\pi\colon P \to M$ is a smooth submersion, 
\item There is a free and transitive right action $P \times G \to P$ such that $\pi$ is $G$-invariant (that is, $\pi(pg) = \pi(p)$), 
\item There exists a \emph{trivializing cover} $\mathfrak{U} = \{U_\alpha\}_{\alpha \in A}$, that is, a cover of $M$ with the property that for every $\alpha \in A$ there exists a diffeomorphism $\Psi_{\alpha}\colon \pi^{-1}(U_\alpha) \to U_\alpha \times G$ such that 
\begin{equation}
\begin{tikzcd}
\pi^{-1}(U_\alpha) \arrow[rd, "\pi"] \arrow[rr, "\Psi_\alpha"] &          & U_\alpha \times G \arrow[ld, "\pi_1"'] \\
                                                               & U_\alpha &                                       
\end{tikzcd}
\end{equation}
commutes and $\Psi_\alpha(pg) = \Psi(p)g$. 
\end{enumerate}
\end{defn}
Again, we call $M$ the \emph{base} and $P$ the \emph{total space} of the bundle. The Lie group $G$ is called the structure group of the bundle. 
\begin{defn}[Morphism of principal bundles]
A \emph{morphism} of principal $G$-bundles $P$ and $P'$ is a smooth map $f \colon (P,\pi) \to (P',\pi')$ that commutes both with the right $G$-action and the projections, that is 
\begin{align}
f(pg) &= f(p)g \\
\pi(f(p)) &= \pi(p),
\end{align}
and an \emph{isomorphism} of principal $G$-bundles is a morphism which is also a diffeomorphism. 
\end{defn}

Again, we can define  
\begin{equation}
\tilde{g}_{\alpha\beta} =\Psi_\beta \circ \Psi_{\alpha}^{-1} \colon U_{\alpha\beta} \times G \to U_{\alpha\beta} \times G
\end{equation}
which are given by \begin{equation}
\tilde{g}_{\alpha\beta}(u,h) = (u,g_{\alpha\beta}(u)h)
\end{equation} 
since the trivializations commute with the projections. Note that the transition functions act from the left, since they commute with the right $G$-action.  The maps 
\begin{equation}
g_{\alpha\beta}\colon U_{\alpha\beta}\to G
\end{equation}
are called the transition or gluing maps.  They satisfy 
\begin{subequations}\label{eqs:gluingMapsPrincipal}
\begin{align}
g_{\alpha\alpha}(u) &= 1_G \\
g_{\alpha\beta}(u) &= g_{\beta\alpha}^{-1} \\
g_{\alpha\gamma}(u) &= g_{\beta\gamma}(u)g_{\alpha\beta}(u)
\end{align}
\end{subequations} We also have a similar proposition: 
\begin{prop}
Let $\mathfrak{U} = \{U_\alpha\}_{\alpha \in A}$ be a cover of $M$ and suppose $g_{\alpha\beta} \colon U_{\alpha\beta} \to G$ satisfy \eqref{eqs:gluingMapsPrincipal}. Then there exists a principal $G$-bundle $P$ over $M$ with trivializing cover $\mathfrak{U}$ and gluing maps $g_{\alpha\beta}$, and this bundle is unique up to isomorphim. 
\end{prop}

Even though vector bundles and principal bundles are different objects, in some sense they are like two sides of the same coin. This is explained by the following fundamental example: 
\begin{expl}
If $V$ is a vector spaces, then a \emph{frame} of $V$ is an ordered basis $\underline{e}= (e_1,\ldots,e_r)$. The set of frames is denoted by $\mathrm{Fr}(V)$. Let $\pi\colon E \to M$ be a rank $n$ $k$-vector bundle. Then, the \emph{frame bundle} $\mathrm{Fr}(E)$ of $E$ is the smooth manifold $\mathrm{Fr}(E) = \sqcup_{x \in M}\mathrm{Fr}(E_x)$. This manifold has a natural projection $\overline{\pi}\colon \mathrm{Fr(E)} \to M$. We can define a right $GL_n(k)$ action on $\mathrm{Fr}(E)$ in the following way. Let $\{U_\alpha\}_\alpha$ be a trivializing cover for $E$. Over a trivializing chart $\psi_\alpha \colon \pi^{-1}(U_\alpha) \to U_\alpha \times k^n$, the right action is given by \begin{equation}\underline{e}\cdot g = \psi_\alpha^{-1}(g^{-1}\psi_\alpha(e_1),\ldots,g^{-1}\psi(e_n)).\end{equation}
 This gives $\mathrm{Fr}(E)$ the structure of a principal $GL_n(k)$-bundle. If $\psi_\alpha = (\pi, A_\alpha$, then a trivialization of $\mathrm{Fr}(E)$ is given by 
 \begin{align*}
 \overline{\psi}_\alpha(\underline{e} = (\overline{pi}(\underline{e}, A_\alpha(\underline{e}).
 \end{align*}
 Here $A_\alpha(\underline{e}) = (A_\alpha e_1,\ldots A_\alpha e_n) \in GL_n(k)$
 Now, one can check that this principal $GL_n(k)$-bundle has \emph{the same} gluing maps $g_{\alpha\beta} = A_\beta A_\alpha^{-1}$: 
 $${A}_\beta(\underline{e}) = A_\beta (A_\alpha A_\alpha^{-1})(\underline{e}) = g_{\alpha\beta}A_\alpha(\underline{e}),$$
 hence $g_{\alpha\beta}$ are the transition functions of $P$ since they satisfy the defining equation
 
$$ \overline{\Psi}_\beta \circ \overline{\Psi}_\alpha(u,h) = (u, g_{\alpha\beta}h).$$
\end{expl}
This fact is important so we record it again: 
\par
\begin{center}
\fbox{\parbox{0.7\textwidth}{Let $E=(\mathfrak{U},g_{\alpha\beta})$ be a vector bundle. Then, its frame bundle is the principal  $GL_n(k)$-bundle $P=(\mathfrak{U},g_{\alpha\beta})$.}}
\end{center}
Thus, we can use the same data to define either  vector bundles or principal $GL_n(k)$ bundles. It is in this sense that we mean they are two sides of the same coin. However, we can construct vector bundles also from principal bundles with other structure groups.
\begin{defn}
Given a principal $G$-bundle $P = (\mathfrak{U},g_{\alpha\beta})$, and a representation $\rho \to GL_n(k)$, we define the \emph{associated vector bundle} $P \times_\rho k^n$ by 
\begin{equation}
E = (\mathfrak{U},\rho(g_{\alpha\beta}))
\end{equation}
\end{defn}
We say that a vector bundle $E$ has a $(G,\rho)$-structure if $E = P\times_\rho k^n$ for a principal $G$-bundle $P$. We denote this bundle by $\mathrm{Ad P}$. 
\begin{expl}
If $G \subset GL_n(k)$ is a subgroup, it has the trivial representation $\iota \colon G \hookrightarrow GL_n(k)$. We say that $E$ has a $G$-structure if it has a $(G,\iota)$ structure. This means that we can find $\mathfrak{U}$ and $g_{\alpha\beta}$ such that $E= (\mathfrak{U},g_{\alpha\beta})$, where the $g_{\alpha\beta}$ take values in $G \subset GL_n(k)$. For example, an orientation of $M$ is the same as an $SL_n(k)$-structure on $TM$.
\end{expl}
\begin{expl}[Adjoint bundle]
Let $G$ be a Lie group and $\mathfrak{g}$ be its Lie algebra. Then $G$ acts on $\mathfrak{g}$ via the adjoint action (if $G$ is a matrix group then this action is given by conjugation $g \cdot X = g X g^{-1}$). Hence, for every principal $G$-bundle $P$ we have the \emph{adjoint bundle} $P \times_{\rho}\mathfrak{g}$. 
\end{expl}
\subsection{Connections on vector bundles}
Very roughly, a \emph{connection} on a fiber bundle is a consistent way to move from one fiber in the bundle to the other. The concept of connection exists over both vector bundles and principal bundles.  We start with the concept of a connection on vector bundles. 

\begin{defn}[Connection on vector bundle]
Let $\pi\colon E \to M$. Then a \emph{connection on $E$} is a linear map
\[\nabla \colon \Gamma(E) \to \Gamma(T^*M \otimes E)\]
such that, for $f \in C^\infty(M)$ and $\sigma \in \Gamma(E)$,  the Leibniz rule
\begin{equation}
\nabla(f\sigma) = df \otimes \sigma + f \nabla \sigma
\end{equation}
\end{defn}
The connection $\nabla$ induces a \emph{covariant derivative} along vector fields on sections on $E$.
\begin{defn}[Covariant derivative]
Let $\nabla$ be a connection on the vector bundle $E$ over $M$. Let $X$ be a vector field on $M$. Then, the map $\nabla_X \colon \Gamma(E) \to \Gamma(E)$ given by 
\begin{equation}
\nabla_X\sigma = \iota_X\nabla\sigma
\end{equation}
is called the \emph{covariant derivative of $\sigma$ along $X$ (with respect to $\nabla$)}.
\end{defn}
The Leibniz rule for the covariant derivative is 
\begin{equation}
\nabla_X(f\sigma) = (L_Xf)\sigma + f \nabla_X\sigma. 
\end{equation}
\begin{expl}
On the trivial bundle $M \times k^n$ we have a connection given by the de Rham differential $(f_1,\ldots f_n) \mapsto (df_1, \ldots df_n)$. This connection is called the \emph{trivial connection}.  
\end{expl}
\begin{prop}
If it is not empty, the space of connections $\mathcal{A}_E$ is an affine space modeled on the vector space $\Gamma(M,T^*M\otimes E^* \otimes E) = \Gamma(M, T^*M \otimes \End E) =  \Omega^1(\End E)$.
\end{prop}
\begin{proof}
First observe that if $\nabla^1$ and $\nabla^2$ are connections on $E$, then their difference $A = \nabla^1 - \nabla^0$ satisfies $A(f\sigma) = fA(\sigma)$. Hence $A$ defines a vector bundle morphism $E \to T^*M \otimes E$. It follows that $A \in \Gamma(\Hom(E,T^*M \otimes E)) \cong \Omega^1(\End(E))$.
\end{proof}
In particular, every connection on the trivial bundle $M \times k^n$ is of the form 
$\nabla = d + A$, where $A \in \Omega^1(\End(k^n))$. In a basis of $k^n$ we write\footnote{We follow the Einstein summation convention that repeated indices are summed over. This does not apply to indices labeling covers (usually $\alpha,\beta,\gamma$).} $A(e_j) = A^i_je_i$, then, 
$$\nabla(f_1,\ldots, f_n) = (df_1,\ldots df_n) + (A^i_1f_i,\ldots A^i_nf_i).$$ \ Thus one can think of $A$ as a 1-form with values in matrices, or, equivalenty, as a matrix of 1-forms $A^i_j$. Both viewpoints are sometimes helpful. \par
If $\Psi\colon E \to F$ is an isomorphism of vector bundles, and $\nabla$ is a connection on $E$, then the map $\check{\nabla} = (\mathrm{id} \otimes \Psi)\circ \nabla \circ \Psi^{-1}$ is a connection on $F$. It is the unique map that makes 
\begin{equation}
\begin{tikzcd}
\Gamma(E) \arrow[rr, "\nabla"]                                       &  & \Gamma(T^*M \otimes E) \arrow[d, "\mathrm{id}\otimes\Psi"] \\
\Gamma(F) \arrow[u, "\Psi^{-1}"'] \arrow[rr, "\check\nabla", dashed] &  & \Gamma(T^*M\otimes F)                                     
\end{tikzcd}
\end{equation}
commute. In particular, consider a local trivialization $\Psi_\alpha\colon \restr{E}{U_\alpha} \to U_\alpha\times k^n$ of $E$. Then a connection $\nabla$ on $E$ induces a connection on $U_\alpha \times k^n$, hence an element $A_\alpha$ of $\Omega^1(U_\alpha,\End k^n)$ and we have 
\begin{equation}\label{eq:LocalConnAction}
(\nabla\sigma)_\alpha = d\sigma_\alpha + A_\alpha\sigma_\alpha
\end{equation} (in the second term there is matrix-vector multiplication). This is called the \emph{connection 1-form} of $\nabla$ in $U_\alpha$. If $U_\beta$ is another local trivialization, one can ask how $A_\alpha$ and $A_\beta$ are related. 
\begin{prop}
Let $\nabla$ be a connection on the vector bundle $E$ and $A_\alpha,A_\beta$ be the connecction 1-forms on two local trivializations $U_\alpha$, $U_\beta$. Then 
\begin{equation}\label{eq:ConnTransformVector}
A_\beta = g_{\alpha\beta}A_{\alpha}g_{\alpha\beta}^{-1} - (dg_{\alpha\beta})g_{\alpha\beta}^{-1}.
\end{equation}
\end{prop}
\begin{proof}
Let $\sigma \in \Gamma(E)$. Then, we know that 
\begin{equation}\label{eq:ProofConnTransform1}
\sigma_\beta = g_{\alpha\beta}\sigma_\alpha
\end{equation}
and 
\begin{equation}\label{eq:ProofConnTransform2}
(\nabla\sigma)_\beta = g_{\alpha\beta}(\nabla\sigma)_\alpha.
\end{equation}
Expanding \eqref{eq:ProofConnTransform2} using \eqref{eq:LocalConnAction}, we obtain 
\begin{align*}
d\sigma_\beta + A_\beta\sigma_\beta &= g_{\alpha\beta}(d\sigma_\alpha + A_\alpha \sigma_\alpha) 
\end{align*}
On the other hand, using \eqref{eq:ProofConnTransform1} we obtain 
\[(\nabla\sigma)_\beta = d\sigma_\beta + A_\beta\sigma_\beta = d(g_{\alpha\beta}\sigma_\alpha) + A_\beta g_{\alpha\beta}\sigma_\alpha  = dg_{\alpha\beta}\sigma_{\alpha} + g_{\alpha\beta}d\sigma_\alpha + A_\beta g_{\alpha\beta}\sigma_{\alpha}.\]
We conclude that 
$$A_\beta g_{\alpha\beta} \sigma_\alpha = g_{\alpha\beta} A_{\alpha} \sigma_\alpha - dg_{\alpha\beta}s_\alpha.$$
Since this holds for all $s$, we see that 
$$ A_\beta g_{\alpha\beta} = g_{\alpha\beta}A_\alpha - dg_{\alpha\beta}$$
from where the claim follows.
\end{proof}
Hence, we can characterize a connection on a bundle $E = (\mathfrak{U},g_{\alpha\beta})$ as a collection of 1-forms $A_\alpha \in \Omega^1(U_\alpha,\End E)$. Notice that $\End E  = \End( k^n)$ is the Lie algebra of $GL_n(k)$.  This suggest a natural generalization of the concept of connections to principal bundles, as discussed in the next subsection. 
\subsection{Connections on principal bundles}
We start with some definitions. Suppose $G$ is a Lie group with Lie algebra $\mathfrak{g}$ that acts on a manifold $P$ from the right. For fixed $p \in P$, there is a map 
\begin{align*}
\mu_p \colon G &\to P \\
 g &\mapsto pg.
\end{align*}
The differential of this map at the identity element $e\in G$ is map
\begin{align*}
(d\mu_p)_e\colon \mathfrak{g} \cong T_eG &\to T_pP \\
X \mapsto (d\mu_p)_eX
\end{align*}

\begin{defn}[Fundamental vector field]
 Let the Lie group $G$ act on the manifold $P$ from the right and let $X \in \mathfrak{g}$. Then, the fundamental vector field $X^\sharp$ on $P$ is the section of $TP$ defined by 
 \begin{equation}
 X^\sharp_p = (d\mu_p)_eX.
 \end{equation}
\end{defn}
In particular, $G$ acts on itself by right multiplication. For fixed $g \in G$, this action reads 
\begin{align*}
\mu_g \colon G &\to G \\
h \mapsto gh
\end{align*}
Hence we have $\mu_g = L_g$ (left multiplication by $g$). Let $X \in \mathfrak{X}$. The fundamental vector field of the right action of $G$ on itself is given by $X \mapsto (dL_g)_eX$. 
\begin{defn}[Maurer-Cartan Form]
The Maurer-Cartan Form $\phi \in \Omega^1(G,\mathfrak{g})$ is defined by $\phi_g(X^\sharp_g) \equiv X \in \mathfrak{g}$, where $X^\sharp$ is the fundamental vector field of the right action of $G$ on itself.
\end{defn}
\begin{rem}
From the discussion above it follows that 
\begin{equation}
\phi_g = (dL_{g^{-1}})_g \colon T_gG \to T_eG \cong \mathfrak{g}.
\end{equation}
In particular, for matrix groups it is given by $\phi_g = g^{-1}dg$. 
\end{rem}
We can now define a connection on a principal bundle.
\begin{defn}[Connection on a principal bundle]
Let $\pi\colon P \to M$ be a principal $G$-bundle, and let $\mathfrak{g}$ be the Lie algebra of $G$. A \emph{connection on $P$} is a 
1-form $\Omega \in \Omega^1(P,\mathfrak{g})$ satisfying\footnote{Here, $\mathrm{Ad}_g \colon \mathfrak{g} \to \mathfrak{g}$ is given by differentiating the map $\mathrm{Ad}_g \colon  G \to G$ at the identity. For matrix groups $G \subset GL_n(k)$ we have $\mathfrak{g} \subset gl_n(k)$ and the adjoint action is $\mathrm{Ad}_{g}X = gXg^{-1}$. }
\begin{enumerate}[i)]
\item For all $g \in G$, 
\begin{equation}
R_g^*\Omega = \mathrm{Ad}_{g^{-1}}\Omega\label{eq:defConn1}\end{equation}
\item For all $X \in \mathfrak{g}$, 
\begin{equation}
\Omega(X^\sharp) = X \in \mathfrak{g} \label{eq:defConn2}
\end{equation}
\end{enumerate}
\end{defn}
Notice that here, the 1-form is on the total space $P$. Recall that a local section $\sigma\colon U \to \restr{P}{U}$ defines a local trivialization $\pi^{-1}(U) \to U \times G$ of $P$ via 
\begin{equation}
p \mapsto (\pi(p),\sigma(\pi(p))).
\end{equation}
Let $\sigma'$ be another section over $U$. Then there exists a map $g\colon U \to G$ such that $\sigma'(x) = \sigma(x)g(x)$. The following lemma describes the behaviour of a connection under such a change of trivialization.
\begin{lem}\label{lem:connTrivChange}
If $\sigma,\sigma'$ are as above, then 
\begin{equation}
(\sigma')^*\Omega = \mathrm{Ad}_{g^{-1}}\sigma^*\Omega + g^*\phi, 
\end{equation}
where $\phi \in \Omega^1(G, \mathfrak{g})$ is the Maurer-Cartan Form introduced above.
\end{lem}
\begin{proof}
Let $x \in U$, and $v\in T_xM$. Then $((\sigma')*\Omega)_xv = \Omega_{\sigma'(x)}d\sigma'_x v$. On the other hand we can write $\sigma'$ as the composition 
\[\begin{tikzcd}
U \arrow[r,"{(s,g)}"]           & P\times G \arrow[r, "{\mu}"]              & P        \\
x  \arrow[r, maps to] & {(s(x),g(x))} \arrow[r, maps to] & s(x)g(x)
\end{tikzcd} \]
We first compute the pullback $\mu^*\Omega$ to $P\times G$. For this, note that $$(d\mu_p)_gw = (d\mu_{pg})_e(dL_{g^{-1}}w$$ 
(this follows from the chain rule). Then, we have 
\begin{align*}
(\mu^*\Omega)_{(p,g)}(v,w) &= \Omega_{pg}(d\mu_p)_g w + \Omega_{pg}(dR_g)_pv \\
&= \Omega_{pg}\underbrace{(d\mu_{pg})_e(dL_{g^{-1}})_gw}_{((dL_{g^{-1}})_gw)^\sharp} + ( R_g^*\Omega)_p v \\
&= (dL_{g^{-1}})_gw + \mathrm{Ad}_{g^{-1}}\Omega_p = \phi_gw + \mathrm{Ad}_{g^{-1}}
\end{align*} 
where in the last equality we used the two properties of a connection. Now, the first term is exactly $\phi_g$. Pulling back to $U$ with $(s,g)$, we obtain the result since pulling back with $\sigma$ commutes with the adjoint action of $g$, which acts only on the Lie algebra factor of $\Omega$.

\end{proof}
The next proposition establishes the relationship of this definition with the one of a connection on a vector bundle.  
\begin{prop}
Let $P=(\mathfrak{U},g_{\alpha\beta}) $be a principal bundle with structure group $G\subset GL_n(k)$, i.e. $G$ is a matrix group\footnote{Notice that the spin group is also a matrix group since the Clifford algebra is isomorphic to a matrix algebra. However, $\Spin_n \nsubseteq Gl_n(k)$, rather, $\Spin_n \subset GL_{2^n}(k)$.}. Then a connection on $P$ is equivalent to a collection of 1-forms $A_\alpha\in \Omega^1(M,\mathfrak{g})$ such that 
\begin{equation}
A_\beta = g_{\alpha\beta}A_\alpha g^{-1}_{\alpha\beta} - dg_{\alpha\beta}g^{-1}_{\alpha\beta}.
\end{equation}
\end{prop}
\begin{proof}
Suppose we are given a connection $\Omega$ on $P$ and let $U_\alpha \in \mathfrak{U}$. Consider the constant section $\sigma_\alpha \colon U_\alpha \to  U_\alpha \times G, u\mapsto(u,1).$ Then, we define $A_\alpha := (\Psi_\alpha^{-1}\circ \sigma_{\alpha})^*\Omega$. Now, notice that the section $\sigma_\alpha$ over the $U_\beta$ is given by $g_{\alpha\beta}$. Hence 
$$\sigma_\beta = \sigma_\alpha g_{\alpha\beta}^{-1}.$$
Now we can apply proposition \ref{lem:connTrivChange} for $g^{-1}$.
 Now, notice that we have 
 $0 = d(gg^{-1}) = dg g^{-1} + gd(g^{-1})$ and hence $d(g^{-1}) = - g^{-1}dg g^{-1}$. This implies that for a matrix group, we have $\phi_{g^{-1}} = gd(g^{-1}) = - dg g^{-1}$. This proves that a connection $\Omega$ is described by such 1-forms in a local trivialization. \\
Conversely, assume that we are given a family of such 1-forms. Then we set $\Omega_\alpha(u,g) := \mathrm{Ad}_{g^{-1}}A_\alpha + \phi_g$ on $U_{\alpha} \times G$ and define $\Omega$ on $\pi^{-1}(U_{\alpha})$ as $\Psi_{\alpha}^*\Omega_{\alpha}$. We then glue together the connection using a partition of unity. The resulting 1-form $\Omega$ which is a connection since the local pieces are, and the conditions \eqref{eq:defConn1} and \eqref{eq:defConn2} are convex. 
\end{proof} 
We have the following corollary: 
\begin{cor}
A connection $\nabla$ on a vector bundle $E$ induces a connection $\Omega$ on the bundle of frames $\mathrm{Fr}(E)$ and vice versa. 
\end{cor}
Hence, one can study connections on vector bundles by studying connections on principal bundles. This will be our approach in this course. 
\subsubsection{Curvature}
An important notion associated to a connection is the concept of \emph{curvature}. 
For this, we need the concept of Lie Bracket on Lie algebra-valued forms, which is defined on elements of the form $A=\alpha \otimes \xi, B=\beta \otimes \xi'$, where $\alpha\in\Omega^k(M),\beta\in\Omega^l(M),\xi,\xi'\in\mathfrak{g}$, by
\begin{align}
[\cdot,\cdot]\colon\Omega^k(M,\mathfrak{g}) \times \Omega^l(M,\mathfrak{g}) \to \Omega^{k+l}(M,\mathfrak{g}) \notag \\
[\alpha\otimes \xi, \beta \otimes \xi'] := \alpha \wedge \beta \otimes [\xi,\xi']
\end{align}
and extended bilinearly. In particular, for matrix groups we have $[\xi,\xi'] = \xi\xi' -\xi'\xi$ and then 
\begin{equation}
[A,B] = A \wedge B - (-1)^{|A||B|} B\wedge A = -(-1)^{|A||B|}[B,A]
\end{equation}
where the wedge product operation is defined by matrix multiplication: If $A=\alpha \otimes \xi, B=\beta \otimes \xi'$ as above, then
\begin{equation}
A \wedge B = \alpha \wedge \beta \otimes \xi\xi'
\end{equation}
The Lie bracket satisfies 
\begin{align}
d[A,B] &= [dA,B] + (-1)^{|A|}[A,dB] \label{eq:LieBracketLeibniz} \\ 
[A,[B,C]] &= [[A,B],C] + (-1)^{|A||B|}[B,[A,C]]\label{eq:GradedJacobi}
\end{align}
Now, the curvature is easily defined from the abstract viewpoint on connections:
\begin{defn}[Curvature]
Let $\Omega \in \Omega^1(P,\mathfrak{g})$ be a connection on a principal $G$-bundle $\pi\colon P \to M$. Then, the \emph{curvature} of $\Omega$ is the 2-form $F \in \Omega^2(P,\mathfrak{g})$ defined by 
\begin{equation}
\mathbf{F} = d\Omega + \frac{1}{2}[\Omega,\Omega]\label{eq:defCurv1}
\end{equation}
\end{defn}
We summarize some properties of the curvature as exercises. 
\begin{exc}
Let $\mathfrak{U}$ be a local trivialization of $P$. Denote $F_{\alpha} := (s_\alpha)^*\mathbf{F} = dA_\alpha + \frac12 [A_\alpha,A_\alpha].$
Then 
\begin{equation}
F_\beta = g_{\alpha\beta}F_\alpha g_{\alpha\beta}^{-1}
\end{equation}
\end{exc}
It follows that the $F_\alpha$ define a section $F \in \Omega^2(M,\mathrm{Ad} P)$. 
\begin{exc}
Let $E \to M$ be a vector bundle and let $\nabla$ be a connection on $E$. Define the two-form $F^\nabla \in \Omega^2(M,\End E)$ by 
\begin{equation}
F^\nabla(X,Y) = \nabla_X \nabla_Y - \nabla_Y\nabla_X - \nabla_{[X,Y]}.
\end{equation}
Show that this is the curvature 2-form of the associated connection on $\mathrm{Fr}(E)$. \\
\emph{Hint:} Work over a trivializing chart and remember the formula for the de Rham differential of a 1-form: 
$$d\omega(X,Y) = X\omega(Y) - Y\omega(X) - \omega([X,Y]).$$
\end{exc}
\subsubsection{Exterior Derivative}
A connection on a principal bundle $P \to M$ induces an \emph{exterior derivative} on $\mathrm{ad} P$-valued differential forms. In a trivializing chart $U_\alpha$, it is defined by 
\begin{equation}
(d_\Omega\omega)_\alpha = d\omega_\alpha + [A_\alpha,\omega_\alpha]\label{eq:defExtDerivative}
\end{equation}
\begin{prop}
\begin{enumerate}[i)]
\item The exterior derivative in local trivializations by \eqref{eq:defExtDerivative} defines a map 
\begin{align*}
d_\Omega \colon \Omega^k(M,\mathrm{ad} P) &\to \Omega^{k+1}(M,\mathrm{ad} P) \\
\omega \mapsto d_\Omega \omega
\end{align*}
\item We have 
\begin{equation}
d_\Omega d_\Omega \omega = [F_\Omega,\omega]
\end{equation}
\item The curvature satisfies 
\begin{equation}
d_\Omega F_\Omega = 0,
\end{equation}
the Bianchi identity. 
\end{enumerate}
\end{prop}
\begin{proof}
\begin{enumerate}[i)]
\item One simply checks by direct computation that $$(d_\Omega\omega)_\beta = g_{\alpha\beta}(d_\Omega\omega)_\alpha g_{\alpha\beta}^{-1}.$$
\item By the first point, it is enough to check this in a trivializing chart. Here, again the proof is a simple computation: 
\begin{align*}
(d_\Omega d_\Omega \omega)_\alpha &= d_\Omega (d\omega_\alpha + [A_\alpha,\omega_\alpha]) \\
&= d(d\omega_\alpha) + [A_\alpha,d\omega_\alpha] + d[A_\alpha,\omega_\alpha] + [A_\alpha,[A_\alpha,\omega_\alpha]]\\
&= [A_\alpha,d\omega_\alpha] + [dA_\alpha,\omega_\alpha] - [A_\alpha,d\omega_\alpha] + \frac12[[A_\alpha,A_\alpha],\omega_\alpha]\\
&= [F_\alpha,\omega_\alpha]
\end{align*}
where we have used \eqref{eq:LieBracketLeibniz} and \eqref{eq:GradedJacobi}.

\item Again one can check this in a trivializing chart. Here we simply compute 
\begin{align*}
(d_\Omega F)_\alpha &= dF_\alpha + [A_\alpha,F_\alpha]  \\
&= d(dA_\alpha) + \frac12 d[A_\alpha,A_\alpha] + [A_\alpha,dA_\alpha]+\frac12[A_\alpha,[A_\alpha,A_\alpha]]
\end{align*}
The last term vanishes due to \eqref{eq:GradedJacobi} and the other terms cancel due to \eqref{eq:LieBracketLeibniz}.
\end{enumerate}
\end{proof}
\subsection{Metrics and metric compatibility}
Vector bundles a priori have structure group $GL_n(k)$, i.e. the transition functions take values in $GL_n(k)$. An important question is, given a vector bundle $\pi \colon E \to M$, when (and how) one can choose a trivializing cover $(\{(U_\alpha,\psi_\alpha)\}$ such that the associated transition functions take values in a subgroup $G \subset GL_n(k)$ (such a choice is called a \emph{reduction of the structure group to $G$}. One possibility to achieve this is via metrics on vector bundles: They allow to reduce the structure group to $O(p,q)$, or $U(n)$ in the complex case. 
\subsubsection{Metrics}
\begin{defn}
Let $\pi \colon E \to M$ be a real vector bundle. A \emph{metric} $g$ is  a  a smooth family of symmetric, bilinear, non-degenerate maps $g_x\colon E_x \times E_x \to \R$ for $x \in M$. Put differently, a metric $g$ on $E$ is a section of $\mathrm{Sym}^2E^*$ that is non-degenerate at every point. 
\end{defn}
Observe that the signature of $g_x$ is the same for all $x \in M$. A \emph{euclidean vector bundle} is a vector bundle with a positive definite metric. 
\begin{defn}
Let $\pi \colon E \to M$ be a complex vector bundle. A \emph{hermitian metric} $h$ on $E$ is a smooth family of hermitian, sesquilinear, non-degenerate maps $h_x \colon E_x \times E_x \to \C$. Equivalently, $h$ is  a section of $E^* \otimes \overline{E}^*$ that is hermitian and non-degenerate. 
\end{defn}
It is an easy but important observation that choosing a metric on a vector bundle is equivalent to a reduction of the structure group to $O(p,q)$ (in the real case with signature $(p,q)$) or $U(n)$ (in the complex case). We will call a \emph{metric vector bundle} a vector bundle over $\R$ or $\C$ equipped with a real metric or a Hermitian metric. 
\begin{prop}\label{prop:isometriccharts}
Let $\pi \colon E \to M$ be a metric vector bundle. Then there exists a trivializing cover $\{U_\alpha, \psi_\alpha\}$ such that the linear maps 
\[ \restr{\psi_\alpha}{E_x} \colon \pi^{-1}(\{x\}) \to \{x\} \times k^n \]
are isometries for every $x \in M$, where $k^n$ carries the standard metric (of signature $(p,q)$ or Hermitian in the complex case).
\end{prop}
We will call these trivializing covers \emph{metric}.
\begin{proof}
Take any trivializing cover $U_\alpha,\psi_\alpha$. For every $x \in U_\alpha$, $g$ induces a metric $g_\alpha(x)$ on $k_n$. Now let $F_\alpha(x) \colon k^n \to k^n$ be the linear isomorphism that trivializes $g_\alpha(x)$ and define $\psi'_\alpha = F_\alpha \circ g_\alpha$. 
\end{proof}
The transition functions of such a trivializing cover take values in the subgroup of $GL_n(k)$ preserving the standard metric. Hence a metric provides a reduction of the structure group to $O(p,q)$ or $U(n)$ respectively. The principal $O(p,q)$ (resp. $U(n)$) bundle defined by a choice of metric on $E$ is the bundle of orthonormal frames\footnote{For mixed signature, ``orthonormal'' means that the vectors in the frame are orthogonal and normalized so that $g(e_i,e_i) = \pm 1$.} $\mathrm{oFr}(E)$. Similarly to the frame bundle, this is the set of all orthonormal frames of $E$: 
\begin{equation}
\mathrm{oFr(E)} = \coprod_{x \in M}\{\underline{e}=(e_1,\ldots,e_r)| g(e_i,e_j) = \pm \delta_{ij}\}.
\end{equation}
Here the sign is determined by the signature of the metric. Equivalently, the frame bundle has fiber over $x$ given by isometries $\underline{e} \colon k^n \to E_x$, where $k^n$ has the standard metric.  \\
The following is an easy but important exercise.
\begin{exc}
Let $E \to M$ be a real or complex bundle. Then, there exist metric of every signature on $E$ (or hermitian metrics in the complex case). 
\emph{Hint: Use a partition of unity subordinate to a trivializing cover.}
\end{exc}
An important special case is the tangent bundle $TM$. 
\begin{defn}
A \emph{Riemannian metric} on $M$ is a Euclidean metric on $TM$. A \emph{pseudo-Riemannian metric} on $M$ is a metric of indefinite signature on $TM$. In particular, a \emph{Lorentzian metric} on $M$ is a metric on $TM$ of signature $(1,n-1)$. 
\end{defn}
\begin{rem}
On the tangent bundle there are particular local trivializations given by coordinate neighbourhoods. Notice that for generic metrics on $TM$, these do not give isometries between $T_xM$ and $k^n$ as in Proposition \ref{prop:isometriccharts}. In fact, if such coordinates exist around every point, the corresponding metric is called \emph{flat}. 
\end{rem}
\subsubsection{Metric compatible connections}
Now we know what metrics and connections on vector bundles are. A natural question is what the relation between these two concepts is. 
\begin{defn}
Let $\pi\colon E \to M$ be a vector bundle with metric $g$. A connection $\nabla$ on $E$ is called \emph{metric} (or \emph{metric compatible}) if, for all vector fields $X$ on $M$ and sections $\sigma, \tau \in \Gamma(E)$, we have 
\begin{equation}
L_Xg(\sigma,\tau) = g(\nabla_X\sigma,\tau) + g(\sigma,\nabla_X\tau).
\end{equation}
\end{defn}
It follows that the connection 1-forms of $\nabla$ take values in anti-symmetric (resp. anti-hermitian) endomorphisms of $E$. In particular, working over a metric trivializing cover, we see that $\nabla$  defines a connection on the orthonormal frame bundle $\mathrm{oFr}(E)$. Conversely, a connection on the orthonormal frame bundle defines a metric connection on $E$. In particular, we conclude that metric covers exist. 
\subsubsection{Digression: Torsion and the Levi-Civita Connection}
In Riemannian geometry, the existence of the Levi-Civita connection is of central importance. It is more easily formulated in terms of connections on vector bundles. To formulate it one needs the concept of torsion, for which one needs the following two remarks. 
\begin{rem}
Let $\pi \colon E \to M$ be a vector bundle with connection $\nabla$, and consider differential forms with values in $E$: 
$$\Omega^\bullet(M,E) = \Gamma(\wedge^\bullet(T^*M) \otimes E).$$
Then the connection induces an exterior covariant derivative $d^\nabla \colon \Omega^k(M,E) \to \Omega^{k+1}(M,E)$. Over a local trivialization $U$ where $\nabla = d + A$, with $A \in \Omega^1(U,\End(k^n))$, it acts on a form $\tau = \omega \otimes \sigma$, where $\omega\in \Omega^k(U)$ and $\sigma \in \Gamma(U,U\times k^n)$, by 
\begin{equation}
d^\nabla\tau = d\omega\otimes \sigma + \omega \otimes A\sigma.
\end{equation}
One can easily check that this local definition gives rise to a globally defined map. 
\end{rem}
The main difference to the exterior derivative of Lie-algebra valued forms is that here one uses matrix-vector multiplication (instead of matrix multiplication). 
\begin{rem}
There is a canonical one form $\theta \in \Omega^1(M,TM)$ given by 
\begin{equation}
\mathrm{id} \in \Hom(TM,TM) \cong \Gamma(T^*M \otimes TM) \cong \Omega^1(M,TM) \ni \theta.
\end{equation} In local coordinates $(x,v)$ on $TM$ it is given by $\theta = \sum_i v_idx^i$. 
\end{rem}
\begin{defn}
Let $\nabla$ be a connection on $TM$. Then the \emph{torsion} of $\nabla$ is defined by \begin{equation} T^\nabla = d^\nabla\theta \in \Omega^2(M,TM). 
\end{equation}
\end{defn}
\begin{exc}
Show that 
\begin{equation}
T(X,Y) = \nabla_X Y - \nabla_Y X - [X,Y] \in \Gamma(TM)
\end{equation}
\end{exc}
The following theorem is of central importance in Riemannian geometry:
\begin{thm} 
Let $(M,g)$ be a Riemannian manifold. Then there exists a unique metric connection $\nabla$ on $TM$ such that $T^\nabla = 0$. 
\end{thm}
For a proof see e.g. \cite{Carmo1992}. This connection is called the \emph{Levi-Civita connection} of $(M,g)$.
\section{Chern-Weil theory}
The idea of Chern-Weil theory is to produce, from bundles equipped with connections, cohomology classes on the base of the bundle that are independent of the connection (and hence depend only on the bundle itself). Let us present the rough idea of this construction. \\
Let $G \subset GL_n(k)$ and consider a priniciple $G$-bundle with connection $P = (\mathfrak{U},g_{\alpha\beta}, A_\alpha)$. Then we have that the curvature $F \in \Omega^2(M,\mathrm{Ad} P)$, i.e. it satisfies $F_\beta  = g_{\alpha\beta}F_\alpha g_{\alpha\beta}^{-1}$. It follows that the 2-form $\mathrm{tr}F_{\alpha\beta}$, defined by taking the trace of the Lie algebra component of $F$, defines a global 2-form on $M$, since the trace is invariant under conjugation. The Bianchi identity implies 
\begin{equation}
d\mathrm{tr}(F_\alpha) = \mathrm{tr}dF_\alpha = \mathrm{tr}[F_\alpha,A_\alpha] = 0,
\end{equation}
where the last equality uses that the trace vanishes on commutators\footnote{This is an ``honest'' commutator since $F$ has degree 2.}. 
Hence $\mathrm{tr} F$ defines a cohomology class $\mathrm{tr} F \in H^2(M)$. We will see below that this class is \emph{independent} of the choice of connection on $P$, hence it is a ``characterisitic'' class of the bundle $P$\footnote{Even though we did not empasize this, isomorphisms of bundles also act on connections (via pullback of forms) and so isomorphic bundles give rise to the same class. }. Let us explain how the general construction works. 
\subsection{$\mathrm{Ad}$-invariant polynomials}
Roughly, an $\mathrm{Ad}$-invariant polynomial on a Lie algebra $\g$ is a polynomial invariant under the adjoint action of the Lie group $G$ on $\g$, such as the trace on a matrix algebra. We will only work with matrix groups, but everything we discuss can be generalized to arbitrary Lie groups.
\begin{defn} 
Let $\g \subset gl_n(k)$ be the matrix Lie of a Lie group $G \subset GL_n(k)$. A degree $l$ $\mathrm{Ad}$-invariant polynomial $P$ is a multilinear map $P\colon \g^\otimes l \to \C$ such that for all $X_1,\ldots,X_l \in \g$, $g\in G$, and permutations $\sigma \in S_l$ we have 
\begin{equation}
P(X_{\sigma(1)},\ldots,X_{\sigma(n)}) = P(gX_1g^{-1},\ldots, gX_ng^{-1}) = P(X_1,\ldots,X_n) = P(X_1,\ldots,X_n).
\end{equation}
\end{defn}
We also define
\begin{equation}
I_l(g) = \{ \mathrm{Ad}\text{-invariant polynomials on } \g\}.
\end{equation}
\subsection{Chern-Weil theorem}
We introduce the following notation: If $P\in I_l(\g)$, $F_1, \ldots, F_l \in \Omega^\bullet(U,\g)$ and $F_{i} = \sum (F_i)_j \otimes \xi_i^j$, with $\xi_i^j \in \g$, then 
$$P(F_1, \ldots, F_l) =  \sum_{i_1,\ldots, i_l} (F_1)_{i_1} \wedge \ldots \wedge (F_l)_{i_l}\cdot P(\xi_1^{i_1},\ldots,\xi_l^{i_l}).$$
In particular, for $F \in \Omega^2(U,\g)$, we have 
\begin{equation}
P(F) = P(F,\ldots,F) = \sum_{i_1,\ldots,i_l}F_{i_1}\wedge \ldots \wedge F_{i_l}P(\xi^{i_1},\ldots, \xi^{i_l}). 
\end{equation}
The main theorem of this chapter is the following. 
\begin{thm}[Chern-Weil]
Let $(\mathfrak{U},g_{\alpha\beta},A_\alpha)$ be a principle bundle with connection and let $P \in I_l(\g)$. Then 
\begin{enumerate}[i)]
\item The collection $P(F_\alpha)$ defines a global $2l$-form $P(F) \in \Omega^{2l}(M)$, 
\item $dP(F) = 0$, 
\item If $A_0$ and $A_1$ are connections on $P$ then $P(F(A_1)) - P(F(A_0)) = d \beta$, for some $\beta \in \Omega^{2l-1}(M)$. 
\end{enumerate}
\end{thm}
\begin{proof}
\begin{enumerate}[i)]
\item This is immediate from the $\mathrm{Ad}$-invariance of $P$. Indeed, let $F_\alpha = \sum_i (F_\alpha)_i\otimes\xi^i$, then we have 
\begin{align*}
P(F_\beta) &= P(g_{\alpha\beta}F_\alpha g_{\alpha\beta}^{-1}) \\
&= \sum_{i_1,\ldots,i_l}(F_\alpha)_{i_1}\wedge \ldots \wedge (F_\alpha)_{i_l}P(g_{\alpha\beta}\xi_\alpha^{i_1}g_{\alpha\beta}^{-1},\ldots, g_{\alpha\beta}\xi_\alpha^{i_l}g_{\alpha\beta}^{-1}) \\
&= \sum_{i_1,\ldots,i_l}(F_\alpha)_{i_1}\wedge \ldots \wedge (F_\alpha)_{i_l}P(\xi_\alpha^{i_1},\ldots,\xi_\alpha^{i_l})  \\
&= P(F_\alpha)
\end{align*}
hence the collection $P(F_\alpha)$ defines a  global 2-form $P(F) \in \Omega^{2l}(M)$.  
\item This follows from the Bianchi identity. Namely, the $\mathrm{Ad}$-invariance of $P$ implies for all $\xi, X_1, \ldots, X_l \in \g$
$$P([\xi, X_1],X_2,\ldots,X_l) + P(X_1,[\xi, X_2],X_3,\ldots,X_l) + \ldots + P(X_1, X_2, \ldots, [\xi,X_l]) = 0$$ 
(by taking derivative of the acting with $g_t = \exp(t\xi)$ at $t=0$). But then, we have 
\begin{align*}
dP(F_\alpha) &= P(dF_\alpha,F_\alpha, \ldots F_\alpha) + \ldots P(F_\alpha,\ldots, F_\alpha, dF_\alpha)  \\
&= -P([A_\alpha,F_\alpha],F_\alpha,\ldots,F_\alpha) - \ldots - P(F_\alpha, \ldots, F_\alpha, [A_\alpha,F_\alpha]) =0. 
\end{align*} 
\item The difference between the two connections $B := A_1 - A_0$ is a 1-form with values in $\mathrm{Ad}P$. Then we define 
\begin{equation}
\eta = k \int_0^1P(F_t,\ldots,F_t,B) dt.\label{eq:Transgression}
\end{equation}
Notice that $\eta \in \Omega^{2k-1}(M)$, as in the proof of part i). 
We claim that $d\eta = P(F(A_1)) - P(F(A_0))$. It is enough to show this in a single trivializing chart. Hence, fix a trivializing chart $U_\alpha$. To simplify the notation, we set $B= B_\alpha,A_1 = (A_1)_\alpha, \ldots$. Define $A_t = A_0 + tB$ and $F_t = F(A(t))$. 
Then, letting $A(t) = A_0 + tB$, we have 
\begin{align*}\dot{F}_t &= \frac{d}{dt}F_t = \frac{d}{dt} \left(d(A_0 + tB ) + \frac{1}{2}0 + tB,A_0 + tB] \right) \\ 
&= \frac{d}{dt}\left( F_0 + t (dB + [A_0,B]) + \frac{t^2}{2}[B,B]\right) = dB + [A_0,B] + t[B,B] \\
&= dB + [A_t,B] = d^{A_t}B
\end{align*}
\end{enumerate}
In particular, we have 
\begin{align*}
P(F_1) - P(F_0) &= \int_0^1 \frac{d}{dt}P(F_t,\ldots F_t) dt \\
&= k\int_0^1  P(F_t,\ldots, F_t), \dot{F}_t))) dt \\
&= k \int_0^1 P(F_t,\ldots, F_t), d^{A_t}B) dt.
\end{align*}
Here we have used the symmetry of $P$.
Hence, to prove the claim it is sufficient to show that $dP(F_t,\ldots,F_t,B) = P(F_t, \ldots, F_t, d^{A_t}B)$. This follows again from Bianchi identity and $\mathrm{Ad}$-invariance. Namely, 
\begin{align*}
dP(F_t,\ldots,F_t,B) &= P(dF_t, \ldots, F_t, B) + \ldots + P(F_t, \ldots, F_t, dB) \\
&= P(dF_t, \ldots, F_t, B) + \ldots + P(F_t, \ldots, dF_t, B) - P(F_t,\ldots,F_t,[A_t,B]) \\
&+ P(F_t, \ldots, F_t, dB) + P(F_t,\ldots,F_t,[A_t,B])  \\
&= \underbrace{-P([A_t,F_t], \ldots, F_t, B) - \ldots - P(F_t, \ldots, [A_t,F_t], B) - P(F_t,\ldots,F_t,[A_t,B])}_{=0 \text{(invariance)}} \\
&+ P(F_t, \ldots, F_t, d_{A_t}B).
\end{align*}
This finishes the proof.
\end{proof}
\begin{rem}
The formula in Equation \eqref{eq:Transgression} gives an explicit expression for the homotopy between the closed forms $P(F(A_1))$ and $P(F(A_0))$. 
\end{rem}
\begin{defn}
We define the space of \emph{inhomogeneous $\mathrm{Ad}$-invariant polynomials} on $\g$ by  \begin{equation} 
\C[\g^*]^G := \oplus_{ l\geq 0} I_l(\g) 
\end{equation}
and  the space of \emph{$\mathrm{Ad}$-invariant formal power series} by 
\begin{equation} 
\C[[\g^*]]^G := \prod_{ l\geq 0} I_l(\g) 
\end{equation}
\end{defn}
If $f$ is an $\mathrm{Ad}$-invariant formal power series, then $f(F(A)) \in \Omega^{\text{even}}(M)$ is well-defined (since only finitely many terms in the power series survive). 
\begin{defn}
The map 
\begin{align*}
\C[[g^*]]^G \times \mathcal{A}_P &\to \Omega^{\text{even}}(M) \\
(f,A) &\mapsto f(F(A))
\end{align*}
\end{defn}

is called \emph{Chern-Weil correspondence}.
In the following sections we look at two important examples of Lie groups and $\mathrm{Ad}$-invariant power series on their Lie algebras: The groups $U(n)$ and $O(n)$. 
\subsection{Characteristic classes of $U(n)$-bundles}
The Lie algebra $u(n)$ of $U(n)$ is the Lie algebra of skew-hermitian matrices, i.e. the Lie algebra of matrices $X$ such that $\overline{X}^T = -X$. Such matrices can be diagonalised by unitary matrices. In particular, any $\mathrm{Ad}$-invariant formal power series $f(X)$ on $u(n)$ is a formal power series in the eigenvalues of $X$. Such formal power series can be constructed as follows: Let $$g(x) = a_0 + a_1x + \ldots + a_kx^x + \ldots = \sum a_kx^k \in \C[[x]]$$ be a formal power series with complex coefficients such that $a_0 = 1$ (this ensures that $g$ is invertible as a power series). Then 
$$f(X) = \det f\left(\frac{\ii}{2\pi}X\right) \in \C[[u(n)^*]].$$
The factor $\frac{i}{2\pi}$ is conventional.
\begin{expl}
Let $g(x) = 1 + x$. Then $c(X) = \det \left(I_{n\times n} + \frac{i}{2\pi}X\right)$ is called the \emph{Chern polynomial}. 
\end{expl}
Introducing a parameter $t$ we can write 
\begin{equation}
c(tX) = 1 + tc_1(X) + t^2c_2(X) + \ldots + t^nc_n(X),
\end{equation}
where $c_i(X)$ is an invariant polynomial\footnote{In fact $c_i$ is the $i$-th symmetric polynomial in $n$ variables evaluated on the eigenvales of $X$.}  of degree $i$ on $u(n)$.
\begin{defn}
Let $A$ be a connection on a principal $U(n)$-bundle $P$ with curvature $F$. Then $$\left[c_i(F)\right]\in H^{2i}(M)$$ is called the \emph{$i$-th Chern class} of $P$ and 
$$\left[c\left(F\right)\right] = \left[1 + c_1\left(F\right) + \ldots+ c_n(F)\right] \in H^{\text{even}}(M)$$
is called \emph{total Chern class of $P$.}
\end{defn}
In particular, the first Chern class $c_1$ is simply given by 
$$[c_1(F)] = \left[\mathrm{tr}\left(\frac{\ii}{2\pi}F\right)\right].$$
Another example which is relevant for geometry is given by the \emph{Todd function}
\begin{equation}td(x) = \frac{x}{1-e^{-x}} = 1 + \frac{1}{2} + \sum_{k=1}^{\infty}\frac{B_{2k}}{(2k)!}x^{2k}\label{eq:toddfunction}
\end{equation}	
where the $B_{2k}$ are the \emph{ Bernoulli numbers}, defined by Equation \eqref{eq:toddfunction} and $B_0 = 1, B_1 = -1/2, B_{2k+1} = 0$ for $k \geq 1$.  The associated characterstic class is called the \emph{Todd class of $P$}. Finally, let us mention also the \emph{Chern character} 
$$ch(X) = \mathrm{tr}\exp\left(\frac{\ii}{2\pi}X\right).$$
Its associated characteristic class is called the \emph{Chern character of $P$}. 
\subsection{Characteristic classes of $O(n)$-bundles}
Here we will only consider the case for $n=2k$ even. The Lie algebra $o(n)$ of $O(n)$ is the Lie algebra of skew-symmetric matrices, i.e. the Lie algebra of matrices $X$ satisfying $X^T = - X$. The matrix $\ii X$ is \emph{hermitian} and hence has real eigenvalues. One can check that they come in pairs $\pm \lambda_j, j=1, \ldots k$. Any skew-symmetric matrix is conjugate by an orthogonal matrix to a block matrix of the form 
$$ X[\lambda_1,\ldots,\lambda_j] = \begin{pmatrix} 0 & -\lambda_j & \cdots & 0 \\
\lambda_j & 0 & \cdots & 0 \\
& & \ddots & \\
0 & \cdots &  0 & -\lambda_j  \\
0 & \cdots & \lambda_j & 0
\end{pmatrix}
$$
Thus, any power series in $\lambda_j^2$ defines an $\mathrm{Ad}$-invariant polynomial on $o(n)$. 
In particular, let $f$ be an even power series 
$$f = 1 + a_2x^2 + a_4x^4 + \ldots,$$
then we have 
$$\det(\ii X) = \prod f(\lambda_j)f(-\lambda_j)  = \prod_{j=1}^kf(\lambda_j)$$
and hence 
$$\det{}^{1/2}f(\ii X) = \prod_{j=1}^k f(\lambda_j).$$
The left hand side can be expressed as a formal power series in the entries of $iX$ which is the square root of the formal power series $\det f(\ii X)$ (this square root is unique if fix its first coefficient to be 1). For a power series $f$ we define 
\begin{equation}
p_f(X) := \det{}^{1/2}f\left(\frac{\ii}{2\pi}X\right)  \in \C[[o(n)^*]]^G
\end{equation}
Again, we are interested in applying this to particular power series $f$. 
\begin{defn}
Let $P$ be a principal $O(n)$-bundle with connection $A$.
Let $f = 1 + x^2$. Then, the corresponding characteristic class 
\begin{equation}\left[ p_f(F(A)) \right]= \left[\det{}^{1/2}\left(1 + \left(\frac{i}{2\pi}X\right)^2 \right)\right] \in \bigoplus_{k\geq 0}H^{4k}(M)
\end{equation}
is called the \emph{total Pontryagin class of $P$}. 
\end{defn}
We can expand the total Pontryagin form in homogeneous degrees 
$$p_f(tF(A)) = 1 + t^2p_1(F(A)) + \ldots t^{2k}p_{k}(F(A)).$$
The  cohomology classes of $p_i(F(A))$ is called $i$-th Pontryagin class of $P$. For example, one can show that 
\begin{equation}
p_1(F(A)) = -\frac{1}{8\pi^2}\mathrm{tr}(F(A)\wedge F(A)).
\end{equation} 
Two more power series are important for geometry. Denote 
\begin{align}
L(x) &= \frac{x}{\tanh(x)} = 1 + \sum_{k=0}^\infty \frac{2^{2k}B_{2k}}{(2k)!}x^{2k} \\
\hat{A}(x) &= \frac{x/2}{\sinh(x/2)} = 1 + \sum_{k=1}^\infty \frac{2^{2k-1}-1}{2^{2k-1}(2k)!}B_{2k}x^{2k}
\end{align}
\begin{defn}
Let $P$ be a principal $O(n)$-bundle with connection $A$.
The \emph{$L$ genus} of $P$ is the characteristic class 
\begin{equation}[p_L(F(A))] =\left[ \det{}^{1/2}\left(\frac{\frac{iF(A)}{2\pi}}{\tanh\left(\frac{\ii F(A)}{2\pi}\right)}\right)\right] \in \bigoplus_{k\geq 0}H^{4k}(M) \end{equation}The \emph{$\hat{A}$ genus} of $P$ is the characteristic class\begin{equation}[p_L(F(A))] = \left[\det{}^{1/2}\left(\frac{\frac{iF(A)}{2\pi}}{\tanh\left(\frac{\ii F(A)}{2\pi}\right)}\right)\right] \in \bigoplus_{k\geq 0}H^{4k}(M) \end{equation}
\end{defn}
\section{$\Spin$ structures, spinors, Dirac operators}
In this section we will discuss Spin structures and see our first examples of Dirac operators. Here we will use a lot of the algebra that we encountered in the first part. 
\subsection{Spin structures}
Let $(M,g)$ be an oriented $n$-dimensional Riemannian\footnote{From now on we stick to the Euclidean case. Most concepts have straightforward analogues in the pseudo-Riemannian case, and the interested reader is invited to think of them as exercises.} manifold. Then, we can define the bundle $\mathrm{soFr}(M)$ of \emph{oriented orthonormal frames of $TM$}: The fiber over $x \in M$ is the collection of all orientation-preserving isometries $e_x\colon (\R^n,g_{std}) \to (T_xM,g_x)$. This is a principal $SO(n)$ bundle and we have 
\begin{equation}
TM \cong \mathrm{soFr}(M) \times \iota \R^n,\label{eq:SOn_structure}
\end{equation}
where $\iota\colon SO(n) \to GL_n(k)$ denotes the inclusion. Conversely, an $SO(n)$-structure on $TM$ - i.e. a principal $SO(n)$-bundle $P$ such that $TM \cong P \times_\iota \R^n$ - defines an orientation and a Riemannian metric on $M$, by declaring the fiber over $x$ to consist of oriented orthonormal frames. Now, recall that we have the short exact sequence \eqref{eq:Spin_SES}
$$ 1 \to \Z_2 \to \Spin_n \to^{\rho} SO(n) $$ 
where $\rho$ is the \emph{adjoint representation} defined in \ref{eq:def_adjoint_rep}, given by 
\begin{equation}
\rho(x)v = x v x^{-1}
\end{equation}
where we use Clifford multiplication on the right hand side. 
\begin{defn}
A \emph{spin structure} on $M$ is a principal $\Spin_n$ bundle $P$ together with an isomorphism 
\begin{equation}
TM \cong P \times_\rho \R^k. 
\end{equation}
\end{defn}
The existence of a spin structure $P$ implies the existence of an $SO(n)$-structure $\rho(P)$ on $TM$, hence gives $M$ the structure of an oriented Riemannian manifold. Conversely, given an oriented Riemannian manifold $M$ with tangent bundle $TM = (\mathfrak{U},g_{\alpha\beta})$, a spin structure is a collection of lifts $tilde{g}_{\alpha\beta}$ such that the following diagram commutes: 
\[
\begin{tikzcd}
                                                                                           & \mathrm{Spin}_n \arrow[d, "\rho"'] \\
U_{\alpha\beta} \arrow[r, "g_{\alpha\beta}"] \arrow[ru, "\tilde{g}_{\alpha\beta}", dashed] & SO(n)                             
\end{tikzcd}
\]
and that satisfy the relations \eqref{eqs:gluingMapsPrincipal} for a prinicipal $Spin_n$-bundle: 
\begin{align}
\tilde{g}_{\alpha\alpha} &= 1 \label{eq:lifts1} \\
\tilde{g}_{\alpha\beta} &= \tilde{g}_{\beta\alpha}^{-1}  \label{eq:lifts2}\\
\tilde{g}_{\alpha\gamma} &= \tilde{g}_{\beta\gamma}\tilde{g}_{\alpha\beta}.\label{eq:lifts3}
\end{align}
When do such lifts exist? From covering theory, one can conclude that if $U_{\alpha\beta}$ is simply connected, then lifts satisfying \eqref{eq:lifts1} and \eqref{eq:lifts2} do alwawy exist. The only non-trivial question is whether one can find a lift satisfying \eqref{eq:lifts3} (called the cocycle condition). To answer this question we digress into algebraic topology and discuss the concept of \v{C}ech cohomology. 
\subsection{Digression: \v{C}ech cohomology}
Let $\mathfrak{U}= \{U_\alpha\}_{\alpha\in A}$ be an open cover of $M$. For $\alpha_1,\ldots,\alpha_n$ we set 
$$U_{\alpha_1 \ldots \alpha_n} := U_{\alpha_1} \cap \cdots \cap U_{\alpha_n}.$$
\begin{defn}
A cover $\mathfrak{U}= \{U_\alpha\}_{\alpha\in A}$ is \emph{good} if for all $k \in \N$ and $\alpha_1, \ldots, \alpha_k \in A$ we have $U_{\alpha_1\ldots\alpha_k}$ contractible.
\end{defn}
It is a general fact that manifolds always admit good covers. Given any open cover, one can define the \v{C}ech cochains and coboundary operator as follows. 
\begin{defn}
Let $G$ be an abelian group and $\mathfrak{U}= \{U_\alpha\}_{\alpha\in A}$ be an open cover of $M$. Then the space of \emph{\v{C}ech cochains} is 
\begin{equation}
\check{C}^k(\mathfrak{U},G) = \left\lbrace f\colon \{(\alpha_0,\ldots,\alpha_k)\in A^{k+1} |U_{\alpha_0\ldots\alpha_k}\neq \varnothing\} \to G \right\rbrace
\end{equation}
\end{defn}
We typically denote $f(\alpha_0,\ldots,\alpha_k) = f_{\alpha_0\ldots\alpha_k} \in G$. $\check{C}^k(\mathfrak{U},G)$ is itself an abelian group. 
\begin{defn}
The \v{C}ech coboundary operator $\delta \colon \check{C}^k(\mathfrak{U},G) \to \check{C}^{k+1}(\mathfrak{U},G)$ is defined on elements $\sigma \in \check{C}^k(\mathfrak{U},G)$ by 
\begin{equation}
(\delta\sigma)_{\alpha_0\ldots\alpha_{k+1}} = \sigma_{\alpha_1}\ldots - \sigma_{\alpha_0} + \ldots + (-1)^{k+1}\sigma_{\alpha_0\ldots\alpha_k}
\end{equation}
\end{defn}
Then we have the following Lemma, for whose proof we refer e.g. to \cite{Bott1995}: 
\begin{lem}
$\delta$ is a group homomorphism and 
$\delta^2 = 0$.
\end{lem}
Hence, we can define \v{C}ech cocyles, coboundaries, and cohomology as follows. 
\begin{defn}
The subgroup $Z^k(\mathfrak{U},G) \subset \check{C}^k(\mathfrak{U},G)$ of \emph{\v{C}ech cocycles} is 
\[
Z^k(\mathfrak{U},G) = \ker\left(\delta\colon \check{C}^k(\mathfrak{U},G) \to \check{C}^{k+1}(\mathfrak{U},G) \right)
\]
The subgroup $B^k(\mathfrak{U},G) \subset \check{C}^k(\mathfrak{U},G)$ of \emph{\v{C}ech coboundaries} is 
\[
B^k(\mathfrak{U},G) = \mathrm{im}\left(\delta\colon \check{C}^{k-1}(\mathfrak{U},G) \to \check{C}^{k}(\mathfrak{U},G) \right)
\]
The $k$-th \v{C}ech cohomology group of $\mathfrak{U}$ with coefficients in $G$ is the quotient
\[\check{H}^k(\mathfrak{U},G) = Z^k(\mathfrak{U},G) \big/ B^k(\mathfrak{U},G). 
\]
\end{defn}
The importance of \v{C}ech coohomology comes (partly) from the following theorem (see \cite{}): 
\begin{thm}
Let $\mathfrak{U}$ be a good cover of $M$. Then 
$$\check{H}^k(\mathfrak{U},G) \cong H^k_{sing}(M,G).$$
\end{thm}
In particular, this implies that the \v{C}ech cohomology of good covers does not depend on the cover. 
\begin{rem}
The precise definition of the object on the right is not important right now (and one can even take the theorem as a definition). However, it is good to know that there are plenty of ways in algebraic topology to compute the group on the right hand side. 
\end{rem}
\subsubsection{The first Stiefel-Whitney class and orientability}
Let us study an easy but important example. Let $G = \Z_2 = (\{\pm 1\}, \cdot)$. Let $M$ be a Riemannian manifold and $\mathfrak{U}$ a good trivializing cover for the bundle of orthonormal frames $\mathrm{oFr}(M)$, with transition functions $g_{\alpha\beta}\colon U_{\alpha\beta} \to SO(n)$. Define\footnote{One can also drop the Riemannian metric and instead work with $c_{\alpha\beta} = \mathrm{sign} \det g_{\alpha\beta}$, but this is equivalent and heavier on notation.} $$c_{\alpha\beta} := \det g_{\alpha\beta} = \pm +1 \in \Z_2.$$ 
Since $c_{\alpha\beta} \colon U_{\alpha\beta} \to \Z_2$ is continuous and $U_{\alpha\beta}$ is contractible, $c_{\alpha\beta}$ is constant and hence defines a \v{C}ech 1-cochain $c \in \check{C}^k(\mathfrak{U},G)$. We claim that $c$ actually defines a \v{C}ech 1-cocycle. To see this, we simply compute 
\begin{align*}
(\delta c)_{\alpha\beta\gamma} &= c_{\beta\gamma}c^{-1}_{\alpha\gamma}c_{\alpha\beta}  \\
&= \det g_{\beta\gamma} \det g_{\gamma\alpha}\det g_\alpha\beta \\
&= \det (g_{\beta\gamma}\underbrace{g_{\gamma\alpha}g_{\alpha\beta}}_{g_{\gamma\beta}} ) = 1.  
\end{align*}
\begin{defn} The cohomology class $[c]:=w_1(M) \in H^1(M,\Z_2)$ defined by $c$ is called the \emph{first Stiefel-Whitney class} of $M$.
\end{defn}
\begin{thm}
The first Stiefel-Whitney class vanishes if and only if $M$ is orientable.
\end{thm}
\begin{proof}
Notice that we can compute \v{C}ech cocyles with respect to any good open cover. Suppose $M$ is orientable, and pick an orientation. Then we can find a good open trivializing cover for the principal $SO(n)$- bundle of oriented orthonormal frames. Computing the first Stiefel-Whitney class in this open cover we see it is trivial, since all transition functions have determinant 1. \\
Conversely, suppose the first Stiefel-Whitney class vanishes. Pick a Riemannian metric $g$ on $M$ and a good open trivializing cover for the orthonormal frame bundle. If the first Stiefel-Whitney class $w_1(M)$, we know that the \v{C}ech cocycle $c$ that computes it is a \v{C}ech coboundary: $c_{\alpha\beta} = (\delta s)_{\alpha\beta} = c_\beta c_{\alpha}^{-1}$. Now, we redefine the trivializations $\psi_{\alpha} \colon \pi^{-1}(U_\alpha) \to U_\alpha \times \R^n$ by post-composing with the map $c_\alpha \colon \R^n \to \R^n, (x^1,\ldots,x^n) \to (c_\alpha x^1, \ldots , x^n)$: $\psi'_\alpha = c_\alpha \circ \psi_\alpha$. Then the new transition functions $g'_{\alpha\beta}$ satisfy 
$$\det g'_{\alpha\beta} = \det c_\beta \det g_{\alpha\beta} \det_{c_\alpha^{-1}} = \det g_{\alpha\beta}^2 = 1,$$
and we conclude that $M$ is orientable since it admits a system of transition functions with determinant 1. 
\end{proof}
After this warm-up, let us return to the question of spin structures. 
\subsubsection{The second Stiefel-Whitney class and spin structures}
Let $(M,g)$ be a Riemannian manifold and let $\mathfrak{U}$ be a good trivializing cover for $\mathrm{soFr}(M)$, with transition functions $g_{\alpha\beta} \colon U_{\alpha\beta} \to SO(n)$. Find a lift $\tilde{g}_{\alpha\beta} \colon U_{\alpha\beta}\to \Spin_n$ such that $\tilde{g}_{\alpha\alpha}$ and $\tilde{g}_{\alpha\beta} = \tilde{g}_{\beta\alpha}^{-1}$. Define 
\begin{equation}
\varepsilon_{\gamma\beta\alpha} = \tilde{g}_{\gamma\alpha}\tilde{g}_{\beta\gamma}\tilde{g}_{\alpha\beta}
\end{equation}
We want to show that this assignment gives rise to a \v{C}ech 2-cocycle with values in $\Z_2$ associated to $\mathfrak{U}$ whose cohomology class is independent of the lift. This is done through the following series of claims.
\begin{clm}
For all $\alpha,\beta,\gamma$ we have
$$\varepsilon_{\gamma\beta\alpha} \in \ker\rho\cong\Z_2$$
where $\rho\colon \Spin_n\to SO(n)$ denotes the adjoint representation. 
\end{clm}
\begin{proof}
Simply apply $\rho$ to $\varepsilon$ and use the cocycle condition for $\mathrm{soFr}(M)$: 
\begin{align*}\rho(\varepsilon_{\gamma\beta\alpha}) &= \rho(\tilde{g}_{\gamma\alpha}\tilde{g}_{\beta\gamma}\tilde{g}_{\alpha\beta} ) \\
&= \rho(\tilde{g}_{\gamma\alpha})\rho(\tilde{g}_{\beta\gamma})\rho(\tilde{g}_{\alpha\beta}) \\
&= g_{\gamma\alpha}g_{\beta\gamma}g_{\alpha\beta} = 1. 
\end{align*}
\end{proof}
\begin{clm}
$\varepsilon \in \check{C}^k(\mathfrak{U},\Z_2)$ defines a \v{C}ech 2-cocycle.
\end{clm}
\begin{proof}
The proof is a straightforward computation with some smart tricks: 
\begin{align*}
(\delta \varepsilon)_{\delta\gamma\beta\alpha} &=  \varepsilon_{\gamma\beta\alpha}\varepsilon^{-1}_{\delta\beta\gamma}\varepsilon_{\delta\gamma\alpha}\varepsilon_{\delta\gamma\beta}^{-1} \\
&= \tilde{g}_{\gamma\alpha}\tilde{g}_{\beta\gamma}\tilde{g}_{\alpha\beta}(\tilde{g}_{\delta\alpha}\tilde{g}_{\beta\delta}\tilde{g}_{\alpha\beta})^{-1}\tilde{g}_{\delta\alpha}\tilde{g}_{\gamma\delta}\tilde{g}_{\alpha\gamma}(\tilde{g}_{\delta\beta}\tilde{g}_{\gamma\delta}\tilde{g}_{\beta\gamma})^{-1} \\
&= \tilde{g}_{\gamma\alpha}\tilde{g}_{\beta\gamma}\underbrace{\tilde{g}_{\alpha\beta}(\tilde{g}_{\beta\alpha}}_{=1}\tilde{g}_{\delta\beta}\underbrace{\tilde{g}_{\alpha\delta}) \tilde{g}_{\delta\alpha}}_{=1}\tilde{g}_{\gamma\delta}\tilde{g}_{\alpha\gamma}(\tilde{g}_{\gamma\beta}\tilde{g}_{\delta\gamma}\tilde{g}_{\beta\delta}) \\
&= \tilde{g}_{\gamma\alpha}\underbrace{\tilde{g}_{\beta\gamma}\tilde{g}_{\delta\beta}\tilde{g}_{\gamma\delta}}_{=\varepsilon_{\beta\delta\gamma}\in \{\pm 1\} }\tilde{g}_{\alpha\gamma}(\tilde{g}_{\gamma\beta}\tilde{g}_{\delta\gamma}\tilde{g}_{\beta\delta}) \\
&=  \underbrace{\tilde{g}_{\gamma\alpha}\tilde{g}_{\alpha\gamma}}_{=1}(\tilde{g}_{\gamma\beta}\varepsilon_{\beta\delta\gamma}\tilde{g}_{\delta\gamma}\tilde{g}_{\beta\delta}) \\
&= \tilde{g}_{\gamma\beta}\tilde{g}_{\beta\gamma}\tilde{g}_{\delta\beta}\tilde{g}_{\gamma\delta}\tilde{g}_{\delta\gamma}\tilde{g}_{\beta\delta} = 1.
\end{align*}
\end{proof}
\begin{clm}
The \v{C}ech cohomology class of $\varepsilon$ is independent of the lift $\tilde{g}_{\alpha\beta}$. In fact, if $\tilde{g}'_{\alpha\beta}$ is another lift, then $\varepsilon' = \varepsilon \delta \kappa$, where $\kappa_{\alpha\beta} = \tilde{g}_{\alpha\beta} \tilde{g}'_{\beta\alpha}$. 
\end{clm}
\begin{proof}
Note that if $\tilde{g}'_{\alpha\beta}$ and $\tilde{g}_{\alpha\beta}$ are lifts of $g_{\alpha\beta}$, then $\kappa_{\alpha\beta} = \tilde{g}_{\alpha\beta} \tilde{g}'_{\beta\alpha}$ satisfies $\rho(\kappa_{\alpha\beta})= 1$, hence $\kappa_{\alpha\beta}$ is a \v{C}ech 1-cochain. Now, we observe that 
\begin{align*}
\varepsilon_{\gamma\beta\alpha}(\delta\kappa)_{\gamma\beta\alpha} 
 &= (\tilde{g}_{\gamma\alpha}\tilde{g}_{\beta\gamma}\tilde{g}_{\alpha\beta} )\kappa_{\beta\alpha}\kappa_{\gamma\alpha}^{-1}\kappa_{\beta\gamma}\\
&= \kappa_{\gamma\alpha}^{-1}\tilde{g}_{\gamma\alpha}\tilde{g}_{\beta\gamma}\kappa_{\gamma\beta}\tilde{g}_{\alpha\beta}\kappa_{\beta\alpha} \\
&= (\tilde{g}'_{\gamma\alpha}\tilde{g}'_{\beta\gamma}\tilde{g}'_{\alpha\beta} ) = \varepsilon'_{\gamma\beta\alpha}.
\end{align*}
\end{proof}
We conclude that the cohomology class $[\varepsilon]  \in H^2(M, \Z_2)$ defined by the \v{C}ech cocycle $\varepsilon$ is independent of the lift.
\begin{defn}
The cohomology class $w_2(M):=[\varepsilon]\in H^2(M,\Z_2)$ is called the \emph{second Stiefel-Whitney class of $M$}. 
\end{defn} 
The importance of this class comes from the following theorem. 
\begin{thm}
$M$ admits a spin structure if and only if the second Stiefel-Whitney class vanishes. 
\end{thm}
\begin{proof}
Suppose $M$ admits a spin structure. Then there exists a lift $\tilde{g}_{\alpha\beta}$ satisfying the cocycle condition. In that case $\varepsilon_{\gamma\beta\alpha} \equiv 1$. \\
For the other direction, we refer to the literature (e.g. \cite{Taubes2011}. 
\end{proof}
We also briefly discuss the classification of spin structures. An isomorphism of spin structures is an isomorphism $T = \{T_\alpha\}$ of principal $\Spin_n$-bundles that leaves $\rho$ invariant, i.e. the following diagram commutes: 
\[
\begin{tikzcd}
U_{\alpha}\times\mathrm{Spin}_n \arrow[rd, "\rho"'] \arrow[rr, "T_\alpha"] &                       & U_\alpha\times\mathrm{Spin}_n \arrow[ld, "\rho"] \\
                                                                           & U_\alpha \times SO(n) &                                                 
\end{tikzcd}
\]
This means that $T_{\alpha}(x) \in \ker\rho$, hence $T_\alpha(x) = \pm 1$ is constant. If the two spin structures are given by lifts $\tilde{g}_{\alpha\beta},\tilde{h}_{\alpha\beta}$, then the requirement that $T_{\alpha}$ defines a map of principle bundles is $T_\beta = \tilde{h}_{\alpha\beta}T_{\alpha}\tilde{g}_{\alpha\beta}^{-1}$ or equivalently 
\begin{equation}
h_{\alpha\beta} = T_\beta g_{\alpha\beta} T_{\alpha}^{-1} = g_{\alpha\beta}(\delta T)_{\alpha\beta}
\end{equation}
(since $T_\beta$ is in the center of $\Spin_n$). Hence, two spin structures are isomorphic if and only if they differ by a \v{C}ech 1-coboundary. 
On the other hand, assume we are given a spin structure $\tilde{g}_{\alpha\beta}$ and a \v{C}ech 1-cochain $c_{\alpha\beta} \in \check{C}^2(\mathfrak{U},\Z_2)$. Then $h_{\alpha\beta} = g_{\alpha\beta}c_{\alpha\beta}$ satisfies the cocycle condition, and hence defines a spin structure if and only if $\delta c =1$, i.e. $c$ defines a \v{C}ech 1-cocycle. These two spin structures are isomorphic if and only if $c = \delta T$ is a coboundary. Hence, we get an action of $H^1(M,\Z_2)$ on the set of spin structures, given by $c.\tilde{g} = g_{\alpha\beta}c_{\alpha\beta}$. 
\begin{thm}
The action of $H^1(M,\Z_2)$ is free and transitive. In particular, there is a non-canonical bijection 
$$H^1(M,\Z_2) \leftrightarrow \{ \text{isomorphism classes of spin structures} \}.$$
\end{thm}
For a proof, see \cite{Taubes2011} or \cite{Nicolaescu2013}.
\begin{expl}
As an example, let us consider the circle $S^1$. A good cover $\mathfrak{U} = \{U_1,U_2,U_3\}$ is given by three intervals overlapping only at the ends (see figure \ref{fig:circle}). 
\begin{figure}[h!]
\centering
\begin{tikzpicture}
\draw (0,0) circle (2cm);
\draw[red] (2.5,0) arc (0:140:2.5cm) node[pos=0.5,above,black]{$U_1$};
\draw[blue] (120:3cm) arc (120:260:3cm) node[pos=0.5,left,black]{$U_2$};
\draw[pink] (240:2.75cm) arc (240:380:2.75cm) node[pos=0.5,below right,black]{$U_3$};
\draw[line width=1mm] (120:2cm) arc (120:140:2cm) node[pos=0.5,below right,black]{$U_{12}$};
\draw[line width=1mm] (240:2cm) arc (240:260:2cm) node[pos=0.5,above,black]{$U_{23}$};
\draw[line width=1mm] (0:2cm) arc (0:20:2cm) node[pos=0.5,left,black]{$U_{13}$};
\end{tikzpicture}
\caption{Good cover of the circle}\label{fig:circle}
\end{figure}
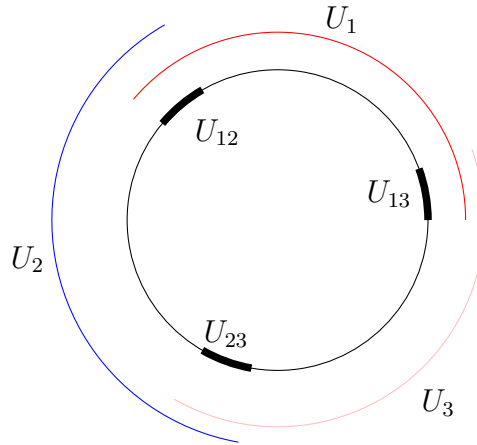 Since all triple intersections are 0, we see immediately that $H^2(S^1,G) = \{0\}$ for any abelian group $G$. In particular, $S^1$ admits spin structures. Since $S^1$ is orientable, its tangent bundle admits an $SO(1)$-structure, but $SO(1) = \{1\}$. Hence the tangent bundle is trivial and the $SO(1)$-transition functions in $\mathfrak{U}$ are $g_{12}=g_{23}=g_{13} = 1$. It follows that we can choose the trivial lifts\footnote{In this particular case, the spin structure is itself a 1-cocycle, this is of course highly peculiar to the 1-dimensional case.} $\tilde{g}_{12} = \tilde{g}_{23} = \tilde{g}_{13} = 1$.  We claim that there is exactly one other spin structure on $S^1$, given (up to isomorphism) by\footnote{This is not the $-1$ in $O(1)$ but in $\Spin_1$. However, we have $\Spin_1 \cong O(1)\cong \Z_2$. This why there are exactly two isomorphism classes of line bundles and two isomorphism classes of spin structures on the circle. } $\tilde{h}_{12} = \tilde{h}_{23} = 1$, $\tilde{h}_{13} = -1$. To see this,  note that any 1-coboundary $(\delta T)_{\alpha\beta} = T_{\beta}T_\alpha^{-1}$ can flip either zero or two signs of the three possible ones. This proves that $H^1(M,\Z_2) = \Z_2)$ with generators $\tilde{g},\tilde{h}$. 
Confusingly, $\tilde{h}$ is usually called the \emph{trivial spin structure} (since it is the one which extends to the disk). In string theory, $\tilde{h}$ is known as the Neveu-Schwarz spin-structure, while $\tilde{g}$ is known as the Ramond spin structure. 

\end{expl}
\begin{exc}
Find good covers of the 2-sphere $S^2$ (four sets are enough) and the 2-torus  $T^2 = S^1 \times S^1$ (three sets are enough). Do they admit spin structures? If so, how many? 
\end{exc}
\begin{rem}
The second Stiefel-Whitney class $w_2(M) \in H^2(M,\Z_2)$ is called the obstruction for a spin structure. This is one of the beginnings of what is known as \emph{obstruction theory}: If we know that the obstruction vanishes then we know spin structures exist. Of course, the easiest case is when $H^2(M,\Z_2) = \{0\}$: The obstruction has simply no place to exist, hence, it must vanish. However, there are many more subtle criteria as to when spin structures exist. We refer to the literature for a deeper discussion of these issues (some further results and references can be found in \cite{Nicolaescu2013}). 
\end{rem}
\subsection{Spinors}
After this long digression on existence and classification of spin structures, let us return on track. Let $(M,g)$ be an $n$-dimensional Riemannian manifold with a spin structure $P$. Recall the complex spinorial representation $\Delta_n$ of $Cl_n^c$ from Definition \ref{def:complex_spin_rep_cl}. It restricts to a representation of $\Spin_n$ via the sequence
\begin{equation}
\Spin_n \subset Cl_n \subset Cl_n^c \overset{c}{\longrightarrow} \End(\Delta_n)
\end{equation}
\begin{defn}
The \emph{spinor bundle} $S_n$ associated to $P$ is the vector associated to the prinicipal $\Spin_n$-bundle $P$ via the complex spinorial representation: 
\begin{equation}
S_n = P \times_c \Delta_n
\end{equation}
\end{defn}
\begin{rem}Recall that if $n$ is even, then $\Delta_n = \Delta_n^+ \oplus \Delta_n^-$ splits as a direct sum of irreps. This implies a splitting of the spinor bundle as $S_n = S^+_n \oplus S_n^-$. 
\end{rem}
Next, we define a bundle of Clifford algebras over $M$. 
\begin{defn}
The \emph{Clifford bundle} is the vector bundle over $M$ with typical fiber the Clifford algebra $ Cl(M)_x = Cl(T^*_xM,g_x)$.
\end{defn} Let us analyze this bundle more closely. Recall that the standard representation of $SO(n)$ embeds into algebra automorphisms of $Cl_n$: For $A \in SO(n)$, the map $\tau(A)\colon Cl_n \to Cl_n$ defined on generators $v\in\R^n$ of the Clifford algebra by $v \mapsto \tau(A)v=Av$ is an algebra homomorphism since 
$$\tau(A)(vw+wv) = Av Aw + Aw Av = -2\langle Av, Aw \rangle = -2\langle v,w\rangle.$$ Notice also that, for a Riemannian manifold $M$, if $\mathfrak{U},g_{\alpha\beta}$ is the bundle orthonormal frames, the cotangent bundle $T^*M = (TM)^*$ is given by transition maps $h_{\alpha\beta} = (g^*)_{\alpha\beta}^{-1} =g_{\alpha\beta}$, since $g_\alpha\beta \in O(n)$. 
This discussion can be summarized in the following proposition: 
\begin{prop}
The Clifford bundle is a bundle associated with the bundle of oriented orthonormal frames: 
$$Cl(M) = \mathrm{soFr(M)} \times_\tau Cl_n.$$
The transition functions act by algebra automorphisms. 
\end{prop}
We conclude that $Cl(M)$ has a well-defined Clifford multiplication (defined over trivializations by multiplication of sections) 
$Cl(M) \oplus Cl(M) \to Cl(M)$.  We now want to define the action of this algebra bundle of the spinor bundle. Notice that we did not need the spin structure to define the Clifford bundle. However, the spin structure will be essential in defining the Clifford multiplication of spinors. 
\begin{prop} 
The Clifford multiplication $\R^n \times \Delta_n \to \Delta_n, (v,s) \mapsto c(v)s$, extends to a map of sections 
\begin{align}
c \colon \Gamma(T^*M)\times S_n &\to S_n \\
(\theta,\psi) \mapsto c(\theta)\psi \notag
\end{align}
\end{prop}
\begin{proof}
By construction, a trivializing chart $U_\alpha$ for $\mathrm{soFr(M)}$ trivializes both $T^*M$ and $S_n$. Over $U_\alpha$, a section $\theta$ of $T^*M$ is given by $\theta_\alpha \colon U_\alpha\to \R^n$, and a section $\psi$ of $S_n$ is given by $\psi_\alpha\colon U_\alpha \to \Delta_n$. We define 
$$(c(\theta)\psi)_\alpha = c(\theta_\alpha)\psi_\alpha.$$
We have to show that this is a section of $S_n$, that is, we have to show that $$(c(\theta)\psi)_\beta = c(\tilde{g}_{\alpha\beta})(c(\theta)\psi)_\alpha$$
(the transition functions for the spinor bundle are $c(\tilde{g}_{\alpha\beta}$). To see this, we compute 
\begin{align*}
(c(\theta)\psi)_\beta &= c(\theta_\beta)\psi_\beta = c(g_{\alpha\beta}\theta_\alpha)c(\tilde{g}_{\alpha\beta})\psi_\alpha \\
\text{(M is spin !!!)}&= c (\rho(\tilde{g}_{\alpha\beta})\theta_\alpha)\psi_\alpha \\ 
&= c (\tilde{g}_{\alpha\beta}\theta_\alpha\tilde{g}_{\alpha\beta}^{-1})c(\tilde{g}_{\alpha\beta})\psi_\alpha = c (\tilde{g}_{\alpha\beta}\theta_\alpha\tilde{g}_{\alpha\beta}^{-1}\tilde{g}_{\alpha\beta})\psi_\alpha \\ 
&= c(\tilde{g}_{\alpha\beta})c(\theta_\alpha)\psi_\alpha = c(\tilde{g}_{\alpha\beta})(c(\theta)\psi)_\alpha
\end{align*}
\end{proof}
In the proof, the fact that $g_{\alpha\beta} = \rho(\tilde{g}_{\alpha\beta})$ is crucial! Without the spin structure, we cannot define the Clifford multiplication. 
\subsection{Spin connections}
Let $\Omega$ be a connection (for example, the Levi-Civita connection) on the prinicipal $SO(n)$-bundle of oriented orthonormal frames, described over trivializing charts by $A_\alpha \in \Omega^1(U_\alpha,so(n))$. It gives rise to a connection $\nabla$ on $T^*M$ with the same 1-forms $A_\alpha$. The trivializing chart defines a local orthonormal frame $e_i$ of $T^*M$. In this frame, we can express $A_\alpha = \frac{1}{2}A^{ij}e_i\wedge e_j$, where $e_i \wedge e_j$ is the skew-symmetric endomorphism defined in Equation \eqref{eq:isosonwedge2}. Remember that in Equation \eqref{eq:propLieAlgIso1} we computed the explicit isomorphism 
$\rho_*\colon spin_n \to so_n$ and we had $\rho_*^{-1}(e_i \wedge e_j) = \frac{1}{4}[e_i,e_j]= \frac{1}{4}(e_ie_j-e_je_i)$. 
\begin{defn}
The connection $\widetilde{\nabla}$ on the vector bundle $S_n$ defined by the one-forms $B_{\alpha} = c(\rho_*^{-1}(A_\alpha) \in \Omega^1(U_\alpha, \End (\Delta_n)$ is called the \emph{spin connection} associated to $\nabla$. 
\end{defn}
With respect to a local orthonormal frame, the connection 1-forms $B_\alpha$ can be written as $B_\alpha = \frac{1}{8}A^{ij}c([e_i,e_j]) = \frac{1}{4}A^{ij}c(e_iej)$.
The most important fact about spin connections is that they are compatible with spinor multiplication. 
\begin{prop}\label{prop:spinconnection}
Let $\widetilde{\nabla}$ be the spin connection associated to $\nabla$. Then, for all and sections $\theta \in \Gamma(T^*M), \psi \in \Gamma(S_n)$, we have 
 \begin{equation}
 \widetilde{\nabla}(c(\theta)\psi) = c(\nabla\theta)\psi + c(\theta) \widetilde{\nabla}\psi.
 \end{equation}
\end{prop}
\begin{proof}
First, let $v,e_i,e_j \in \R^n$, where $e_i,e_j$ are elements of an orthonormal basis. Then, in the Clifford algebra $Cl_n$ we have, using the elementary Clifford relation \eqref{eq:CliffordRelation},
$$e_ie_j v = ve_ie_j - 2\langle v,e_j\rangle e_i + 2\langle v,e_i\rangle e_j = ve_ie_j + 2 (e_i \wedge e_j)(v).$$
Using this we simply compute\footnote{Over the trivializing chart given by the local orthonormal frame, both $\theta$ and $\psi$ are sections of trivial bundles, as such $d$ acts upon them satisfying the Leibniz rule.}
\begin{align*}
\widetilde{\nabla}(c(\theta)\psi) &= d(c(\theta)\psi) + \frac{1}{4}A^{ij}c(e_ie_j)c(\theta \psi \\
&= c(d\theta)\psi + c(\theta)d\psi + \frac{1}{4}c(\theta)A^{ij}c(e_ie_j)\psi + \frac12 A^{ij}c(e_i \wedge e_j(\theta))\psi \\
&= c\left(d\theta + \frac12A^{ij}(e_i \wedge e_j)(\theta)\right)\psi + c(\theta)(d\psi + \frac14 A^{ij}c(e_ie_j)\psi) \\
&= c(\nabla\theta)\psi + c(\theta)\widetilde{\nabla}\psi.
\end{align*}

\end{proof}
The Clifford multiplication extends to $Cl(M)$ and we have 
\begin{cor}
Denote $\omega_\C$ the section which on fibers is the complex volume element (of definition \ref{def:complex_volume}) given by $\omega_\C = \ii^{\lfloor \frac{n+1}{2}\rfloor}e_1\ldots e_n$ in a local orthonormal frame.  Then 
\begin{equation}
\widetilde{\nabla}(c(\omega_\C)\psi) = c(\omega_\C)\widetilde{\nabla}\psi\label{eq:corcomplexvolume}
\end{equation}
\end{cor}
\begin{proof} Fix $p \in M$ and assume that the orthonormal frame satisfies $\nabla e_i = 0$, at $p$\footnote{Such orthonormal frames are called \emph{synchronous at $p$} and always exist.}. Then,  repeated use of proposition \ref{prop:spinconnection} immediately implies \eqref{eq:corcomplexvolume} at the point $p$. Since $p$ is arbitrary, we conclude the statement. 
\end{proof}
\subsection{Dirac operators}
Finally, we have all the necessary ingredients to define Dirac operators. Let $(M,g)$ be a Riemannian manifold with spin structure $P$. We recall that this induces a bundle of spinors $S_n$. The lift of the Levi-Civita connection $\nabla_g$ to the spin bundle is a connection $\widetilde{\nabla}_g$ on $S_n$.
\begin{defn}
The \emph{Dirac operator} $D \colon \Gamma(S_n) \to \Gamma(S_n)$ associated to $P$ is given by the composition 
\begin{equation}
\Gamma(S) \overset{\widetilde{\nabla}_g}{ \longrightarrow }\Gamma(T^*M \otimes S) \overset{c}{\longrightarrow}\Gamma(S)
\end{equation}
\end{defn}
In a local orthonormal frame, the Dirac operator is given by $D = \sum c(e_i)\widetilde{\nabla}_{e_i}$. In particular, for $M = \R^4$ with euclidean metric, we have $D = c(e_i)\frac{\de}{\de x^i}$ (the spinor bundle and spin connection are trivial) and hence 
\begin{equation}D^2f = c(e_i)\frac{\de}{\de x^i}\left(c(e_j)\frac{\de}{\de x^j}f\right) = c(e_ie_j) \frac{\de^2}{\de x^i\de x^j}f = \frac{1}{2}c(e_ie_j + e_je_i)\frac{\de^2}{\de x^i\de x^j}f = -\sum \frac{\de^2}{\de (x^i)^2}f.\label{eq:Dsquare}
\end{equation}
We will see that a similar equation holds for all Dirac Operators.
\begin{rem}
If $\dim M = 2k $ is even, the $S_n = S^+_n \oplus S_n^-$ and we have have 
\begin{align*}
D_+:= \restr{D}{S^+_n} &\colon S_n^+ \to S^-_n \\
D_-:= \restr{D}{S^-_n} &\colon S_n^- \to S^+_n \\
\end{align*}
\end{rem}
\section{Clifford Modules}
In the last section, the spin structure was essential in defining the bundle of spinor and the Clifford multiplication.  However, to define the Dirac operator, the spin structure was not needed. All we used was the Clifford multiplication and the spin connection. We can try to extract the structure necessary for defining the Dirac operator. This will prove highly beneficial. 
\begin{defn}
Let $(M,g)$ be an oriented Riemannian manifold with Levi-Civita connection $\nabla^g$. 
\begin{enumerate}[i)]
\item A \emph{super vector bundle} over $M$ is a vector bundle $E \to M$ together with a direct sum decomposition $E = E_+ \oplus E_-$, where $E_0,E_1$ are vector bundles over $M$.
\item A \emph{Clifford module} over $M$ is a complex super vector bundle $E = E_0 \oplus E_1$ with a hermitian metric $h$ and a map \begin{align*} c\colon \Gamma(T^*M) \times \Gamma(E) &\to \Gamma(E) \\
(\theta,\psi) &\mapsto c(\theta)\psi 
\end{align*}
such that for all $\theta,\theta_1,\theta_2 \in\Gamma(T^*M), \psi,\psi_1,\psi_2 \in \Gamma(E)$ we have
\begin{enumerate}[a)]
\item the Clifford relation: 
\begin{equation}
c(\theta_1)c(\theta_2)\psi + c(\theta_2)c(\theta_1)\psi = -2g(\theta_1,\theta_2)\psi,
\end{equation}
\item unitarity: if $g(\theta,\theta) \equiv 1$, then \begin{equation}
h(c(\theta)\psi_1,\psi_2) + h(\psi_1,c(\theta)\psi_2) = 0, \label{eq:unitarity}
\end{equation}
\item the action is odd: 
\begin{equation}
\restr{c(\theta)}{\Gamma(E)}\colon \Gamma(E_\pm) \to E_\mp
\end{equation}
\end{enumerate}
\item A \emph{Clifford connection} on a Clifford bundle $E$ is a metric connection $\nabla^E$ which is compatible with the Clifford multiplication, that is, for all $X \in \Gamma(TM), \theta \in \Gamma(T^*M), \psi \in \Gamma(E)$ we have 
\begin{equation}
\nabla^E_X(c(\theta)\psi) = c(\nabla_X^g\theta)\psi + c(\theta)\nabla_X^E\psi
\end{equation}
(remember that $\nabla^g$ denotes the Levi-Civita connection). 
\item A \emph{Dirac bundle} is a pair $(E,\nabla^E)$ of a Clifford bundle with a Clifford connection. 
\item The \emph{Dirac operator} associated to a Dirac bundle $(E,\nabla^E)$ is the composition 
\begin{equation}
D_E \colon \Gamma(E) \overset{\nabla^E}{\longrightarrow} \Gamma(T^*M \otimes E) \overset{c}{\longrightarrow} \Gamma(E).
\end{equation}
\end{enumerate}
\end{defn}
Let us give some remarks on this definition.
\begin{rem}
Notice that Equation \eqref{eq:unitarity} is equivalent to 
\begin{equation}
h(c(\theta)\psi_1,c(\theta)\psi_2) = g(\theta,\theta)h(\psi_1,\psi_2).
\end{equation} 
In particular, sections of unit norm act unitarily under Clifford multiplication. 
\end{rem}
\begin{rem}
Although we will not prove it explicitly, the important piece of data here is the Clifford multiplication. The metric and a metric Clifford connection always exist, given that one has defined Clifford multiplication. See also example \ref{expl:spinors} below.
\end{rem}
\begin{rem}
By virtue of the Clifford relation, the Clifford multiplication $c \colon \Gamma(T^*M) \to \Gamma(\End E)$ extends to a map of algebra bundles $c \colon Cl(M) \to \End(E)$. 
\end{rem}
This generalizes the notion of Dirac operator on Spin manifolds considerably, as we will see in an example below. First, let us establish that spin structures are examples of Clifford modules: 
\begin{expl}\label{expl:spinors}
Let $M$ be an even-dimensional Riemannian manifold with a Spin structure: $\dim M = n = 2k$. Then the bundle of spinors $S_n$ decomposes as $S_n = S_n^+ \oplus S_n^-$ and we defined a Clifford multiplication $c$ and a connection $\Tilde{\nabla}^g$ satisfying the axioms of a Clifford module and a Clifford connection. The only thing we have to show is that the following. 
\begin{prop} 
\begin{enumerate}[i)]
\item There exists a metric $h$ on $S_n$ which is compatible with the Clifford multiplication (i.e. Equation \eqref{eq:unitarity} is satisfied).
\item Given any such metric, the spin connection $\Tilde{\nabla^g}$ is compatible with it. 
\end{enumerate}
\end{prop}
\begin{proof}
First, choose an inner product $\langle\cdot,\cdot\rangle_{\Delta_n}$ on $\Delta_n$ which satisfies $\langle c(\alpha)v,c(\alpha)w\rangle$ for all $\alpha \in Cl_n$ and all $v,w \in \Delta_n$. We can produce such an inner product by averaging: Let $e_1,\ldots, e_n$ be an orthonormal basis of $\R^n$ and consider the finite subgroup $G \subset Cl_n^\times$ generated by $1,e_1,\ldots,e_n$. Then, choose any inner product $\langle\cdot,\cdot\rangle'$ on $\Delta_n$ and define 
$$\langle v,w \rangle_{\Delta_n} := \sum_{g\in G}\langle c(g)v, c(g)w \rangle'.$$
Then $\langle\cdot,\cdot\rangle$ is also an inner product and invariant under the action of $G$. Since the representation is linear, the inner product is invariant under the action on all unit vectors. Now, since we have an invariant inner product on $\Delta_n$, we can construct an invariant metric on $S_n$ by using a partition of unity subordinate to a trivializing cover. This proves i). For ii), notice that Spin group acts unitarily in such a metric. Since the spin connection on spinors is induced by a connection on the principal $\Spin_n$ bundle given by the spin structure, which acts unitarily, we conclude that the connection is compatible with the metric. 
\end{proof}
\end{expl}
\begin{rem}
This proof can be generalized to other Clifford bundles. Existence of compatible connections can be shown similarly (again constructing it locally). 
\end{rem}
Let us now consider another example, which shows that any manifold admits Clifford modules. 
\begin{expl}\label{expl:Tangent_Clifford}
Let $E = \bigwedge^\bullet TM \otimes \C$ denote the exterior algebra of the tangent bundle, with the hermitian metric $h$ given by the complexification of the metric $g$ extended to $\bigwedge T^*M$. It has an obivous super vector bundle structure 
$$ \bigwedge{}^{even}TM \oplus \bigwedge{}^{odd}TM. $$
For $\theta \in \Gamma(T^*M) = \Omega^1(T^*M), \omega \in \Omega^\bullet(T^*M) = \Gamma\left(\bigwedge^\bullet T^*M\right)$  we define the Clifford multiplication by 
\begin{equation}
c(\theta)\omega = \theta \wedge \omega  - \iota_\theta \omega,
\end{equation}
where the contraction $\iota_\theta \colon \Omega^k(M,\C) \to \Omega^{k-1}(M,\C)$ is defined\footnote{Here we put $\alpha$ in the first argument so that $\iota_\theta\omega$ is $\C$-linear in $\omega$ (it is only $\R$-linear in $\theta$).} by 
\begin{align} 
\iota_\theta f &= 0 \quad f \in C^\infty(M,\C) \\
\iota_\theta\alpha &= h(\alpha,\theta) \quad \alpha \in \Omega^1 \\
\iota_\theta (\tau \wedge \omega) &= (\iota_\theta \tau) \wedge \omega + (-1)^\tau \wedge  \iota_\theta \omega. \label{eq:contractionderivation}
\end{align}
Equation \eqref{eq:contractionderivation} says that the contraction is a degree 1 derivation. Equivalently, we can define it as contraction with the vector field $v=g(\theta,\cdot)$. \\
Next, the Levi-Civita connection $\nabla^g$ can be extended to the exterior algebra by defining 
\begin{align*}
\nabla^gf &= df, \\
\nabla^g(\theta_1 \wedge \theta_2) &= \nabla^g(\theta_1) \wedge \theta_2 + \theta_1 \wedge \nabla^g\theta_2
\end{align*}
i.e. by extending it to the de Rham algebra as a degree 0 derivation. By induction on the degree of $\omega$, one can prove that 
\begin{equation}
\nabla^g(\iota_\theta\omega) = \iota_{\nabla^g\theta}\omega + \iota_{\theta}\nabla^g\omega. \label{eq:nablaiota}
\end{equation}
Together with the fact that the Levi-Civita connection is a degree 0 derivation of the wedge product, this gives the compatibility of $\nabla^g$ with $c$. 
\begin{exc}
Check the claims made in this example in detail: Prove the Clifford relation,  Equation \eqref{eq:nablaiota}, and the compatibility of $\nabla$ with $c$.
\end{exc}

Thus, every manifold admits a Dirac bundle $(E,\nabla^g)$ given by picking a Riemannian metric. 
We will see later that the corresponding Dirac operator is given by $(d+d^*)$, whose square is the Hodge - de Rham Laplacian on differential forms.
\end{expl}
Associated to Clifford modules is a differential form called the \emph{twisting curvature} (we will see where terminology comes from in a minute). Recall the map $\rho\colon \Spin_n \to SO(n)$ and the associated isomorphism of Lie algebras
$$\rho_*^{-1}\colon so(n) \to spin_n \subset Cl_n$$
\begin{defn}
Denote $R^g = F^{\nabla^g}\in \Omega^2(so(n))$ and $R^E = c(\rho_*^{-1}(R^g))$. Then, the \emph{twisting curvature} $F^{E/S}$ of $(E,\nabla^E)$ is defined as 
\begin{equation}
F^{E/S} = F^{\nabla^E} - R^E \in \Omega^2(\End E)
\end{equation}
\end{defn}
\begin{expl}Suppose $M$ is spin and consider the Dirac bundle $(S_n, \Tilde{\nabla}^g)$. Then the twisting curvature $F^{S_n/S} = 0$, since $\Tilde{\nabla}^g = c(\rho_*^{-1})\nabla^g$ and hence $F^{\nabla^E} = R^E$. 
\end{expl}
\begin{prop}
Let $(E,\nabla^E)$ be a Dirac bundle and $(W, \nabla^W)$ be a hermitian super vector bundle with a metric connection $\nabla^W$ that preserves the $\Z_2$-grading. Then\footnote{Remember that $E \hat{\otimes} W$ denotes the graded tensor product.} $(E \hat{\otimes} W, \nabla^{E\hat{\otimes} W})$, 
where $$ \nabla^{E \hat{\otimes} W} = \nabla^E \otimes \mathrm{id}_W + \mathrm{id}_E \otimes \nabla^W)$$ is a Dirac bundle with the Clifford multiplication $c^{E \hat{\otimes} W}(\theta) = c^E(\theta) \otimes \mathrm{id}_W$. 
\end{prop}
\begin{proof}
The Clifford relation is obivously satisfied. The compatibility is an easy check: 
\begin{align*}
\nabla^{E\hat{\otimes} W} c^{E \hat{\otimes} W}(\theta)\psi^E \otimes \psi^W &= (\nabla^E \otimes \mathrm{id}_W + \mathrm{id}_E \otimes \nabla^W ) (c^E(\theta)\psi^E \otimes \psi^W) \\
&= \nabla^Ec^E(\theta)\psi^E \otimes \psi^W + \psi^E \otimes \nabla^W\psi^W \\ 
&= c^E(\nabla^g\theta)\psi^E \otimes \psi^W + c^E(\theta)\nabla^E\psi^E \otimes \psi^W + \psi^E \otimes \nabla^W\psi^W\\
&= c^{E \hat{\otimes} W}(\nabla^g\theta)\psi^E \otimes \psi^W +  c^{E \hat{\otimes} W}(\theta)\nabla^{E \hat{\otimes W}}(\psi^E \otimes \psi^W) 
\end{align*}
\end{proof}
\begin{defn}We call $(E \hat{\otimes} W, \nabla^{E\hat{\otimes} W})$ the \emph{twist of $E$ by $W$.}
\end{defn}
 From the definition of the twist it is immediate that 
 \begin{equation}
 F^{E\hat{\otimes}W /S} = F^{E /S} + F^W.
 \end{equation}
 \begin{rem}
 On a spin manifold, one can show that all Clifford bundles are twists of the bundles of the bundle of spinors. The twisting curvature of the Clifford bundle is then just the curvature of the twist, which explains the name.
 \end{rem}
 The important property of the twisting curvature is the following. 
 \begin{prop}\label{prop:twist_curv_commutes}
 For all $X,Y \in \Gamma(TM), \theta \in \Gamma(T^*M)$,  $F^{E/S}(X,Y)$ commutes with $c(\theta)$. 
 \end{prop}
 \begin{proof}
Note that the compatibility of the connection with the Clifford multiplication can be rewritten as 
 \begin{equation}
 [\nabla_X^E,c(\theta)] = c(\nabla_X^g\theta).\label{eq:cliff_compatible_comm}
 \end{equation}
 Using this we want to prove that $[F^{E/S}(X,Y),c(\theta)]$ vanishes (here all commutators are regular, not supercommutators). From the Jacobi identity and repeated use of \ref{eq:cliff_compatible_comm} , we obtain 
 \begin{align*}[[\nabla_X^E,\nabla_Y^E],c(\theta)] &=
 [[\nabla_X^E,c(\theta)],\nabla_Y^E] + [\nabla_X^E,[\nabla_Y^E,c(\theta)]] \\
 &= [c(\nabla^g_X\theta),\nabla_Y^E] + [\nabla_X^E,c(\nabla_Y^g\theta)]  \\
 &= -c(\nabla_Y^g\nabla_X^g\theta) + c(\nabla_X\nabla^Y\theta) = c(\nabla_{[X,Y]}^g\theta). 
 \end{align*} 
 Using $F^{\nabla^E}(X,Y) = [\nabla_X^E,\nabla_Y^E] - \nabla_{[X,Y]}^E$, we obtain 
 $$[F^E(X,Y),c(\theta) ] = c(R^g(X,Y)\theta).$$
 Now, since $F^{E/S} = F^{\nabla^E} - R^E$, we have to show that $[R^E(X,Y),c(\theta)] = c(R^g(X,Y)\theta)$. 
 Remember that $R^g(X,Y) = \frac12R^g_{ij}(X,Y)e^i\wedge e^l $ and hence $$\rho_*^{-1}R^g(X,Y) = \frac18R^g_{ij}(X,Y)[e^i,e^j] = \frac14R_{ij}e^ie^j.$$Now we use again the fact that we used in the proof of Proposition \ref{prop:spinconnection}: $e_ie_jv = ve_ie_j + 2(e_i \wedge e_j)(v)$ or $[e^ie^j,v] = 2(e^i \wedge e^j)(v)$ which implies 
 \begin{align*}
 [R^E(X,Y),c(\theta)] &= [c(\rho_*^{-1}(R^g(X,Y))),c(\theta)] = c([\rho_*^{-1}R^g(X,Y),\theta])  \\
 &= c\left[\frac14R^g_{ij}(X,Y)e^ie^j,\theta\right] = c\left(\frac12R^g_{ij}(X,Y)(e^i\wedge e^j)(\theta)\right) \\
 &= c(R^g(X,Y)\theta).
 \end{align*}
 Hence
 $$ [F^{\nabla^E}(X,Y) - R^E(X,Y),c(\theta)] = 0.$$
 \end{proof}
\chapter{Analysis}\label{ch:Analysis}
In this rather short chapter, we recall some of the analysis on manifolds that is necessary to state (and prove) the index theorem. 
\section{Elliptic Operators}
We are interested in the situation where there are two complex vector bundles $E$ and $F$ over a (compact) manifold $M$. In this situation we work in \emph{coordinate charts} $(U,\phi)$ of $M$ that are also trivializing charts for $E$ and $F$ (such charts always exist, simply intersect a coordinate chart with two trivializing charts for $E$ and $F$) such that the situation reduces to the diagram below: 
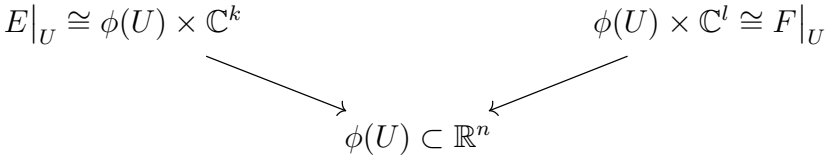
\[
\begin{tikzcd}
\restr{E}{U}\cong\phi(U)\times\mathbb{C}^k \arrow[rd] &                            & \phi(U)\times\mathbb{C}^l \cong \restr{F}{U} \arrow[ld] \\                                                   & \phi(U)\subset\mathbb{R}^n &                                                                                                           
\end{tikzcd} 
\]
We will call these charts ``good charts'' for the sake of brevity. Let us briefly discuss what this entails. We have coordinates $(x^1,\ldots,x^n,z^1,\ldots,z^k)$ on $\restr{E}{U} \cong \phi(U) \times \C^k \ni$. On the overlap of two good charts $U = (x^1,\ldots,x^n,z^1,\ldots,z^k)=(x,z)$ and $V = (y^1,\ldots,y^n,w^1,\ldots,w^k) = (y,w)$, we have 
\begin{subequations}
\begin{align}
y &= \phi_{UV}(x) \\
w &= g_{UV}^E(x)z
\end{align}
\end{subequations}
where $g_{UV}\colon \phi(U) \to GL_k(\C)$ are the transition functions\footnote{To match the notation with the previous chapters, this should be $g_{UV}^E(\phi^{-1}(x))$, but we suppress this for reasons of brevity.} of $E$. If $\sigma \in \Gamma(E)$, then $\restr{\sigma}{U}$ defines a map $\sigma_U\colon \phi(U) \to \mathbb{C}^k$ and we have 
\begin{equation}
\sigma_V(y) = \sigma_V(\phi_{UV}(x)) = g_{UV}(x)\sigma_U(x)
\end{equation}
It makes sense to take partial derivatives of local sections $\sigma_U\colon \phi(U) \to \C^k$, but they no longer define sections of $E$. In fact 
\begin{equation}
\frac{\de^{|I|}}{\de y^I} \sigma_V(y) = \sum_{|J| = |I|} \frac{\de x^J}{\de y^I}\frac{\de }{\de x^J}(g^E_{UV}(x))\sigma_U(x).
\end{equation}
Here, for multiindices $I = (I_1,\ldots,I_q), J= (j_1,\ldots,j_q)$, $\frac{\de x^J}{\de y^I}$ is short for 
$$ \frac{\de x^{j_1}}{\de y^{i_1}} \cdot  \frac{\de x^{j_2}}{\de y^{i_2}}\cdots  \frac{\de x^{j_q}}{\de y^{i_q}}.$$
After these preliminary discussion, we proceed with the definition of a partial differential operator between vector bundles. 
\subsection{Definitions}
\begin{defn}
Let $M,E,F$ as before. A linear map $D\colon\Gamma(E) \to \Gamma(F)$ is called a  \emph{(partial) differential operator} (pdo for short) if, for all $x \in M$, there is a good chart $U \ni x$ and some $m \in \N$ such that 
\begin{equation}
(L\sigma)_U = \sum_{|I| \leq m}A^I_U(x)\frac{\de^{|I|}}{\de x^I}\sigma_U(x),\label{eq:def_pdo}
\end{equation}
where $A^I_U(x) \colon \C^k \to \C^l$ are linear maps called the \emph{coefficients} of $D$. The smallest $m$ such that $D$ has the form \eqref{eq:def_pdo} is called the \emph{order} of $M$ in $U$.
\end{defn}
\begin{rem}
If $D$ has the form \eqref{eq:def_pdo} in a single good chart $U$, then also in all other good charts $V$ that intersect $U$. Namely, we have 
\begin{align}
(L\sigma)_V(y) &= g_{UV}^F(x)(L\sigma)_U = \label{eq:coord_trafo_pdo}
\end{align}
and because of the product rule, this is again of the form \eqref{eq:def_pdo}. However, it should be emphasized that the coefficients $A^I_U$ in general do \emph{not} transform as tensors. 
\end{rem}
\begin{rem}
The computation above shows that the order of $D$ is independent of the chart $U$ used in the definition, and hence locally constant. If $M$ has several connected components, we require that the order of $D$ is the same on every connected component. In this way we can speak of the order of $D$ on $M$. 
\end{rem}
Next, let us look at some examples that appear naturally on manifolds. 
\begin{expl}
A pdo of order 0 is the same as a vector bundle morphism $T \in \underline{Hom}(E,F)$. 
\end{expl}
\begin{expl}
Let $E = F = \bigwedge{}^\bullet(TM) \otimes \C$ and $d$ the de Rham differential on $M$. The good charts are just the coordinates charts\footnote{In this case it makes sense to trivialize the exterior algebra bundle as $\bigwedge{}^\bullet\C^n$, and not explicitly as $\C^{{n \choose 2}}$, but of course this is irrelevant for the discussions above.} $U = (x^1,\ldots, x^n)$, in which the de Rham differential looks like 
\begin{equation}
d = \sum_{i=1}^n (dx^i \wedge ) \frac{\de }{\de x^i}
\end{equation}
where $dx^i \colon \bigwedge{}^\bullet\C^n \to \bigwedge{}^\bullet\C^n$ is the linear map given by (wedge) multiplying with $dx^i$. We conclude that $d$ is a pdo of order 1.
\end{expl}
\begin{expl}
Let $E$ be any (complex) vector bundle over $M$ and $\nabla \colon \Gamma(E) \to \Gamma(T^*M \otimes E)$ be a connection. On a good chart $U$, we have $(\nabla\sigma)_U = d\sigma_U + A_U\sigma_U$, where $A = A_idx^i\in \Omega^1(U,\End \restr{E}{U}) \cong \Omega^1(U,\End \C^k)$. We can rewrite this as 
\begin{equation}
(\nabla\sigma)_U =  \sum_{i = 1}(dx^i \otimes) \frac{\de}{\de x^i} \sigma_U + \sum_{i=1}^n (dx^i \otimes) A_i (\sigma_U)
\end{equation}
Here the first term is the order 1 part. $dx^i \otimes \C^k \to \C^n \otimes C^k$ is the linear map sending $v \to dx^i \otimes v$. The second term is an order 0 term and is given by the linear map $\C^k \to \C^n \otimes \C^k, v \mapsto \sum_i dx^i \otimes A_i(v)$. Notice that the transformation rule for the connection 1-forms $A$ is precisely the transformation for the order 0 part of an order 1 pdo. 
\end{expl}
\begin{rem}
If $D_1 \colon \Gamma(E) \to \Gamma(F), D_2 \colon \Gamma(F) \to \Gamma(E)$ are pdos of orders $m_1$ and $m_2$ respectively, then their compositon $D= D_2D_1$ is also a pdo and its order $m$ is bounded by the sum $m \leq m_1 + m_2$. 
\end{rem}
\begin{expl}
If $(E, \nabla^E)$ is a Dirac bundle over $M$, then $D_E \colon \Gamma(E) \to \Gamma(E)$ is a pdo of order at most 1 since it is the composition of the two pdos $\nabla^E$ and $c^E$. 
\end{expl}
\subsection{The symbol of a partial differential operator}
Let $D$ be a partial operator of order $m$, 
$$\restr{L}{U} = \sum_{|I| \leq M} A_U^I(x) \de^I_x.$$
Then we define a map $$\restr{\sigma(L)(\cdot)}{U} \colon \Gamma(\restr{T^*M}{U}) \to \Hom(\Gamma(\restr{E}{U}),\Gamma(\restr{F}{U}))$$
by setting, for $\xi = \xi_idx^i$, 
\begin{equation}
\restr{\sigma(L)(\xi)}{U} := \ii^m\sum_{|I|=m}A^I_U(x)\xi_I,
\end{equation}
where $\xi_I = \xi_{i_1}\cdots\xi_{i_n}$. 
We have the following important claim. 
\begin{lem}
The collection $\restr{\sigma(L)(\cdot)}{U}$ defines a global section $$\sigma(L) \in \underline{\Hom}(\mathrm{Sym}^mT^*m \otimes E, F).$$
\end{lem}
\begin{proof}
From equation \eqref{eq:coord_trafo_pdo}, we gather that 
\begin{align*}
 g_{UV}^F(x)\sum_{|I| \leq m} A^I_U(x)\frac{\de^{|I|}}{\de x^I}\sigma_U(x) 
&= g_{UV}^F(x)\sum_{|I| \leq m}\sum_{|J| = |I|}A^I_U(x)\frac{\de y^J}{\de x^I}\frac{\de^{|J|}}{\de y^J}(g^E_{VU}(y)\sigma_V(y)) \\
&= \sum_{|I| = m}g_{UV}^F(x)\sum_{|J| = |I|}A^I_U(x)\frac{\de y^J}{\de x^I}(g^E_{UV}(x))^{-1}\frac{\de^{|J|}}{\de y^J}\sigma_V(y)  + \sum_{|I| < M} (\cdots),
\end{align*}
from which we conclude that the coefficients $A_U^I$, for $|I| = m$ transform as a section of the bundle $\mathrm{Sym^m TM \otimes E^* \otimes F} \cong \underline{Hom}(Sym^m T^*M \otimes E, F)$.
\end{proof}
Concretley, given a section $\xi \in \Gamma(T^*M)$ we obtain a vector bundle morphism $\sigma(L)(\xi) \in \underline{Hom}(E,F)$, and the dependence on $\xi$ is polynomial of order $m$. 
\begin{defn}
$\sigma(L)$ is called the symbol of $L$. 
\end{defn}
\begin{lem}
We have $\sigma(D_1D_2) = \sigma(D_1)\sigma(D_2)$.
\end{lem}
\begin{rem}
Often, one is interested in inverting differential operators. Formally, the observation above implies that the symbol of the inverse should be the inverse of the symbol. This motivates the generalization of the symbol to a much larger class of operators, the so-called pseudo-differential operators, where we allow symbol to live in a larger class of functions than polynomials. See e.g. \cite{} for an introduction to pseudodifferntial operators. 
\end{rem}
A particularly nice class of differential operators are given by elliptic operators.
\begin{defn}
A differential operator is called elliptic if 
$$\sigma(L)(\xi)_p \colon E_p \to F_p $$
is an isomorphism whenever $\xi_p \neq 0$.
\end{defn}
\begin{expl} The de Rham differential $d$ is not elliptic: Its symbol is given (locally) by $\sigma(d)(\xi) = \ii\sum_i (dx^i \wedge)\xi_i$
but the maps $dx^i \colon \bigwedge{}^\bullet\C^n \to \bigwedge{}^\bullet\C^n$ are not invertible. 
\end{expl}
\subsection{Formal adjoints}
We now suppose that $M$ is Riemannian with metric $g$ and that $E,F$ carries hermitian metrics $\langle \cdot,\cdot\rangle_E,\langle \cdot,\cdot\rangle_F$. This defines an inner product on sections of $E$ and $F$: For instance for $u,v \in \Gamma(E)$ we define 
\begin{equation}
(u,v)_{L^2,E} = \int_M \langle u(x),v(x) \rangle_E dvol_g
\end{equation}
where $dvol_g$ is the volume form associated to the Riemannian metric. The spaces of smooth sections $\Gamma(E),\Gamma(F)$ are not complete with respect to this inner product, but one can complete them to obtain Hilbert spaces. We will not do this and work with formal adjoints instead
\begin{defn}
Let $L \colon\Gamma(E) \to \Gamma(F)$ be a pdo. A pdo $L^*\colon\Gamma(F) \to \Gamma(E)$ is called \emph{formal adjoint} of $L$ if, for all $u\in\Gamma(E)$ and $v \in \Gamma(F)$ we have 
\begin{equation}
\int_M \langle Lu,v \rangle_F dvol_g = \int_M \langle u, L^*v \rangle_E dvol_g.
\end{equation}
\end{defn}
We list without proof the following facts.
\begin{prop}Let $L,L_1,L_2$ be  pdo's.
\begin{enumerate}[i)]
\item The formal adjoint of $L$ always exists and is unique. 
\item $L = (L^*)^*$.
\item $(L_1L_2)^* = L_2^*L_1^*$ if the composition makes sense. 
\item $\sigma(L)(\xi)^* = \sigma(L^*)(\xi)$ for all $\xi \in \Gamma(T^*M)$.
\end{enumerate}
\end{prop}
\begin{expl}
Consider the de Rham differential $d \colon \bigwedge^kTM \to \bigwedge^{k+1}TM$ (with metrics induced by $g$) and let 
$d^* = (-1)^{nk + n + 1} * d *$, where $* \colon \Omega^{k}(M) \to \Omega^{n-k}(M)$ is the Hodge star associated to $g$. Then $d^*$ is the formal adjoint of $d$. To see this, remember that 
$$\int_M\langle\tau,\omega\rangle_{\bullet^kTM} dvol_g = \int_M \tau \wedge *\omega$$
(and in fact the Hodge star can be defined this way) and that $*^2 = \pm 1$. Then, we compute 
\begin{align*}
\int_M\langle d\tau,\omega\rangle_{\bigwedge^{k+1}TM} &= \int_M d\tau \wedge *\omega = \pm \int_M \tau \wedge d*\omega = \pm \int_M \tau \wedge * * d * \omega\\
&=\pm \int_M \langle \tau , *d*\omega\rangle_{\bigwedge^kTM} = \int_M\langle\tau,d^*\omega\rangle_{\bigwedge^kTM}.
\end{align*}
\end{expl}
\begin{expl}
Let $E$ be a hermitian vector bundle and let $\nabla$ be a metric connection. Let $F = T^*M \otimes E$ with the product metric. 
Then $$\nabla = dx^i \otimes \left(\frac{\de }{\de x^i} + A_i\right)$$ and hence 
$$\nabla^* =\sum_i \left(\frac{\de }{\de x^i} + A_i\right)^*(dx^i)^* = \sum \left(-\frac{\de }{\de x^i}-A_i \right)\iota_{dx^i}$$
where we used that $(dx^i)^* = \iota_{dx^i}$ and $\frac{ \de}{\de x^i}^* = -\frac{\de}{\de x^i}$. Also, $A_i^* = - A_i$ because the connection is metric (and hence the connection 1-forms take values in $u(n)$). 
\end{expl}
\begin{defn}
If $E$ is a hermitian vector bundle with connection, then $\nabla^*\nabla \colon \Gamma(E) \to \Gamma(E)$ is called the connection Laplacian of $E$. 
\end{defn}
The symbol of such an operator can easily be computed:
\begin{prop}
$\sigma(\nabla^*\nabla)(\xi) = g(\xi,\xi)\otimes \mathrm{id}_E$.
\end{prop}
In particular, the connection Laplacian is always elliptic. 
\begin{proof}
We have 
\begin{align*}
\sigma(\nabla^*\nabla)(\xi)_p &= \sigma(\nabla^*)(\xi)_p\sigma(\nabla)(\xi)_p \\
&= \left(-\ii \sum_i\iota_{dx^i}\xi_i\right)\left(\ii\sum_j (dx^j \otimes ) \xi_j \right) \\
&= \sum_{i,j} g(dx^i,dx^j) \otimes \mathrm{id_E}  \\
&= g(\xi,\xi) \otimes \mathrm{id}_E 
\end{align*}
\end{proof}
This motivates the following definition. 
\begin{defn}
A pdo $\Delta$ on $E$ is called \emph{generalized Laplacian} if $\sigma(\Delta)(\xi) = g(\xi,\xi)\otimes \mathrm{id}_E$.
\end{defn}
\section{Fredholm index} 
An important fact about elliptic operators is the following theorem. 
\begin{thm}
Let $L \colon \Gamma(E) \to \Gamma(F)$ be an elliptic p.d.o. Then $\ker L$ is finite-dimensional. 
\end{thm}
Also, we have the following central fact about elliptic operators: 
\begin{thm}[Fredholm alternative]
Let $L$ be an elliptic operator. Then 
\begin{align*}
\mathrm{Im} L = (\ker L^*)^\perp \\
\mathrm{Im} L^* = (\ker L)^\perp 
\end{align*}
\end{thm}
A direct corollary is the following.
\begin{cor}
Let $L$ be an elliptic operator. Then both $\ker L$ and $\ker L^*$  are finite dimensional and we have 
$\dim \ker L^* = \dim \mathrm{coker} L$.
\end{cor}
\begin{defn}
Let $L$ be an elliptic operator. The \emph{Fredholm index} of $L$ is 
\begin{equation}
\mathrm{ind}_L = \dim \ker L - \dim \mathrm{coker} L = \dim \ker L - \dim \ker L^*.
\end{equation}
\end{defn}
\begin{thm}
Let $L_t, 0\leq t \leq 1$ be a continuous\footnote{In the topology given by the sum of the $L^2$ norms of $L$ and $L^*$.} path of elliptic operators. Then the index of $L_t$ is constant, i.e. 
\begin{equation}
\mathrm{ind}(L_0) = \mathrm{ind}(L_1)
\end{equation}
\end{thm}
Remember that this is not true for the dimension of the kernel or cokernel as such. 
\section{Dirac Operators}
Let us summarize some of the analytic properties of Dirac operators. Let $(E,\nabla^E)$ be a Dirac bundle and $D_E\colon \Gamma(E) \to \Gamma(E)$ be a Dirac operator. 
\begin{prop}
The symbol of the Dirac operator is the Clifford multiplication: 
\begin{equation}
\sigma(D_E)(\xi) = ic(\xi)\colon \Gamma(E) \to \Gamma(E)
\end{equation}
\end{prop}
\begin{proof}
We have 
$$\sigma(D_E) = \sigma(c)\sigma(\nabla_E) = ic(\xi_idx^i) = ic(\xi)$$
\end{proof}
\begin{cor}
$D_E$ is elliptic and $D_E^2$ is a generalized Laplacian. 
\end{cor}
\begin{proof}
This follows from the fact that $V \subset Cl(V,g)^\times$ for all vector spaces with non-degenerate bilinear forms $g$. The fact that $D_E^2$ is a generalized Lapacian follows from the Clifford relation. 
\end{proof}
\begin{prop}
The Dirac operator is formally self-adjoint, that is, 
\begin{equation}
\int_M \langle \psi, D\psi \rangle dvol_g = \int_M \langle D\psi, \psi \rangle dvol_g\label{eq:Dselfadjoint} \end{equation}
\end{prop}
\begin{proof}
This follows from the compatibility of the spin connection with the metric and the fact that the Clifford multiplication is unitary.
\end{proof}
\begin{cor}
For the operators $D_E^\pm \colon \Gamma(E^\pm) \to \Gamma(E^\mp)$ we have $(D_E^+)^* = D_E^-$.
\end{cor}
\begin{proof}
This follows from the fact that $E = E^+ \oplus E^-$ is an orthogonal sum by expanding \eqref{eq:Dselfadjoint} into components. 
\end{proof}
A central identity in the proof of the index theorem is the Weitzenböck formula (sometimes called also Lichnerowicz formula) for the square of the Dirac operator that we state here (see e.g. \cite{Nicolaescu2013} for a proof). 
\begin{thm}\label{thm:Weitzenboeck}
Let $(E,\nabla^E)$ be a Dirac bundle on a Riemannian manifold $(M,g)$. Denote the scalar curvature of $g$ by $r(g)$ and define, in a local orthonormal frame $e_i$,
$$c(F^{E/S}) = F^{E/S}(e_i,e_j)c(e^i)c(e^j) \in \Gamma(\End E).$$
Then we have 
\begin{equation}
D_E^2 = (\nabla^E)^*\nabla^E + \frac{r(g)}{4} + c(F^{E/S}).
\end{equation}
\end{thm}
Rememember from Proposition \ref{prop:twist_curv_commutes} that $F^{E/S}(e_i,e_j)$ commutes with $c(e^i)c(e^j)$, so that the order in which we define $c(F^{E/S})$ does not matter. 
\chapter{The index theorem and its applications}\label{ch:Index}
In this central section we state the index theorem, provide some applications and give the idea of the proof. 
\section{Index theorem for spin Dirac operators}
Suppose $(M,g)$ is an even-dimensional spin manifold. Pick a spin structure on $M$ and define accordingly the spinor bundle $S_n$, the spin connection $\tilde{\nabla^g}$ and the Dirac operator $D \colon \Gamma(S_n) \to  \Gamma(S_n)$. We also recall the definition of the $\hat{A}$-genus 
\begin{equation}
\hat{A}(M,g) = \det{}^{1/2}\left(\frac{\frac{i}{4\pi}R^g}{\sinh\left(\frac{i}{4\pi}R^g\right)}\right) = \det{}^{1/2}\hat{A}(iR^g/2\pi) \in \bigoplus_{k}\Omega^{4k}(M)\end{equation}
where $R^g$ is the Riemannian curvature and $\hat{A}(x) = \frac{x/2}{\sinh(x/2)}$.
The index of $D^+\colon\Gamma(S_n^+)\to \Gamma(S_n^-)$ is given by the famous Atiyah-Singer index theorem: 
\begin{thm}[Atiyah-Singer]
\begin{equation}
\mathrm{ind}(D^+) = \int_M \hat{A}(M,g).
\end{equation}
\end{thm}
A few immediate remarks are in order.
\begin{rem}
\begin{itemize}
\item Notice that we always have $\mathrm{ind} D = 0$ (since $D$ is formally self-adjoint).
\item The theorem shows that $\ind D^+ = 0$ if $ 4\nmid \dim M$ (since the $\hat{A}$-genus is concentrated in degrees divisible by 4).
\item Expanding $\hat{A}(M,g)$ in degrees one obtains 
$$\hat{A}(M,g) = 1 -\frac{1}{14}p_1(M,g) + \ldots,$$ 
where 
$p_1(M,g)$ is the first Pontryagin form is given by 
$$p_1(M,g) = \frac{1}{8\pi^2}\mathrm{tr}R^g \wedge R^g.$$
Hence the index theorem implies that the first Pontryagin number 
$$p_1(M) = \int_M p_1(M,g)$$
is divisible by $24$. 
\end{itemize}
\end{rem}
\section{Index theorem for Clifford Modules}
The index theorem has an extension to Clifford modules which has very interesting applications to topology. To state it, we have to introduce the relative Chern character of a Clifford module. 
\begin{defn}
Let $V$ be a representation of the Clifford algebra $Cl_n^c$, and let $F \in \End_{Cl_n^c} V$ (a linear map that commutes with the Clifford action).  Then we define the \emph{relative supertrace of $F$} to be 
\begin{equation}
\mathrm{str}^{E/S} = \frac{1}{2^{n/2}}\mathrm{str}(c(\omega_\C)F)
\end{equation}
\end{defn}
\begin{rem}
One can show that any Clifford module is of the form $V = \Delta_n \otimes W$, where the Clifford action is trivial on $W$. A map that commutes with the Clifford action is then just a linear map on $W$, and the relative supertrace is the supertrace of that map. 
\end{rem}
Let $(E,\nabla^E)$ be a Clifford module with twisting curvature $F^{E/S}$. 
\begin{defn}
We define the \emph{relative Chern character} $\mathbf{ch}_{E/S}$ by 
\begin{equation}
\mathbf{ch}_{E/S} = \mathrm{str}^{E/S} \exp \frac{i}{2\pi}F^{E/S}
\end{equation}
\end{defn}
We can now state the index theorem for general Clifford modules. 
\begin{thm}[Atiyah-Singer]
\begin{equation}
\ind D_E^+ = \int_M \hat{A}(M,g) \mathbf{ch}_{E/S}
\end{equation}
\end{thm}
This index theorem has far-reaching implications in topology, and unites a number of seemingly very different-looking results. The strategy is as follows. One constructs a Clifford module $E$ such that the index of the associated $D_E^+$ is an invariant of $M$ (or maybe some extra structure associated with $M$). The index theorem provides a \emph{local} expression for that invariant in terms of characteristic classes of $M$. An important application is discussed in the following section. 
\section{Chern-Gauss-Bonnet theorem}
The goal of this section is to prove the Chern-Gauss-Bonnet theorem. Let us first introduce the \emph{dramatis personae}. 
\subsection{Definitions}
\begin{defn}
Let $M$  be a manifold. Then the \emph{Euler characteristic\footnote{The Euler characteristic can be defined for much more general classes of topological spaces, but it is not necessary for our discussion.} of $M$} is  
\begin{equation}
\chi(M) = \sum_{i = 0}^{\dim M} (-1)^ib_i = \sum_{i = 0}^{\dim M} (-1)^i\dim H^i(M,\R).
\end{equation}
\end{defn}
For the next definition we need the concept of Pfaffian. 
\begin{defn}
Let $A = A_{ij}$ be an antisymmetric $n \times n$ matrix for some even $n=2k$. Define $\omega_A = \frac12 A_{ij}e_i \wedge e_j$. Then the \emph{Pfaffian} of $A$ is defined by 
\begin{equation}
\Pf(A) e_1 \wedge  \cdots \wedge e_{2k} = \frac{1}{n!}\omega_A^n
\end{equation}
\end{defn}
The Pfaffian satisfies $\Pf(A)^2 = \det A$ and $\Pf(\lambda A) = \lambda^k\Pf(A)$. 
\begin{defn}[Euler form]
Let $(M,g)$ be a Riemannian manifold of dimension $2k$. Then we define the \emph{Euler form} by 
\begin{equation}
e(M,g) = \Pf\left(\frac{1}{2\pi}R^g\right)
\end{equation}
\end{defn}
We can now state the Chern-Gauss-Bonnet theorem. 
\begin{thm}
Let $(M,g)$ be an even-dimensional Riemannian manifold. Then
\begin{equation}
\chi(M) = \int_M e(M,g). 
\end{equation}
\end{thm}
For odd-dimensional manifolds, the Euler characteristic is always zero as a consequence of Poincar\'e duality. 
In the following subsections we want to prove that this is the consequence of the Atiyah-Singer index theorem. 
\subsection{The Clifford module and its index}
The relevant Clifford module for this application is the one we met in Example \ref{expl:Tangent_Clifford}. We repeat here the main points. Let $E = \bigwedge{}^\bullet TM \otimes \C = \bigwedge{}^{even}TM \otimes \bigwedge{}^{odd}TM$ with Clifford multiplication 
$c(\theta) \omega = \theta \wedge \omega - \iota_\theta\omega$. The spin connection is the lift of the Levi-Civita connection. 
\begin{clm}
The Dirac operator of this Clifford module is 
\begin{equation}
D_E = d + d^*.
\end{equation}
\end{clm}
\begin{proof}
Again, this is best seen in a local orthonormal frame $e_1, \ldots, e_n$ with coframe $e^1,\ldots,e^n$. Then we can express $d = e^i \wedge \nabla_{e_i}$. We can then see that its formal adjoint is given by $-\nabla_{e_i} \iota_{e_i}$.
\end{proof}
Now that we know the Dirac operator, we proceed to compute its index. 
\begin{clm}
The index of the chiral Dirac operator $D_E^+$ is 
\begin{equation}
\ind D_E^+ = \chi(M).
\end{equation}
\end{clm}
\begin{proof}
For the proof we use a little Hodge theory. 
Define $\Delta = (d + d^*)^2 = dd^* + d^*d$. Then, the Hodge theorem states that 
\begin{equation}
\Omega^k(M) = \ker \restr{\Delta}{\Omega^k(M)} \oplus d\Omega^{k-1} \oplus d^*\Omega^{k-1}(M)
\end{equation}
and $$\ker \restr{\Delta}{\Omega^k(M)} \cong H^k(M,\R).$$
Next, notice that $\ker d + d^* = \ker d \cap \ker d^* = \Delta_k \cong H^k(M,\R)$ (this is an easy exercise using the formal adjointness of $d$ and $d^*$). Using this we compute the index of the chiral Dirac operator: 
\begin{align*}
\ind D_E^+ &= \dim \ker D_E^+ - \dim \ker D_E^- \\
&= \dim \ker \restr{d+d^*}{\Omega^{even}(M)} -\dim\ker \restr{d+d^*}{\Omega^{odd}(M)} \\
&= \sum_{k \text{ even} }\dim H^k(M) - \sum_{ k \text{ odd}} H^k(M) \\
&= \chi(M).
\end{align*}

\end{proof}
In particular, the index of $D_E^+$ is zero on odd-dimensional manifolds.
\subsection{The relative Chern character}
We now restrict to the case $n=2k$ even. 
The main ingredient in the proof of the Chern-Gauss-Bonnet theorem is the computation of the relative Chern character. We divide this computation into several steps. We start with the definition of another action of $T^*M$ on $\Omega^k(M)$.
\begin{defn}
Let $\theta \in \Gamma(T^*M)$, then we define 
\begin{equation}
\tilde{c}(\theta) \omega = \theta \wedge \omega + \iota_\theta \omega. 
\end{equation}
\end{defn}
We state some of the properties of this new action.
\begin{prop}\label{prop:tilde_c_properties}
Let $e_i$ be a local orthonormal frame, then 
$\tilde{c}$ satisfies
\begin{align}
\{c(e^i),\tilde{c}(e^j)\}&=0\\
\{\tilde{c}(e^i),\tilde{c}(e^j)\} &= +2\delta_{ij}
\end{align}
(note the difference in sign to usual Clifford multiplication).
\end{prop}
\begin{exc}
Prove this proposition. 
\end{exc}
Now we investigate the twisting curvature of this Clifford bundle. 
\begin{prop}
Let $x \in M$ and $e_i$ a synchronous orthonormal frame at $x$, i.e. $(\nabla^g e_i)_x = 0$, with dual frame $e^i$. 
Let $R_{ijkl} = g(e_i,R^g(e_k,e_l)e_j)$. Then at $x$ we have
\begin{equation}
F^{E/S}(e_k,e_l) = -\frac{1}{4}R_{ijkl}\tilde{c}(e^i)\tilde{c}(e^j).
\end{equation}
\end{prop}
\begin{proof}
For a multi-index $I= (i_1,i_2,\ldots,i_m)$, denote  $e^I:= e^{i_1}\wedge \cdots \wedge e^{i_m}$. Recall that $\nabla^E$ is the lift of $\nabla^g$ to $\wedge T^*M$ as a derivation of the wedge product and hence 
$$(\nabla_{e_j}e^I) =(( \nabla_{e_i}e^{i_1} ) \wedge e^{i_2} \wedge \cdots \wedge e^{i_m} + \ldots + e^{i_1} \wedge \cdots \wedge (\nabla_{e_j}e^{i_m}))_x = 0.$$
Hence, working again at $x$ we have 
\begin{align*}
F^{E}(e_k,e_l)e^I &= ([\nabla^E_{e_k},\nabla^E_{e_l}] -\nabla^E_{[e_k,e_l]})e^I \\
&=  [\nabla^E_{e_k},\nabla^E_{e_l}]e^I\\
&=\sum_{j=1}^m e^{i_1}\wedge \cdots\wedge [\nabla^g_{e_k},\nabla^g_{e_l}]e^{i_j}  \wedge \cdots \wedge e^{i_m}\\ &+
\sum_{j_1 \neq j_2 = 1}^m e^{i_1}\wedge \cdots \nabla^g_{e_k}e^{i_{j_1}}\wedge\cdots \wedge \nabla^g_{e_l}e^{i_{j_2}} \wedge \cdots \wedge e^{i_m} - (k \leftrightarrow l) 
\end{align*}
Since $\nabla_{e_k}e^i =0$ at $x$, the second sum vanishes and the first sum is equal to 
$$F^E(e_k,e_l)=\sum_{j=1}^m e^{i_1}\wedge \cdots \wedge R^g(e_k,e_l)e^{i_j}  \wedge \cdots \wedge e^{i_m} =- R_{ijkl}(e^i \wedge) \circ \iota_{e_j}e^I$$
(notice that $(e^i \wedge) \circ \iota_{e_j}$ is a degree 0 derivation, and Einstein summation convention is understood). The sign comes from the fact that $R^g(e_k,e_l)e^j = -R_{ijkl}e^i$. 
Now we express 
\begin{align*}
e^i \wedge &= \frac{1}{2}(c(e^i) + \tilde{c}(e^i)) \\
\iota_{e_i} &= -\frac{1}{2}(c(e^i) -\tilde{c}(e^i)) 
\end{align*}
so that the expression above becomes 
$$F^E(e_k,e_l) = -\frac14 R_{ijkl}(\tilde{c}(e^i)\tilde{c}(e^j) -c(e^i)c(e^j) + c(e^i)\tilde{c}(e^j) - \tilde{c}(e^i)c(e^j)).$$
The last two terms cancel by because $R_{ijkl} = - R_{jikl}$ and $\{c(e_i),\tilde{c}(e_j)\} = 0$. Then, notice that 
$R^E(e_k,e_l) = c(\rho_*^{-1}R^g(e_k,e_l)) = \frac12c(\rho_*^{-1}R_{ijkl}(e^i\wedge e^j)) = \frac14 R_{ijkl}c(e^i)c(e^j) $, 
which concludes the proof. 
\end{proof}
Having established the twisting curvature, we turn to the investigation of the relative supertrace.  Remember that $\mathrm{str} F = \mathrm{tr}\gamma F$, where $\gamma$ is the grading operator. We have the following important observation. 
\begin{clm}\label{clm:gradingoperators}
Let $\gamma$ be the grading operator on $\bigwedge T^*M$, then 
\begin{equation}
\gamma \circ c(\omega_\C) = \tilde{c}(\omega_\C). 
\end{equation}
\end{clm}
As a corollary, we have the following formula for the supertrace: 
\begin{cor}
\begin{equation}
\mathrm{str}^{E/S}(F) = \frac{1}{2^{n/2}}\mathrm{tr}(\tilde{c}(\omega_\C F)) =: \frac{1}{2^{n/2}}\mathrm{str}'F
\end{equation}
where we denote by $\mathrm{str}'$ the supertrace induced by the grading operator $\tilde{c}(\omega_\C)$. 
\end{cor}
\begin{proof}[Proof of Claim \ref{clm:gradingoperators}]
We prove the equivalent statement that $\gamma = (-1)^{\deg} = c(\omega_\C)\tilde{c}(\omega_\C)$. 
For even $n$, we have $\tilde{c}(\omega_\C)c(e_j) = c(e_j)\tilde{c}(\omega_\C)$ as a consequence of Proposition \ref{prop:tilde_c_properties} and $c(\omega_\C) c(e_j) = - c(\omega_\C)c(e_j)$ (this is Exercise \ref{exc:complex_volume}). 
Notice also that $c(\omega_\C)1 = \tilde{c}(\omega_\C)1$. Let $I = (i_1, \ldots, i_k)$ be a multi-index. Putting things together, we obtain 
\begin{align*}
c(\omega_\C)\tilde{c}(\omega_\C) e^I & = c(\omega_\C)\tilde{c}(\omega_\C) c(e^I)1  \\
&= (-1)^k c(e^I)c(\omega_\C)\tilde{c}(\omega_\C) 1 \\
&= (-1)^ke^I.
\end{align*}

\end{proof}
This allows to compute the relative supertrace of $\tilde{c}(e^I)$ for all monomials $e^I$: 
\begin{lem}\label{lem:supertraces} Let $I=(i_1,\ldots,i_k)$, then we have 
\begin{equation}\label{eq:supertraces}
\mathrm{str}^{E/S}\tilde{c}(e^I) = \begin{cases} 0 & k < n \\
2^{n/2}(-i)^{n/2} & I = (1, \ldots, n) 
\end{cases}
\end{equation}
\end{lem}
\begin{proof}
First, consider the case $k < n$. Then there is a $j \notin I$. The trick is to realize that we can write $\tilde{c}(e^I)$ as a supercommutator \emph{with respect to the grading induced by $\tilde{c}(\omega_\C)$}: 
$$\tilde{c}(e^I) = \frac12 [\tilde{c}(e_j),\tilde{c}(e_j)\tilde{c}(e^I)]'$$
where the prime denotes the fact that we are using the grading induced by  $\tilde{c}(\omega_\C)$. Hence $\mathrm{str} \tilde{c}(e^I) = 0$. 
For $I = (1, \ldots, n)$, we have 
$$\mathrm{str}^{E/S}\tilde{c}(e^1 \wedge \cdots \wedge e^n) =\frac{1}{(2i)^{n/2}}\mathrm{tr}(\underbrace{\tilde{c}(\omega_\C)\tilde{c}(\omega_\C)}_1) =  \frac{1}{(2i)^{n/2}}\dim E = (-2i)^{n/2}.$$ 
\end{proof}
\begin{lem} 
Let $A$ be an antisymmetric matrix on $\R^n$, with $n=2k$ even, then we have
\begin{equation}
\mathrm{str}^{E/S} \exp\left(\frac12 A_{ij} \tilde{c}(e^i)\tilde{c}(e^j)\right)  =\frac{\mathrm{Pf}(-2iA)}{\hat{A}(-2A)}
\end{equation}
\end{lem}
\begin{proof}
First one notices that the left hand side is invariant under the adjoint action of $O(n)$, thus we can bring $A$ into block diagonal form 
$$A' =  \begin{pmatrix}
A(\lambda_1) & 0 &\cdots & 0 \\
0 & A(\lambda_2)   & \cdots & 0 \\
\vdots  & \ddots & & \\
0 &  \cdots & & A(\lambda_{k}))
\end{pmatrix} $$with blocks of the form 
$$ A(\lambda) = \begin{pmatrix}
0 & \lambda_i \\
-\lambda_i & 0
\end{pmatrix}.$$
Using $\tilde{c}(e^i)\tilde{c}(e^j) = -\tilde{c}(e^j)\tilde{c}(e^i)$ it follows that $\frac12 A'_{ij}\tilde{c}(e^i)\tilde{c}(e^j) = \sum_{i=1}^k \lambda_i\tilde{c}(e^{2i-1})\tilde{c}(e^{2i})$.  Denoting $J_i = \tilde{c}(e^{2i-1})\tilde{c}(e^{2i})$, we notice that $[J_i,J_k] = 0$ and $J_i^2 = -1$. It follows that 
\begin{align*}
\mathrm{str}^{E/S} \exp\left(\frac12 A_{ij} \tilde{c}(e^i)\tilde{c}(e^j)\right) &= \mathrm{str}^{E/S} \exp\left(\sum_{i=1}^k \lambda_i\tilde{c}(e^{2i-1})\tilde{c}(e^{2i)}\right) \\
\textit{(since $[J_i,J_k]=0$) }&= \mathrm{str}^{E/S} \prod_{i=1}^k  \exp \lambda_i J_i \\
\textit{(since $J_i^2 = -1$) }&= \mathrm{str}^{E/S} \prod_{i=1}^k \cos\lambda_i + \sin \lambda_i J_i  \\
\textit{(By Eq.\eqref{eq:supertraces} )}&= (-2i)^{n/2} \prod_{i=1}^k \sin \lambda_i \\
&= (-2i)^{n/2} \prod_{i=1}^k (-i)\sinh (i\lambda_i)\\
&= (-2i)^{n/2} \underbrace{\left(\prod_{j=1}^k \frac{\sinh i\lambda_j}{i\lambda_j}\right)}_{\det{}^{-1/2}\hat{A}(-2A)}\left(\underbrace{\prod_{j=1}^k\lambda_j}_{\mathrm{Pf}(A)} \right)\\
&= \frac{\mathrm{Pf}(-2iA)}{\det{}^{1/2}(\hat{A}(-2A))}
\end{align*}
\end{proof}
Now the Chern-Gauss-Bonnet theorem follows by letting $A = -\frac{i}{4\pi}R_g(x)$: 
\begin{align*}
\chi(M) &= \int_M \hat{A}(M,g)ch_{E/S} \\
&= \int_M \hat{A}(M,g)\mathrm{str}^{E/S}\exp\left(\frac{\ii}{2\pi}F^{E/S}\right) \\
&= \int_M \mathrm{Pf}(-R_g/(2\pi) = \int_M e(M,g). 
\end{align*}
\section{On the heat kernel proof of the index theorem}
Probably the ``neatest'' proof of the index theorem was given by E. Getzler in \cite{}. We will try to explain the idea of this proof in the following, but we will not present all the analytical details. The main idea of the proof is to analyze the behavior of 
\begin{equation}
f(t) = \mathrm{str}\left(e^{-tD_E^2}\right)
\end{equation}
and establish the three important properties: 
\begin{enumerate}[i)]
\item $\lim_{t \to \infty}f(t) = \ind D_{E}^+$ 
\item $f$ is  independent of $t$, for $t \in (0,\infty)$. 
\item $\lim_{t\to 0} f(t) = \int_M \hat{A}(M,g)ch_{E/S}$.
\end{enumerate}

The main tool in the proof is the heat kernel $e^{-tD_E^2}$ of the Dirac operator that we will explain now. 
\subsection{Some properties of Heat kernels}
\subsubsection{Heat kernel on $\R$}
The heat equation on $\R$ with initial condition $f_0 \in C_c^\infty(\R)$ is 
\begin{equation}\label{eq:HeatEqR}
\begin{cases}
\frac{\de f}{\de t} - \frac{\de f}{\de x^2} &= 0 \\
\lim_{t\to 0} f(t,x) = f_0(x) 
\end{cases}
\end{equation} 
The heat kernel on $\R$ is by definition the fundamental solution of this differential equation, that is, a function 
$$ k\colon \R_{>0} \times \R \times \R \to \R $$ 
with the properties 
\begin{equation}\label{eq:HeatKernelR}
\begin{cases}
\de_t k(t,x,y) - \de_{x}^2 k(t,x,y) &= 0 \\
\lim_{t \to 0} k(t,x,y) &= \delta(x-y) 
\end{cases}
\end{equation}
Here the second requirement is to be understood in the distributional sense, i.e. is equivalent to 
$$\lim_{t \to 0} \int_\R k(t,x,y)f(y) dy = f(x)$$ 
for all $f \in C_c^{\infty}(X)$ and $x \in \R$. Given a function on $\R$ satisfying properties \eqref{eq:HeatKernelR}, a solution to the heat equation \eqref{eq:HeatEqR} can be easily constructed: 
$$f(t,x) = \int_R k(t,x,y) f_0(y) dy.$$ 

On $\R$, the heat kernel can be computed explicitly: 
\begin{prop} The heat kernel on $\R$ is given by 
\begin{equation}
k(t,x,y) = \frac{1}{\sqrt{4\pi t}}\exp\left(-\frac{(x-y)^2}{4t}\right)
\end{equation}
\end{prop}
The proof of this fact is left as an exercise. 
It is remarkable that the proof of point iii) discussed above relies on a similar explicit computation of a heat kernel. 
\subsubsection{Heat kernel on manifolds}
The setup of the heat equation can be vastly generalized. For us the following case will be important. Let $E \to M$ be a hermitian vector bundle over a compact Riemannian manifold $M$, and let $\Delta \colon \Gamma(E) \to \Gamma(E)$ be a generalized Laplacian which is self-adjoint in $L^2(E)$. 
\begin{defn} 
Let $\pi_1,\pi_2 M \times M \to M$ denote the projections to the two factors of $M$, the we define\footnote{This is usually called ``box tensor product'', mostly for want of a better name.}
\begin{equation}
\pi_1^*E \boxtimes \pi_2^*E = \pi_1^*E \otimes \pi_2^*E \to M \times M.
\end{equation}
\end{defn}
\begin{defn}
The space of \emph{smoothing operators} (also called Schwartz kernels) is the space of smooth sections of $E \boxtimes E^*$.
\end{defn}
To a section $s \in \Gamma(E \boxtimes E)$ we associate the operator $K_s\colon \Gamma(E) \to \Gamma(E)$ defined by 
\begin{equation}
(K_s\sigma)(x) = \int_M s(x,y)\sigma(y)dvol_g(y)
\end{equation}
Notice that even for $\sigma \in L^2(E)$, we have $K_s(\sigma) \in \Gamma(E)$, this explains the name smoothing operators.
We quote the following important theorem: 
\begin{thm}[Spectral theorem]
Let $\Delta$ be a self-adjoint generalized Laplacian on $E$. Then 
\begin{enumerate}[i)]
\item The spectrum of $\Delta$ is a discrete subset of $\R_{>0}$.
\item We have $\Delta = \sum_{\lambda \in spec(\Delta)} \lambda P_\lambda$ in the $L^2$-sense, where $P_\lambda$ denoted orthogonal projection to the eigenspace $E_\lambda := \ker(\Delta - \lambda)$. 
\end{enumerate}
\end{thm}
By elliptic regularity (see e.g. \cite{Evans2010}), eigenfunctions of $\Delta$ are smooth sections of  $E$ and the eigenspaces are finite-dimensional. Hence $P_\lambda$ has a Schwartz kernel given by 
$$P_\lambda = \sum_{i=1}^{\dim \ker E_\lambda} \psi_i \boxtimes \psi_i^*$$
where $\psi_i \in \Gamma(E)$ span $E_\lambda$ and $\psi_i^* = \langle\psi_i,\cdot\rangle_E \in \Gamma(E^*)$. 
\begin{defn}
For $f$ a measurable function on $\R$, define $$f( \Delta) = \sum_\lambda \in spec(\Delta) f(\lambda)P_\lambda$$
as an operator on $L^2(E)$. 
\end{defn}
This association is called functional calculus. Depending on the properties of the operator $\Delta$ and the function $f$, the resulting operator can have different analytic properties. We are interested in the following situation. 
\begin{prop}
Suppose $f \in C^{\infty}(\R)$ satisfies $\lim_{x \to \infty} x^k f(x) = 0$ for all $k > 0$, then 
$f(\Delta)$ has Schwartz kernel 
$$k_f = \sum_{\lambda \in spec(\Delta)}f(\lambda)P_\lambda.$$
\end{prop}
\begin{defn}
The \emph{heat kernel} of $\Delta\colon \Gamma(E) \to \Gamma(E)$ is
\begin{equation}
k(t,x,y) = \sum_{\lambda \in spec(\Delta)} e^{-t\lambda}P_\lambda
\end{equation}
\end{defn}
We quote without proof the following theorem from \cite{Berline2003}.
\begin{thm}\label{thm:HeatKernel}
\begin{enumerate}[i)]
\item The heat kernel is a smooth section of the bundle $\widehat{E \boxtimes E^*} \to \R_{>0} \times M \times M$ defined by 
$$\widehat{E \boxtimes E^*} = \pi^* (E \otimes E^*),$$ where $\pi \colon \R_{>0} \times M \times M \to M \times M$. 
\item The heat kernel is the unique smooth section satisfying 
\begin{equation}
\begin{cases}
\de_t k(t,x,y) + \Delta_x k(t,x,y) &= 0 \\
\lim_{t \to 0} k(t,x,y) &= \delta(x,y)
\end{cases}
\end{equation}
\item $\lim_{t \to \infty} k(t,x,y) = P_0 = P_{\ker \Delta}$. 
\end{enumerate}
\end{thm}
Finally, we quote the following theorem: 
\begin{thm}[Lidskii theorem]
Let $L\colon L^2(E) \to L^2(E)$ be trace class (i.e. $\tr L < \sum_{\lambda \in spec(L)} \lambda < \infty$) and be represented by the smooth integral kernel $l(x,y)$, then 
\begin{equation}
\tr L = \int_M \tr l(x,x) dvol_g(x)
\end{equation}
\end{thm}
\subsection{McKean-Singer formula}
We now start to prove points i)-iii) outlined above. The first two combine into the McKean-Singer formula 
\begin{equation}
\ind D_E^+ = \int_M \mathrm{str} k^E(t,x,x) dvol_g(x).
\end{equation}
Here $k^E$ denotes the heat kernel of $D_E^2$. 
This claim follows from two propositions. 
\begin{prop}
$\lim_{t \to \infty} \mathrm{str} k^E(t,x,x) = \ind D_E^+$
\end{prop} 
\begin{proof}
The crucial observation is that 
$$\mathrm{str}(e^{-tD_E^2})=\tr e^{-tD_E^-D_E+} - \tr e^{-tD_E^+D_E^-}$$
is the difference of two heat kernels. Hence, by point iii) of Theorem \ref{thm:HeatKernel}, we have 
$$\lim_{t \to \infty} \mathrm{str}k^E(t,x,y) = \dim \ker D_E^-D_E^+ - \dim \ker D_E^+D_E^-.$$
The claim follows if we can prove $\ker D_E^-D_E^+ = \ker D_E^+$. Similarly, we then have $\ker D_E^+D_E^- = \ker D_E^-$ and hence, using $D_E^- = (D_E^+)^*$ 
$\lim{t \to \infty} \mathrm{str}k^E(t,x,y) = \ind D_E^+$. 
To see that  $\ker D_E^-D_E^+ = \ker D_E^+$, notice that $\supseteq$ is clear. For the other inclusion, let $\sigma \in \ker D_E^-D_E^+$, then 
$$ 0  = \int_M \langle_E  D_E^-D_E^+\sigma, \sigma\rangle = \langle D_E^+\sigma, D_E^+\sigma \rangle $$ 
and hence $\sigma \in \ker D_E^+$. 
\end{proof}
\begin{prop}
For $t > 0$, $\mathrm{str}e^{-tD_E^2}$ is independent of $t$ (and equal to $\ind D_E^+$). 
\end{prop}
\begin{proof}
Again, we have 
$$\mathrm{str}(e^{-tD_E^2})=\tr e^{-tD_E^-D_E+} - \tr e^{-tD_E^+D_E^-} = \sum_{\lambda \in spec(D_E^-D_E^+)}e^{-t\lambda}\dim E_\lambda -\sum_{\lambda' \in spec(D_E^-D_E^+)}e^{-t\lambda'}\dim E_\lambda'. $$
The proof of the proposition now follows from the observations that $spec(D_E^-D_E^+) - \{0\} = spec(D_E^+D_E^-) -\{0\}$ and that $E_{\lambda}(D_E^-D_E^+) \cong E_{\lambda}(D_E^+D_E^-)$. Indeed, if $D_E^-D_E^+ \psi = \lambda \psi$ for some $\lambda \neq 0$, then 
$$(D_E^+D_E^-)D_E^+\psi =  D_E^+(D_E^-D_E^+ \psi) = \lambda D_E^\psi.$$
This shows $$\lambda \in spec(D_E^-D_E^+) -\{0\} \Rightarrow\lambda \in spec(D_E^+D_E^-) -\{0\},$$ and of course the other implication is shows exactly in the same way. For the second observation, notice that $D_E^+ \colon E_\lambda (D_E^-D_E^+)\to E_\lambda (D_E^+D_E^-)$ is invertible with inverse $\lambda^{-1}D_E^-$. 
\end{proof}
\subsection{Asymptotics of the heat kernel}
Step iii) in the proof outline above is considerably more involved than the first two steps. It was also establish later historically, the McKean-Singer formula originates from \cite{McKean1967} while the asymptotics of the heat kernel where first used in \cite{Patodi1971, Patodi1971a} to prove the local Index theorem. We follow here the proof by Getzler (\cite{Getzler1986}). Again, the plan of attack can roughly be divided into 3 steps: 
\begin{enumerate}[Step 1:]
\item Reduce to subsets of $\R^n$,
\item Rescaling, 
\item Applying explicit formulas.
\end{enumerate}
\subsubsection{Reducing to subsets of $\R^n$}
The first observation is that we have reduced the proof of the index theorem to completely local (even pointwise) statement. Thus we fix now a point $x_0 \in M$. Choosing an orthonormal basis $e_1,\ldots,e_n$ of $T_{x_0}M$, we obtain \emph{geodesic coordinates} $\phi(x^1,\ldots,x^n)$. in a neighbourhood $U = \exp(B_R(0))$ (here $\exp\colon T_{x_0}M \to M$ denotes the exponential map of the Levi-Civita connection). Thus we obtain the local coordinate frames $\de_i$ for the tangent bundle and $dx^i$ for the cotangent bundle.  Components in that frame will be denoted $i,j,\ldots$. In these coordinates we have 
\begin{equation}
g_{ij}(x) = \delta_{ij} + O(|x|^2). 
\end{equation} We also obtain local orthonormal frames for $TM$ (resp. $T^*M$) by parallel transporting the frames $e_a$ (resp $e^a$). Components in that frame will be denoted $a,b,\ldots$.  Shrinking $U$ if necessary, we assume that the bundle $E$ can be trivialized on $U$: 
\begin{equation}
\restr{E}{U} \cong \phi(U) \times E_{x_0} \cong \phi(U) \times \Delta_n \times W
\end{equation}
Here we use the fact that\footnote{This follows immediately from the fact that $\Delta_n$ is the unique irrep of $Cl_n^c$ and a little representation theory (see e.g. \cite{Etingof2011}). $W$ is just the multiplicity module of $\Delta_n$. } $E_{x_0}$ is a representation of $Cl_n^c$ and thus can be written as $E_{x_0} = \Delta_n \otimes W$, with trivial Clifford action on the twist $W$. This decomposition induces a decomposition $$\nabla^E = \tilde{\nabla}^g \otimes 1 + 1 \otimes \nabla^W$$ and we have 
\begin{equation}
F^{E/S} = F^W. 
\end{equation}
Over $U$, we can write $\tilde{\nabla}^g = d + \Gamma_idx^i$ and $\nabla^W = d + A_idx^i$. The compatibility of the Clifford connection and Clifford multiplication implies (see \cite{Atiyah1973})
\begin{align*}
\Gamma_i &= \frac12\Gamma_{iab}c(e^a)c(e^b) = -\frac12R_{ijab}(0)x^jc(e^a)c(e^b) + O(|x|^2) \\
A_i &= -F_{ij}^W(0)x^j + O(|x|)^2.
\end{align*}
We can now look at the heat kernel of the operator $D_E^2$ on $\phi(U)$, where it becomes a function 
$$k^E(t,x,y) \in C^{\infty}((0,\infty) \times \phi(U) \times \phi(U)) \otimes Cl_n^c \otimes W.$$
We can identify $Cl_n^c \cong \bigwedge\R^n \otimes \C$, this induces a non-commutative product $\circ$ on $\bigwedge \R^n$. The spinors can be embedded into the exterior algebra according to the discussion of Proposition \ref{prop:spinors} (in the even-dimensional case, but a similar story is possible also in the odd-dimensional case). Under this identification we can write 
\begin{equation}
k^E_t(x,0) = \sum_I a_I(t,x) c(e^I),
\end{equation}
where $a_I(t,x) \colon (0,\infty) \times \phi(U) \to \End W$ and $c$ is the Clifford action of the exterior algebra on itself defined by $c(e^i) = e^i \wedge  - \iota_{e_i}$. 
\subsubsection{Rescaling}
Central to the proof is the idea of rescaling: Instead of taking the limit $k^E(t,0,0)$ as $t\to 0$, we consider the limit of $\varepsilon \to 0$ of
$k^E(\varepsilon t, \varepsilon^{1/2} x, 0)$. The main realization of Getzler is that in this limit we can also rescale the Clifford action $c(e^i)$ to 
\begin{equation}
c_{\varepsilon}(e^i):=\varepsilon^{-1/2} e^i \wedge -  \varepsilon^{1/2}\iota_{e_i}
\end{equation}
and accordingly we obtain a rescaled product $\circ_\varepsilon$.
Thus, in the limit $\epsilon \to 0$, the product $\circ_\varepsilon$ gets dominated by the commutative product $\wedge$. 
We now give the definition of the rescaling: 
\begin{defn}
The rescaled heat kernel $$\delta_\varepsilon k^E \in C^{\infty}((0,\infty) \times \phi(U) \times \phi(U)) \otimes \bigwedge_\C\R^n \otimes W$$ is given by 
\begin{equation}
(\delta_{\varepsilon}k^E)(t,x,y) = \sum_I a_I(\varepsilon t, \varepsilon^{1/2}x)c_{\varepsilon}(e^I).
\end{equation}
\end{defn}
The next lemma says that we can exchange the limits of $t \to 0$ with the rescaling:
\begin{lem}
$$\lim_{t\to 0} \mathrm{str} k_t^E(0,0) = (-2i)^{n/2} [\lim_{\varepsilon \to 0} \varepsilon^{n/2}(\delta_\varepsilon k^E)(t,x)]_{(n)}$$
where $[\cdot]_{(n)}$ denotes the degree $n$ component in $\bigwedge_\C \R^n$. 
\end{lem}
\begin{proof}
On the one hand, we have 
$\lim_{t\to 0} \mathrm{str} k_t^E(0,0) = \lim_{t \to 0}(-2i)^{n/2}a_{(1,\ldots,n)}(t,0)$
since the supertraces of all other components vanish (similar to Lemma \ref{lem:supertraces}). On the other hand, in the limit as $\varepsilon \to 0$, the Clifford action approaches the exterior multiplication, so that the degree $n$ term becomes $\lim_{\varepsilon \to 0}a_{(1,\ldots,n)}(\varepsilon t, \varepsilon^{1/2}x)$.
\end{proof}
This is, in fact, the key ingredient of the proof. We will only sketch the proof of the remaing ingredients, refering to the literature \cite{Getzler1986, Nicolaescu2013, Dai2015} for more details. 
\begin{prop}
The rescaled heat kernel $\delta_\varepsilon k$ is the heat kernel of the rescaled Dirac operator 
$$D_\varepsilon^2 = \delta_\varepsilon D_E^2 \delta_\varepsilon^{-1}.$$
As $\varepsilon \to 0$, this Dirac operator approaches 
\begin{equation}
D_0 = -\sum \left(\de_i - \frac14 \Omega_{ij}x_j\right)^2 + F \wedge \label{eq:D0}
\end{equation}
where $\Omega_{ij} = \frac12 R_{ijab}e^a\wedge e^b$.\end{prop}
\begin{proof}
The first part is a simple computation (using uniqueness of the heat kernel). The second part follows from the Weitzenb\"ock or Lichnerowicz formula (Theorem \ref{thm:Weitzenboeck})
$$D_E^2 = (\nabla^E)^*(\nabla^E) + F^{E/S} + s(g)/4$$
from the explicit formulas before and the explicit action of the rescaling.
\end{proof}
Now one can make use of the fact that the operator \eqref{eq:D0} defines a generalized harmonic oscillator, for which there is an explicit formula for the heat kernel:
\begin{prop}[Mehler's formula]
The heat kernel for the operator \eqref{eq:D0}
is given by 
\begin{equation}
k^0(t,x,0) = (4\pi t)^{-1/2}\hat{A}(t\Omega/2)\exp\left(tF - \frac{1}{4t}\left(\frac{t\Omega/2}{\tanh t\Omega/2 }\right)_{ij}x^ix^j\right)
\end{equation}
\end{prop}
\begin{proof}
One considers first the one-dimensional case $H = -\frac{d^2}{dx^2} + ax^2$. Then, the heat kernel is given by 
\begin{equation}
k(t,x,y) = \frac{1}{(4\pi t)^{n/2}}\left(\frac{2at}{\sinh 2at}\right)^{1/2}\exp\left(-\frac{1}{4t}\frac{2at}{\sinh 2at}(\cosh(2at)(x^2+y^2) - 2xy)\right).
\end{equation}
Passing to the multi-dimensional case, one considers 
$$ H = -\sum\left(\de_i - \frac14 \Omega_{ij}x_j\right)^2 $$
and one can show (diagonalizing the action) that the heat kernel of $H$ satisfies 
$$k(t,x,0) = (4\pi t)^{-n/2}\hat{A}(t\Omega)\exp\left(-\frac{1}{4t}\left(\frac{t\Omega/2}{\tanh t\Omega/2 }\right)_{ij}x^ix^j\right).$$
Now the claim follows from the fact that $H$ and $F\wedge $ commute and that the heat kernel of $F$ is simply $\exp(tF)$. 
\end{proof}
Putting everything together\footnote{There is another subtlety: Does the convergence of operators $D_\varepsilon \to D_0$ imply the convergence of the heat kernels $\delta_{\varepsilon}k^E \to k^0$? Again, for a precise analysis we refer to Getzler's papers \cite{Getzler1983,Getzler1986}.}, the index theorem follows: The term of order $t^{n/2}$ in $\hat{A}(t\Omega)\exp(tF)$ is precisely $[\hat{A}(\Omega)\exp F]_{(n)}$. We then compute 
\begin{align*}
\ind D_E^+ &= \int_M \lim_{t \to 0} \mathrm{str} k^E(t,x,x) \\
&= \int_M \lim_{t \to 0} k^0(t,0,0)(x) \\
&= \int_M \hat{A}(M,g)ch_{E/S}
\end{align*}
\printbibliography
\end{document}